%% file: main.tex
\documentclass[english,11pt]{article}
\usepackage[T1]{fontenc}
\usepackage[latin9]{inputenc}
\usepackage{color}
\usepackage{babel}
\usepackage{verbatim}
\usepackage{gensymb}
\usepackage{textcomp}
\usepackage{algorithm2e}
\usepackage{amsmath}
\usepackage{amsthm}
\usepackage{amssymb}
\usepackage{graphicx}
\usepackage[commandnameprefix=ifneeded]{changes}
\usepackage[authoryear,sort]{natbib}
\usepackage[unicode=true,
 bookmarks=true,bookmarksnumbered=false,bookmarksopen=false,
 breaklinks=false,colorlinks=false,hidelinks]
 {hyperref}

\usepackage{tikz}
\usetikzlibrary{arrows,automata}
\usetikzlibrary{shapes,shapes.geometric}

\makeatletter

\newcommand{\lyxmathsym}[1]{\ifmmode\begingroup\def\b@ld{bold}
  \text{\ifx\math@version\b@ld\bfseries\fi#1}\endgroup\else#1\fi}

\providecommand{\tabularnewline}{\\}

\theoremstyle{plain}
\newtheorem{thm}{\protect\theoremname}
\theoremstyle{plain}
\newtheorem{prop}[thm]{\protect\propositionname}
\theoremstyle{definition}
\newtheorem{example}[thm]{\protect\examplename}
\theoremstyle{definition}
\newtheorem{defn}[thm]{\protect\definitionname}
\theoremstyle{plain}
\newtheorem{cor}[thm]{\protect\corollaryname}
\theoremstyle{plain}
\newtheorem{lem}[thm]{\protect\lemmaname}
\theoremstyle{plain}
\newtheorem*{lem*}{\protect\lemmaname}
\theoremstyle{plain}
\newtheorem*{prop*}{\protect\propositionname}

\input{symbols.tex}
\input{preamble.tex}

\makeatother

\providecommand{\corollaryname}{Corollary}
\providecommand{\definitionname}{Definition}
\providecommand{\examplename}{Example}
\providecommand{\lemmaname}{Lemma}
\providecommand{\propositionname}{Proposition}
\providecommand{\theoremname}{Theorem}

\begin{document}
\title{Robust Reward Design for Markov Decision Processes}
\author{
Shuo Wu\thanks{%
Department of Electrical and Computer Engineering, University of Illinois Chicago, Chicago, IL 60607. Email: \texttt{swu99@uic.edu}; \texttt{hanshuo@uic.edu}.}
\and
Haoxiang Ma\thanks{%
Department of Electrical and Computer Engineering, University of Florida, Gainesville, FL 32611. Email: \texttt{hma2@ufl.edu}; \texttt{fujie@ufl.edu}.}
\and
Jie Fu\footnotemark[2]
\and
Shuo Han\footnotemark[1]
}

\maketitle
\input{abstract.tex}

\section{Introduction and Background}

\input{intro.tex}

\section{Related Work}

\input{related_work.tex}

\section{Preliminaries\label{sec:Problem mdp}}

\input{problem_mdp.tex}

\section{Sensitivity Issues of Optimal Allocations\label{sec:sensitivity}}

\input{sensitivity.tex}

\section{Interior-Point Allocation Are Robust\label{sec:ipa_robust}}

\input{center.tex}

\section{Choosing an Optimal Allocation Region\label{sec:choosing_region}}

\input{regions.tex}

\section{Existence of Optimal Interior-Point Allocations\label{sec:Condition-on-Existence} }

\input{existence_of_interior.tex}

\section{Computing an Optimal Interior-Point Allocation\label{sec:Computation-of-Chbyshev}}

\input{milp.tex}

\section{Interior-Point Reward Functions Are Robust\label{sec:Interior-Point-Reward-Functions}}

\input{reward_regions.tex}

\section{Numerical Experiments}

\input{experiments.tex}

\section{Conclusions}

\input{conclusions.tex}

\section*{Acknowledgments}
We would like to thank Tamer Ba\c{s}ar and Haifeng Xu for helpful discussions.

\bibliographystyle{apalike}
\bibliography{Robust_Reward_Shaping}

\appendix
\input{appendix.tex}
\end{document}

%% file: symbols.tex
\usepackage{amsfonts,amssymb}
\usepackage{mathtools}

\newcommand\numberthis{\addtocounter{equation}{1}\tag{\theequation}}

\newcommand{\norm}[1]{\left\Vert#1\right\Vert}

\newcommand{\iprod}[2]{\langle#1,#2\rangle}

\DeclareMathOperator*{\argmax}{arg\,max}
\DeclareMathOperator*{\argmin}{arg\,min}

\newcommand{\reals}{\mathbb{R}}

\DeclareMathOperator*{\optmin}{\mathrm{minimize}}
\DeclareMathOperator*{\optmax}{\mathrm{maximize}}

\DeclareMathOperator{\optst}{\mathrm{subject\ to}}

\newcommand{\deq}{\triangleq}

\newcommand{\E}{\mathbb{E}}
\newcommand{\BR}{\mathsf{BR}}
\newcommand{\BRdet}{\mathsf{BR}_\mathrm{det}}

\newcommand{\mr}{\tilde{r}_2}

\input{symbols_cc.tex}

%% file: symbols_cc.tex
\newcommand{\cA}{\mathcal{A}}
\newcommand{\cM}{\mathcal{M}}
\newcommand{\cP}{\mathcal{P}}
\newcommand{\cS}{\mathcal{S}}
\newcommand{\cT}{\mathcal{T}}
\newcommand{\cX}{\mathcal{X}}

\newcommand{\Mdet}{\mathcal{M}_\mathrm{det}}
\newcommand{\mdet}{m_\mathrm{det}}
\newcommand{\Pidet}{\Pi_\mathrm{det}}
\newcommand{\pidet}{\pi_\mathrm{det}}
\newcommand{\piopt}{\pi^\mathrm{opt}}
\newcommand{\Qopt}{Q^\mathrm{opt}}
\newcommand{\tildep}{\tilde{\mathcal{P}}}

\newcommand{\ent}{\mathrm{Ent}}

\newcommand{\xopt}{x^{\star}}
\newcommand{\mopt}{m^{\star}}
\newcommand{\pistar}{\pi^{\star}}

\newcommand{\optival}{\mathsf{OptiVal}}
\newcommand{\pessval}{\mathsf{PessVal}}

\newcommand{\Sd}{\mathcal{S}_d}

%% file: preamble.tex
\usepackage[margin=1in]{geometry}
\usepackage{amsfonts,amssymb}
\usepackage{mathtools}
\usepackage[font=small,margin=.05\textwidth]{caption} 
\usepackage{xcolor}
\allowdisplaybreaks[3]
\definecolor{darkgreen}{rgb}{0,0.5,0}

%% file: abstract.tex
\begin{abstract}
The problem of reward design examines the interaction between a \emph{leader} and a \emph{follower}, where the leader aims to shape the follower's behavior to maximize the leader's payoff by modifying the follower's reward function. Current approaches to reward design rely on an accurate model of how the follower responds to reward modifications, which can be sensitive to modeling inaccuracies. To address this issue of sensitivity, we present a solution that offers robustness against uncertainties in modeling the follower, including 1) how the follower breaks ties in the presence of nonunique best responses, 2) inexact knowledge of how the follower perceives reward modifications, and 3) bounded rationality of the follower. Our robust solution is guaranteed to exist under mild conditions and can be obtained numerically by solving a mixed-integer linear program. Numerical experiments on multiple test cases demonstrate that our solution improves robustness compared to the standard approach without incurring significant additional computing costs. 
\end{abstract}

%% file: intro.tex
The problem of reward design is concerned with interactions between two types of players, a \emph{leader} and a \emph{follower}. The leader aims to induce the follower to behave in a desirable way by modifying the follower's reward function. The terms ``leader'' and ``follower'' derive from Stackelberg games~\citep{stackelbergMarktformUndGleichgewicht1934} to describe the asymmetric roles of the players: The leader always modifies the reward before the followers take action. As an example of reward design, imagine a class consisting of a teacher playing the role of the leader and a group of students playing the role of the follower. The teacher creates a grading policy that aims to motivate the students to achieve better learning outcomes. If the policy determines the final grade of a student solely based on the performance of the final exam, the student might neglect homework and class attendance, opting to cram for the final exam the night before. In contrast, a more effective grading policy is to include homework and attendance in the evaluation, which encourages students to learn on a regular basis. Another example arises in the training of autonomous agents via reinforcement learning, in which the agents are trained to maximize a predefined reward function. However, if the reward function is not carefully designed, it may be exploited by the agents to achieve high rewards in unintended ways, a phenomenon known as reward hacking~\citep{panEffectsRewardMisspecification2022}. This can lead to behaviors that do not align with human values. 

We focus on the case where the sequential decision-making problem of the follower is modeled by a Markov decision process (MDP). Specifically, the follower faces a sequential decision-making problem and aims to maximize the cumulative reward. Such a setting is motivated by problems in cybersecurity. In a network, an attacker, who acts as the follower in our problem setting, seeks to compromise a set of target hosts to acquire confidential data or to disrupt the network operation. To achieve the objective, the attacker must breach a sequence of hosts within the network by exploiting causal and logical dependencies between   vulnerabilities   and carrying out sequential attacks. The planning of the sequential attacks is captured using 
 deterministic or probabilistic attack graphs~\citep{jhaTwoFormalAnalyses2002a,hewett2008host,frigaultMeasuringNetworkSecurity2008a,nguyenMultistageAttackGraph2018a}, which can be viewed as a special instance of MDPs. A defender, who acts as the leader in this setting, aims to protect the network by allocating fake hosts or honey-patching vulnerabilities~\citep{qinHybridCyberDefense2023} and, consequently, mislead the attacker by manipulating the attacker's perceived rewards/payoffs.

The reward design problem can be formulated as a Stackelberg game~\citep{chenAdaptiveModelDesign2022,maOptimalResourceAllocation2023,ben-poratPrincipalAgentRewardShaping2024,chakrabortyPARLUnifiedFramework2024,zhangValueBasedPolicyTeaching2008,thakoorExploitingBoundedRationality2020}, for which algorithmic solutions are available. Nevertheless, existing solutions to reward design suffer from three issues of sensitivity to uncertainties in the model of the follower. First, the designed reward may be sensitive to tie-breaking of the follower. When the best response of the follower is not unique under the designed reward, the leader's payoff may depend on the follower's choice, as illustrated in Example~\ref{exa:opti_pess_values} and Example~\ref{exa:never_unique} in this paper. Second, when the leader does not know the exact reward function of the follower, the leader's payoff may drop significantly upon small errors in modeling the follower's reward function. This is because the follower's response does not vary continuously as his reward changes. Lastly, when the follower is irrational and does not play a best response, the leader's payoff may be sensitive to the level of suboptimality of the follower's response (Table~\ref{tab:Table-of-expected-payoffs} and Table~\ref{tab:Table-of-expected-payoffs 10*10}).  

Solutions to these three issues of sensitivity remain under-explored. For achieving robustness to nonunique best responses, one notable folklore solution is to slightly perturb the optimal strategy of the leader~\citep{zamirLeadershipCommitmentMixed2004}. However, it remains unclear how such a perturbation can be determined algorithmically. The method of perturbation may also fail in the presence of a budget constraint on reward allocation, as illustrated in Example~\ref{exa:opti_pess_values}. The closest work to ours is the study of robust Stackelberg equilibria in \citet{ganRobustStackelbergEquilibria2023}, which addresses the issue of sensitivity to irrational followers. In their setting, the follower may play any $\delta$-optimal response, i.e., a response that yields a follower's payoff within $\delta$ from optimality, whereas the leaders aims to find a robust strategy that maximizes the worst-case payoff against all possible $\delta$-optimal responses. They discussed the existence of such a robust strategy and showed how the robust strategy can be computed in theory. In comparison, our work uses the quantal response model~\citep{luceIndividualChoiceBehavior1959}, which is a probabilistic model for characterizing the irrationality of the follower. Interestingly, our proposed robust solution is found to relate to the robust solution studied in \citet{ganRobustStackelbergEquilibria2023}. Proposition~\ref{prop:robustness to delta optimal response} shows that our robust solution is nearly robust optimal against $\delta$-optimal responses of the follower.

\paragraph*{Main results}

The main results of this paper are summarized below:
\begin{itemize}
\item We formulate the reward design problem as a Stackelberg game and identify the issue of sensitivity to illustrate the need for obtaining a robust solution. When the follower's best response is nonunique, we give examples showing that the leader's payoff can be sensitive to how the follower breaks ties (Example~\ref{exa:opti_pess_values}). While a folklore solution in the case of normal-form Stackelberg games is to induce a unique best response of the follower by perturbing the leader's strategy, we show that a unique best response may not be inducible in reward design for MDPs, even for very common cases (Example~\ref{exa:never_unique}). 
\item We propose a robust solution, named an \emph{optimal interior-point allocation}, for the reward design problem. An optimal interior-point allocation is provably robust in three aspects. First, the solution is robust to the follower's tie-breaking choices (Proposition~\ref{prop:robustness_nonunique}). Second, the solution offers robustness when the leader does not know the follower's reward function precisely (Proposition~\ref{prop:robustness_reward}). Third, the solution is robust when the follower's decision-making deviates from rationality, a phenomenon often referred to as \emph{bounded rationality}; the robustness guarantee holds under three different models of bounded rationality, including two based on quantal response models (Propositions~\ref{prop:robustness of ent-regularized-1} and \ref{prop:Robustness to ent-regularized}), and one based on $\delta$-optimal responses (Proposition~\ref{prop:robustness to delta optimal response}). 
\item We prove that the existence of an optimal interior-point allocation only depends on the reward allocation budget of the leader (Theorem~\ref{Thm:Sufficiency}). Specifically, an optimal interior-point allocation is guaranteed to exist when the leader can achieve the optimal payoff without exhausting the reward allocation budget in the optimistic case, i.e., when the follower breaks ties in favor of the leader. 
\item Moreover, we prove that the existence of an optimal interior-point allocation is not only sufficient but also necessary for achieving robustness to tie-breaking of the follower (Theorem~\ref{prop:best responses to existence}). This establishes the pivotal role of interior-point allocations in achieving robustness. It further motivates the development of algorithms for finding an optimal interior-point allocation.

\item We show that an optimal interior-point allocation can be computed from a mixed-integer linear program (MILP). Numerical experiments were conducted to validate the robustness of the solution and to evaluate the computational cost on problems of practical interest (e.g., defending against cyberattacks). Compared to the standard solution to reward design given by~\citet{maOptimalResourceAllocation2023}, our solution was found to improve robustness in all the test cases while maintaining a reasonable computational cost. 
\end{itemize}

%% file: related_work.tex
\paragraph{Reward design}

The problem of choosing a suitable reward function to achieve a desirable outcome is also known as \emph{reward shaping}. Early work on reward shaping in MDPs studies policy invariance, i.e., how the reward function may be modified without changing the set of optimal policies. \citet{ngPolicyInvarianceReward1999} introduced the notion of potential-based reward shaping and showed guaranteed policy invariance in MDPs. Potential-based reward shaping was later extended to multi-agent systems for preserving the set of Nash equilibria~\citep{devlinTheoreticalConsiderationsPotentialBased2011}. More recently, reward shaping was studied in a leader-follower setup~\citep{ben-poratPrincipalAgentRewardShaping2024} to incentivize the follower to act in the leader's interest. Aside from reward shaping, the problem of reward design has also been studied in the context of \emph{policy teaching}~\citep{zhangValueBasedPolicyTeaching2008,zhangPolicyTeachingReward2009,banihashemAdmissiblePolicyTeaching2022} and can be viewed as a special case of \emph{model design}~\citep{chenAdaptiveModelDesign2022} and \emph{environment design}~\citep{zhangGeneralApproachEnvironment2009,yuEnvironmentDesignBiased2022}, in which elements (e.g., the transition kernel) other than the reward function of the MDP can also be modified.

For solving the reward design problem, \citet{zhangValueBasedPolicyTeaching2008} provided a solution based on mixed-integer programming. \citet{ben-poratPrincipalAgentRewardShaping2024} established the NP-hardness of the reward design problem and gave a polynomial-time approximation algorithm. A special case is when the leader aims to induce a specific policy rather than a policy that maximizes the leader's payoff. For this case, \citet{zhangPolicyTeachingReward2009} showed that the desired reward function can be obtained by finding a feasible solution to a linear program. 

\paragraph{Principal-agent problem}

In the reward design problem, the leader and the follower are sometimes called the \emph{principal} and the \emph{agent}, respectively, which are terms originated from economics~\citep{rossEconomicTheoryAgency1973,grossmanAnalysisPrincipalAgentProblem1983}. In the principal-agent problem, an agent takes actions on behalf of the principal through a contract designed by the principal~\citep{boltonContractTheory2005,salanieEconomicsContractsPrimer2005,hartTheoryContracts1987,ganGeneralizedPrincipalAgencyContracts2024}. Two key aspects of the principal-agent problem are\textit{ adverse selection} and\textit{ moral hazard~}\citep{myersonOptimalCoordinationMechanisms1982a}. The former refers to a situation in which the agents have private information that the principal cannot readily access; the latter refers to the situation in which the agents take private actions that the principal cannot directly control or observe. The problem of reward design resembles the principal-agent problem, except that it typically does not involve adverse selection or moral hazard. 

The principal-agent problem has also been studied in control theory, where it is commonly known as \emph{incentive design}~\citep{hoControltheoreticViewIncentives1982}; see recent survey papers such as \citet{ratliff2019perspective} and \citet{basarInducementDesiredBehavior2024}. In incentive design, a decision-maker needs to determine a reward-based incentive strategy to encourage a desired behavior of another agent. The incentive strategy plays a similar role as the contract in the principal-agent problem. The form of optimal incentive strategies was studied in \citet{basarGeneralTheoryStackelberg1982,zhengExistenceDerivationOptimal1982} for deterministic decision problems and in \citet{basarAffineIncentiveSchemes1984} for stochastic decision problems. Similar to moral hazard in the principal-agent problem, some settings of incentive design assume that the agent may influence the state of the environment through actions and that the decision-maker only has access to partial information of the state. 

\paragraph{Stackelberg games}

The reward design problem can be viewed as a Stackelberg game or, mathematically, a bilevel optimization problem~\citep{chenAdaptiveModelDesign2022,maOptimalResourceAllocation2023,ben-poratPrincipalAgentRewardShaping2024,chakrabortyPARLUnifiedFramework2024}. Stackelberg games were first introduced by \citet{stackelbergMarktformUndGleichgewicht1934} to demonstrate the benefits of leadership in games involving two parties. \citet{zamirLeadershipCommitmentMixed2004} extended the strategy space in the model from pure strategies to mixed strategies. When the number of pure strategies is finite, \citet{conitzerComputingOptimalStrategy2006} showed that an optimal strategy of the leader can be computed in polynomial time by solving multiple linear programs. Our results are inspired by the alternative solution presented by Paruchuri et al. (2008), who showed that an optimal strategy for the leader can be computed from a single MILP. 

\paragraph{Robust solution to Stackelberg games}

In a Stackelberg game, the influence of nonunique best responses of the follower on the leader's payoff is captured by the notions of strong and weak Stackelberg equilibria~\citep{bretonSequentialStackelbergEquilibria1988}. In a strong Stackelberg equilibrium, the leader plays optimally assuming that the follower breaks ties in favor of the leader. In contrast, in a weak Stackelberg equilibrium, the leader assumes that the follower breaks ties adversarially. Consequently, the leader will play a strategy that is robust to tie-breaking. Nevertheless, a weak Stackelberg equilibrium may not exist in general \citep[Section~4.3.2]{dempeBilevelOptimizationAdvances2020}. 

Robustness to the unknown follower's reward has been studied by \citet{canseverMinimumSensitivityApproach1983,canseverOptimumNearoptimumIncentive1985}. The results are based on analyzing the local sensitivity of the leader's payoff, which is characterized by derivatives of the payoff function with respect to unknown parameters in the follower's reward function. In comparison, our work uses the notion of \emph{robustness margin} and can guarantee non-local robustness against bounded modeling errors of the follower's reward function.

Robustness to irrational followers has been studied by \citet{ganRobustStackelbergEquilibria2023}, where the goal is to find an optimal leader's strategy against a follower who plays any near-optimal response deterministically. Their model of irrationality differs from the quantal response model used in our work, which assumes that the follower chooses suboptimal responses probabilistically.

\paragraph{Models of irrationality}

Several models have been used in the literature to characterize irrational behaviors of a decision-making agent. Of particular interest is modeling irrational behaviors due to limited reasoning capability of the agent, commonly known as \emph{bounded rationality}~\citep{simonBehavioralModelRational1955}. Our work adopts the \emph{quantal response} model, which assumes that the agent may play suboptimal strategies with a probability that decreases exponentially as the corresponding payoff decreases. The quantal response model in economics intends to model how the selection probabilities of an individual depend on explanatory variables~\citep{mcfaddenQuantalChoiceAnalysis1976}. In game-theoretic settings, quantal response was used to model the behaviors of players in normal-form games~\citep{mckelveyQuantalResponseEquilibria1995} and in Stackelberg games~\citep{yangComputingOptimalStrategy2012}. The resulting equilibrium is known as the \emph{quantal response equilibrium}. Besides quantal response, another popular model of bounded rationality is the \textit{cognitive hierarchy} model~\citep{camererCognitiveHierarchyModel2004}. The model uses iterative decision rules to represent strategic players with different levels of sophistication and can be used to explain nonequilibrium behaviors of the players. Lastly, one may assume that the agent plays any suboptimal strategy that leads to a payoff close to optimality~\citep{ganRobustStackelbergEquilibria2023}.

%% file: problem_mdp.tex
\subsection{Reward design for MDPs\label{subsec:prelim-setup}}

The decision-making process of the follower is modeled as an MDP $M=(\mathcal{S},\mathcal{A},\mathcal{T},r,\gamma,\rho)$. The set $\mathcal{S}$ is the state space, $\mathcal{A}$ is the action space, and $\rho\in\Delta(\mathcal{S})$ is the initial distribution, where $\Delta(\mathcal{S})$ denotes the set of probability distributions over $\cS$. The mappings $\mathcal{T}:\mathcal{S}\times\mathcal{A}\times\mathcal{S}\rightarrow[0,1]$ and $r:\mathcal{S}\times\mathcal{A}\rightarrow\mathbb{R}$ represent the transition kernel and the reward function, respectively. We sometimes also abuse the notation and view $r$ as a vector of dimension $\vert\mathcal{S}\vert\vert\mathcal{A}\vert$. The constant $\gamma\in(0,1]$ is the discount factor. The policy of the follower is denoted by $\pi:\mathcal{S}\times\mathcal{A}\rightarrow[0,1]$, where $\pi(s,\cdot)\in\Delta(\mathcal{A})$ for all $s\in\mathcal{S}$. The set of all policies is denoted by $\Pi$.

The leader may modify the reward function of the follower by allocating rewards in the environment. Let $\mathcal{\mathcal{S}}_{d}=\{s_{1},\dots,s_{K}\}\subseteq\mathcal{S}$ be the set of candidate states at which reward can be allocated. The reward allocation is described by a function $x:\mathcal{S}_{d}\rightarrow\text{\ensuremath{\mathbb{R}}}$, where $x(s_{i})$ is the amount of reward allocated at $s_{i}$. For convenience, we sometimes abuse the notation and treat $x$ as a vector of dimension $|\Sd|$ such that $x_{i}=x(s_{i})$. A reward allocation $x$ is said to be \emph{admissible} if $x$ satisfies the following conditions:
\begin{itemize}
\item The total amount of allocated resource cannot exceed a given budget $C$: $\sum_{i=1}^{|\mathcal{S}_{d}|}x_{i}\leq C$.
\item The amount of resource allocated to each state is nonnegative: $x_{i}\geq0,$ $i=1,\dots,|\Sd|$.
\end{itemize}
The set of admissible reward allocations is denoted by $\mathcal{X}=\left\{ x\mid\sum_{i=1}^{|\mathcal{S}_{d}|}x_{i}\leq C,\ x\succeq0\right\} $. An admissible reward allocation  $x$ induces a new reward function $r_{2}^{x}$ of the follower, where $r_{2}^{x}(s,a)\deq r(s,a) + x(s)$. Throughout most of the paper, the leader's reward function $r_{1} \colon \mathcal{S} \times \cA \to [0,1]$ is given by $r_{1}(s,a)=1$ when $s\in\Sd$ and $r_{1}(s,a)=0$ otherwise. The form of $r_1$ is motivated by problems in cybersecurity, where a defender (i.e., the leader) is rewarded by misleading an attacker (i.e., the follower) into reaching a set $\Sd$ of desired states. The case of general $r_1$ will be discussed in Section~\ref{sec:Interior-Point-Reward-Functions}. Under a given policy $\pi$ of the follower, the leader's value function is defined by $V_{1}^{\pi}(s)=\mathbb{E}_{\pi}\left[\sum_{t=0}^{\infty}\gamma^{t}r_{1}(s_{t},a_{t})\mid s_{0}=s\right]$, and the follower's value function is defined by $V_{2}^{\pi}(s;x)=\mathbb{E}_{\pi}\left[\sum_{t=0}^{\infty}\gamma^{t}r_{2}^{x}(s_{t},a_{t})\mid s_{0}=s\right]$. The expected payoffs of the leader and the follower are given by $\mathbb{E}_{s\sim\rho}\left[V_{1}^{\pi}(s)\right]$ and $\mathbb{E}_{s\sim\rho}\left[V_{2}^{\pi}(s;x)\right]$, respectively. The goal of the leader is to maximize her payoff by influencing the policy of the follower through the reward allocation $x$. Mathematically, this can be cast as the following optimization problem:
\begin{alignat}{2}
 & \optmax_{x,\pi} & \quad & \mathbb{E}_{s\sim\rho}\left[V_{1}^{\pi}(s)\right]\nonumber \\
 & \optst &  & \pi\in\argmax_{\pi}\mathbb{E}_{s\sim\rho}\left[V_{2}^{\pi}(s;x)\right]\label{eq:original problem}\\
 &  &  & x\in\mathcal{X}.\nonumber 
\end{alignat}
We do not assume that $\argmax_{\pi}\mathbb{E}_{s\sim\rho}\left[V_{2}^{\pi}(s;x)\right]$ is a singleton set. When the optimal policy of the follower is not unique under a reward allocation  $x$, the problem formulation in~(\ref{eq:original problem}) assumes that the follower will choose from the set of optimal policies in favor of the leader. This assumption will be revisited in Section~\ref{sec:sensitivity}.

\subsection{Reformulation via the occupancy measure}

While the expected payoff in an MDP is a nonlinear function of the policy $\pi$, it can be rewritten as a linear function by a change of variables. As shown in Section~\ref{subsec:oipa}, writing the expected payoff as a linear function is useful in revealing important geometric structures of~(\ref{eq:original problem}). The change of variables involves replacing the policy $\pi$ with $m\colon\cS\times\cA\to\reals$ defined by 
\begin{equation}
m(s,a)\deq\mathbb{E}_{\pi,s_{0}\sim\rho}\left[\sum_{t=0}^{\infty}\gamma^{t}\mathbb{P}(s_{t}=s,a_{t}=a)\right],\label{eq: pi to measure}
\end{equation}
known as the \emph{occupancy measure} induced by $\pi$. It can be shown that $m$ is a valid occupancy measure if and only if $m(s,a)\geq0$ for all $(s,a)\in\mathcal{S}\times\mathcal{A}$ and 
\begin{equation}
\sum_{a\in\mathcal{A}}m(s,a)=\rho(s)+\gamma\sum_{s',a'}\mathcal{T}(s',a',s)m(s',a')\quad\forall s\in\mathcal{S}.\label{eq:flow_constraint}
\end{equation}
For convenience, we abuse the notation and treat $m$ as a vector of dimension $|\cS||\cA|$, $\rho$ as a vector of dimension $|\cS|$, and succinctly write~(\ref{eq:flow_constraint}) as $Am=\rho$ for some matrix $A$. Denote the set of valid occupancy measures by
\begin{equation*}
\mathcal{M}=\left\{ m\mid m\succeq0,\ Am = \rho \right\}.
\end{equation*}

For the direction from policies to occupancy measures, however, there could be many policies that induce the same occupancy measure. 
\begin{prop}
\label{prop:policies induce occupancy measure}Given any $m\in\cM$, a policy $\pi$ induces $m$ if and only if 
\begin{equation}
\pi(s,a)=\begin{cases}
m(s,a)/\sum_{a'\in\mathcal{A}}m(s,a') & \text{if }\sum_{a'\in\mathcal{A}}m(s,a')\neq0,\\
\pi_{0}(s,a) & \text{otherwise}
\end{cases}\label{eq: measure to pi}
\end{equation}
for some policy $\pi_{0}$. 
\end{prop}

\begin{proof}
See Appendix~\ref{subsec:Proof of polices induce occupancy measure}.
\end{proof}

Define $\langle r_{2}^{x},m\rangle\deq\sum_{(s,a)\in\mathcal{S}\times\mathcal{A}}r_{2}^{x}(s,a)m(s,a)$. It can be shown (see Appendix~\ref{subsec:Proof-of expected reward }) that the expected payoff of the follower becomes a linear function of $m$: 
\begin{equation*}
\mathbb{E}_{s\sim\rho}\left[V_{2}^{\pi}(s;x)\right]=\langle r_{2}^{x},m\rangle.
\end{equation*}
Similarly, the expected payoff of the leader satisfies $\mathbb{E}_{s\sim\rho}\left[V_{1}^{\pi}(s)\right]=\langle r_{1},m\rangle$. For any reward allocation $x$, define the set of best responses of the follower to $x$ by 
\[
\BR(x)\deq\argmax_{m\in\mathcal{M}}\langle r_{2}^{x},m\rangle.
\]
Using the occupancy measure, the reward design problem in~(\ref{eq:original problem}) can be equivalently reformulated as 
\begin{equation}
\begin{alignedat}{2} & \optmax_{x\in\mathcal{X},m\in\mathcal{M}} & \quad & \langle r_{1},m\rangle\\
 & \optst &  & m\in\BR(x).
\end{alignedat}
\label{eq:allocation_problem_om}
\end{equation}
We will hereafter denote the optimal value of~(\ref{eq:allocation_problem_om}) by $v_{1}^{\star}$ and an optimal solution of~(\ref{eq:allocation_problem_om}) by $(x^{\star},m^{\star})$, where $x^{\star}$ is called an \emph{optimal allocation} and $m^{\star}$ an \emph{optimal occupancy measure}. 

\subsection{\label{subsec:Computing-an-optimal}Computing an optimal reward allocation}

The reward design problem in (\ref{eq:allocation_problem_om}) can be solved by an MILP \citep{maOptimalResourceAllocation2023}. In the following, we will briefly review the procedure for completeness. For a given reward allocation $x$, the condition $m\in\BR(x)$ is equivalent to that $m$ is an optimal solution of the following problem:
\begin{equation}
\begin{alignedat}{2} & \optmax_{m} & \quad & \iprod{r_{2}^{x}}{m}\\
 & \optst &  & Am=\rho,\qquad m\succeq0,
\end{alignedat}
\label{eq:prob_follower}
\end{equation}
where the constraints come from the condition $m\in\mathcal{M}$. Because problem~(\ref{eq:prob_follower}) is a linear program and is always feasible, it is known that $m$ is optimal if and only if there exists (a dual variable) $\nu$ such that $(m,\nu)$ satisfies the Karush--Kuhn--Tucker (KKT) conditions:
\begin{equation}
Am=\rho,\quad m\succeq0,\quad A^{T}\nu-r_{2}^{x}\succeq0,\quad m\perp A^{T}\nu-r_{2}^{x}.\label{eq:KKT}
\end{equation}
The reward design problem in~(\ref{eq:allocation_problem_om}) can then be rewritten as\begin{subequations}\label{MILP} 
\begin{alignat}{2}
 & \optmax_{x,m,\nu} & \quad & \langle r_{1},m\rangle\label{eq:MILP_a}\\
 & \optst &  & Am=\rho,\qquad m\succeq0,\qquad A^{T}\nu-r_{2}^{x}\succeq0\label{eq:MILP_b}\\
 &  &  & m\perp A^{T}\nu-r_{2}^{x}\label{eq:MILP_c}\\
 &  &  & x\in\mathcal{X}.\label{eq:MILP_d}
\end{alignat}
\end{subequations}The complementary slackness constraint in~(\ref{eq:MILP_c}) can be reformulated as affine constraints with integer variables~\citep{vielmaMixedIntegerLinear2015}. Because other constraints are affine, and the objective is linear, problem~(\ref{MILP}) can be reformulated as an MILP. 

\subsection{Notation\label{subsec:Notation}}

A policy $\pi$ is called \emph{deterministic} if for any $s\in\cS$, there exists $a\in\cA$ such that $\pi(s,a)=1$. A policy is called \emph{randomized} if it is not deterministic. The set of all deterministic policies is denoted by $\Pidet$. For any $\pi\in\Pi$, we sometimes denote by $m^{\pi}$ the occupancy measure \emph{induced} by $\pi$ according to (\ref{eq: pi to measure}). An occupancy measure is called \emph{deterministic} (resp.\ \emph{randomized}) if it is induced by a deterministic (resp.\ randomized) policy. The set of deterministic occupancy measures is denoted by $\Mdet$. For any $m \in \cM$, denote the set of policies of the form (\ref{eq: measure to pi}) by $\Pi(m)$, which is the set of policies that induce $m$ according to Proposition~\ref{prop:policies induce occupancy measure}.

For a given (possibly randomized) policy $\pi$, we abuse the notation of $\Pidet$ and define
\begin{equation*}
\Pidet(\pi) \deq\{\pi'\in\Pidet\mid\pi'(s,a)=0\text{ when }\pi(s,a)=0\}.
\end{equation*}
In plain words, at any $s\in\cS$, a policy $\pi'\in\Pidet(\pi)$ is only allowed take an action that $\pi$ takes with nonzero probability. The set of occupancy measures induced by policies in $\Pidet(\pi)$ is denoted by $\Mdet(\pi)\deq\{m^{\pi'}\mid\pi'\in\Pidet(\pi)\}$ with an abuse of notation. Such a definition has a nice property that for any $\pi_{1},\pi_{2}\in\Pi(m)$, it holds that $\Mdet(\pi_{1})=\Mdet(\pi_{2})$. Interested reader may refer to Appendix~\ref{subsec:proof of induce same set of occupancy measure}.

%% file: sensitivity.tex
In problem~(\ref{eq:allocation_problem_om}), the maximization over $m$ implicitly assumes that the follower chooses a best response that maximizes the expected payoff of the leader. Under this assumption, the expected payoff received by the leader under an allocation $x$ is given by the optimal value of the problem
\begin{equation}
\begin{alignedat}{2} & \optmax_{m\in\mathcal{M}} & \quad & \iprod{r_{1}}{m}\\
 & \optst &  & m\in\BR(x).
\end{alignedat}
\label{eq:optimistic problem}
\end{equation}
In practice, however, the follower may choose arbitrarily from the set of best responses. In the worst case for the leader, the follower may choose a best response that is most unfavorable to the leader. In this case, the leader's  expected payoff under the allocation $x$ is given by the optimal value of the problem 
\begin{equation}
\begin{alignedat}{2} & \optmin_{m\in\mathcal{M}} & \quad & \iprod{r_{1}}{m}\\
 & \optst &  & m\in\BR(x).
\end{alignedat}
\label{eq:pessimistic problem}
\end{equation}
We define the optimal value of (\ref{eq:optimistic problem}) as the \emph{optimistic value} of $x$, denoted by $\optival(x)$, and the optimal value of (\ref{eq:pessimistic problem}) as the \emph{pessimistic value} of $x$, denoted by $\pessval(x)$. We also refer to $\sup_{x\in\mathcal{X}}\optival(x)$ as the \textit{optimal optimistic value} and $\sup_{x\in\mathcal{X}}\pessval(x)$ as the \textit{optimal pessimistic value}. In general, when the best response of the follower under an allocation $x$ is not unique, the optimistic value $\optival(x)$ and the pessimistic value $\pessval(x)$ may differ. In addition, the optimal optimistic value and the optimal pessimistic value may also differ, as illustrated in Example~\ref{exa:opti_pess_values}.
\begin{example}[Optimistic and pessimistic values]
\label{exa:opti_pess_values}Consider a $4\times4$ grid world in Figure~\ref{fig: one decoy}. The follower starts from the position $(1,2)$ and can move in any direction or stay in the same place by taking an action in $\cA=\{\textsf{left},\textsf{right},\textsf{up},\textsf{down},\textsf{stay}\}$. The true goals are at $(4,1)$ and $(4,3)$. The leader is allowed to allocate the reward at $(3,3)$. The game's discount factor is $1$. The payoff for arriving at $(4,1)$ is $3$, and for arriving at $(4,3)$ is $2$. The allocation budget is $1$, and $(4,1)$ and $(4,3)$ are the absorbing states, meaning that the follower cannot leave the state once enter them. The leader will get $1$ unit of payoff once the follower arrives at the states with allocated reward.

First, consider the optimistic case where the follower breaks any tie in favor of the leader. Denote by $x_{1}$ the allocation strategy that spends the entire budget of $1$ at $(3,3)$. Under $x_{1}$, choosing Path 1 to arrive at $(3,3)$ and $(4,3)$ in $4$ steps gives the follower a total payoff of $3$. Path 2 has the same payoff for the follower because the true goal at $(4,1)$ is in the path. Since the game ends in $4$ steps, $3$ is the maximum payoff that the follower can achieve, implying that two paths are the only best responses. The leader, however, will obtain different payoffs when the follower chooses different paths. If the follower reaches the allocated reward at $(3,3)$ with probability $1$ by following Path 1, the leader will receive a payoff of $1$. In comparison, if the follower reaches the true goal at $(4,1)$ by following Path 2, the leader will receive no payoff. Under the optimistic assumption, the follower will choose Path 1. Therefore, the leader's payoff for allocating all the budget at $(3,3)$ is $1$. For any other allocation strategy, the allocation at $(3,3)$ is less than $1$, and the follower's payoff for following Path 1 is less than the payoff for following Path 2. The follower will then follow Path 2, in which case the payoff for the leader is $0$. Therefore, $x_{1}$ is an optimal allocation strategy, and the optimal optimistic value of the game is $1$.

Next, consider the pessimistic case. It can be seen that the follower will always follow Path 2 under any admissible allocation. When the allocation at $(3,3)$ is less than $1$, the payoff for the leader remains the same as in the optimistic case because the best response of the follower is unique and is following Path 2. When the allocation at $(3,3)$ is $1$, the follower will still take Path 2 under the pessimistic assumption. This is because both paths are best responses of the follower, but Path 2 leads to a lower payoff for the leader. Therefore, the leader will always receive a payoff of $0$ in the pessimistic case regardless of the allocation strategy, which implies that the optimal pessimistic value of the game is $0$.

In summary, this example shows that the optimistic value and pessimistic value under $x_{1}$ are different: $\optival(x_{1})=1$ and $\pessval(x_{1})=0$. In addition, the optimal optimistic value of the game is $1$, whereas the optimal pessimistic value of the game is $0$.

\begin{figure}
\begin{centering}
\includegraphics[width=0.5\textwidth]{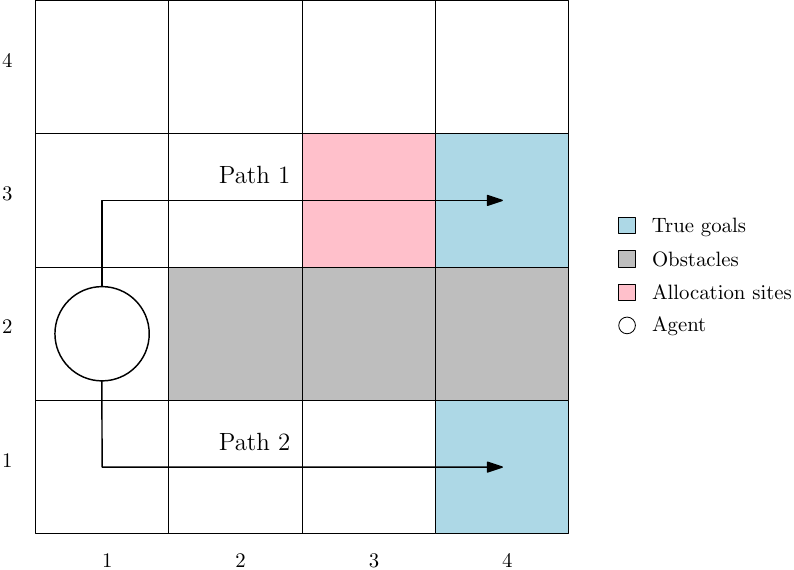}
\par\end{centering}
\caption{Reward design problem with one state admits allocated reward. The payoff for arriving at $(4,1)$ is $3$ and that for arriving at $(4,3)$ is $2$. At most $1$ unit of resource can be allocated toward the reward at $(3,3)$. The follower starts at $(1,2)$, and all the transitions are deterministic \label{fig: one decoy}.}
\end{figure}
\end{example}

Besides the handcrafted setup in Example~\ref{exa:opti_pess_values}, our numerical experiments in Section~\ref{subsec:Robustness-to-nonunique experiment} suggest that the optimal allocation obtained using the method in Section~\ref{subsec:Computing-an-optimal} also exhibits sensitivity to nonunique best responses. This motivates us to search for an optimal allocation $\xopt$ that is robust to nonunique best responses of the follower. Mathematically, this requires finding $\xopt$ that satisfies 
\begin{equation*}
\optival(\xopt)=\pessval(\xopt)=v_{1}^{\star}.
\end{equation*}

The issue of sensitivity to tie-breaking in Stackelberg games is a known issue in the literature, where a folklore solution (see page 12 in \citet{zamirLeadershipCommitmentMixed2004}) is to perturb the optimal allocation slightly such that the best response becomes unique. Nevertheless, two issues prevent the immediate application of this solution. First, it remains unclear how such a perturbation should be obtained algorithmically. Indeed, when the perturbation is not chosen appropriately, the induced best response may be to the disadvantage of the leader and yield the pessimistic value. Second, the best response may remain nonunique regardless of how the optimal allocation is perturbed, as illustrated in Example~\ref{exa:never_unique}.
\begin{example}[Best response is never unique]
\label{exa:never_unique}

Consider the environment in Figure~\ref{fig:two obstacles}, which has a similar setting as Example~\ref{exa:opti_pess_values} with the exception that one obstacle has been removed. Both Path 1 and Path 3 go through the allocated reward at $(3,3)$ and the true goal at $(4,3)$ but never reach the true goal at $(4,1)$. Therefore, for the follower, the payoff for choosing Path 1 and Path 3 are always the same under any reward allocation. This implies that Path 1 can never be a unique best response. For a similar reason, for the follower, the payoff for choosing Path 2 and Path 4 are always the same under any reward allocation. This implies that Path 2 can never be a unique best response. 

In the meantime, it is not difficult to see that either Path 1 or Path 2 must be a best response under any allocation. Therefore, the existence of Paths 3 and 4 shows that it is impossible to induce a unique best response regardless of the allocation strategy.

\begin{figure}
\begin{centering}
\includegraphics[viewport=0bp 0bp 362bp 273bp,width=0.5\textwidth]{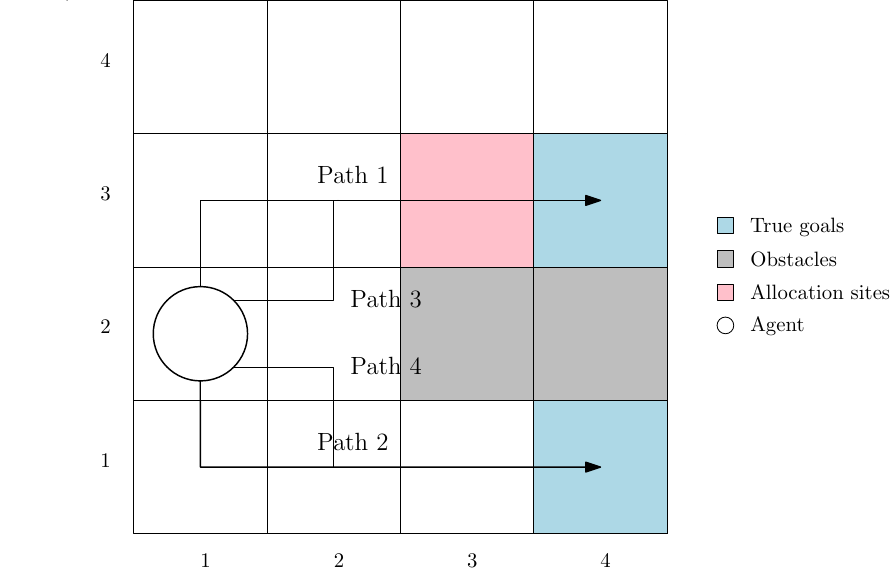}
\par\end{centering}
\caption{Reward design problem with two obstacles. Other settings are similar to Example~\ref{exa:opti_pess_values}. \label{fig:two obstacles}}

\end{figure}
\end{example}

%% file: center.tex
We introduce the concept of an \emph{optimal interior-point allocation}, which is a special type of optimal allocations, and show that any optimal interior-point allocation is robust to nonunique best responses of the follower. Namely, by using an optimal interior-point allocation, the leader is always guaranteed to receive the optimal expected payoff of the game regardless of how the follower breaks ties. Moreover, an optimal interior-point allocation is able to provide other forms of robustness guarantees. One is robustness to uncertainty in the follower's reward function: It ensures that the leader still receives the optimal payoff without exactly knowing how the follower perceives the modified reward function. Another form is robustness to bounded rationality in the follower: It ensures that the leader does not lose much payoff when the follower starts to behave irrationally by responding with a slightly suboptimal policy. 

\subsection{Optimal interior-point allocations\label{subsec:oipa}}

Since any allocation $x\in\reals^{|\Sd|}$ corresponds to one or more best responses described by $\BR(x)$, it is possible to divide $\reals^{|\Sd|}$ into (possibly overlapping) regions based on the corresponding best response. Each region is called an \emph{allocation region} and is identified by a unique occupancy measure.
\begin{defn}[Allocation region]
\label{def:allocation_region}Let $m\in\cM$ be an occupancy measure. The \emph{allocation region} of $m$ is defined by
\begin{equation*}
\mathcal{P}_{m}=\{x\in\mathbb{R}^{|\Sd|}\mid m\in\BR(x)\}=\{x\in\mathbb{R}^{|\Sd|}\mid\iprod{r_{2}^{x}}{m}\geq\iprod{r_{2}^{x}}{m'}\text{ for all }m'\in\mathcal{M}\}.
\end{equation*}
Furthermore, the allocation region of an optimal occupancy measure is called an \emph{optimal allocation region}. 
\end{defn}

The term allocation region highlights the fact that $\cP_{m}$ is a subset of $\reals^{|\Sd|}$, the space where allocation vectors live. For $x^{\star}$ to be an optimal allocation, it is necessary from Definition~\ref{def:allocation_region} that $\xopt$ belongs to some optimal allocation region. However, when $\xopt$ is on the boundary of an optimal allocation region, the best responses under $\xopt$ are not unique and may lead to undesirable sensitivity to tie-breaking. This motivates us to examine the interior points of an allocation region, which we refer to as \emph{interior-point allocations}.
\begin{defn}[Interior-point allocation]
\label{def:interior-point}Let $m\in\cM$ be an occupancy measure. An allocation vector $x\in\cX$ is called an\emph{ interior-point allocation of $\cP_{m}$} if there exists $c>0$ such that 
\begin{equation}
x+cv\in\cP_{m}\quad\text{for all }\norm{v}_{1}\leq1.\label{eq:interior_pt_condition}
\end{equation}
The largest constant $c$ for (\ref{eq:interior_pt_condition}) to hold is called the \emph{margin} of $x$. An interior-point allocation of an optimal allocation region is called an \emph{optimal interior-point allocation}.

Because all norms are equivalent in a finite-dimensional space, the choice of the norm in~(\ref{eq:interior_pt_condition}) does not affect the definition of interior-point allocation: An interior-point allocation in one norm must also be an interior-point allocation in another norm. The choice of $\norm{\cdot}_{1}$ is for the convenience of computation. For computing an optimal interior-point allocation, the condition in (\ref{eq:interior_pt_condition}) can be rewritten equivalently as finitely many constraints given by~(\ref{eq:1-norm transformation}) in Section~\ref{sec:Computation-of-Chbyshev}.
\end{defn}

\subsection{Robustness to nonunique best responses\label{subsec:robustness_nonunique}}

We will show that an optimal interior-point allocation $\xopt$ always yields the optimal value $v_{1}^{\star}$ of the game regardless of how the follower breaks ties. In other words, the pessimistic value of $\xopt$ is equal to the optimistic value of $\xopt$. 
\begin{prop}[Robustness to nonunique best responses]
\label{prop:robustness_nonunique}If $\xopt$ is an optimal interior-point allocation, then $\optival(\xopt)=\pessval(\xopt)=v_{1}^{\star}$.
\end{prop}

\begin{proof}
See Appendix~\ref{subsec: proof of Robustness 1}. 
\end{proof}
The importance of an interior-point allocation can be understood as follows. Suppose that $\xopt$ is an optimal interior-point allocation of $\cP_{m_{1}}$ for some optimal occupancy measure $m_{1}$. This implies $m_{1}\in\BR(\xopt)$. When $\optival(\xopt)\neq\pessval(\xopt)$, there must exist $m_{2}\in\BR(x^{\star})$ with $\iprod{r_{1}}{m_{1}}>\iprod{r_{1}}{m_{2}}$. Let $x(\epsilon)=x^{\star}-\epsilon\mathbf{1}$. It follows that $r_{2}^{x(\epsilon)}=r_{2}^{x^{\star}}-\epsilon r_{1}$. Because $m_{1},m_{2}\in\BR(x^{\star})$, it holds that $\iprod{r_{2}^{x^{\star}}}{m_{1}}=\iprod{r_{2}^{x^{\star}}}{m_{2}}$. Thus, for any $\epsilon>0$, it holds that $\iprod{r_{2}^{x(\epsilon)}}{m_{1}}<\iprod{r_{2}^{x(\epsilon)}}{m_{2}}$ or, equivalently, $x(\epsilon)\notin\cP_{m_{1}}$. This implies that no neighborhood of $\xopt$ is contained in $\cP_{m_{1}}$, which contradicts with the fact that $\xopt$ is an interior-point allocation of $\cP_{m_{1}}$.

While Proposition~\ref{prop:robustness_nonunique} shows that finding an optimal interior-point allocation is sufficient for achieving robustness to nonunique best responses of the follower, the existence of an optimal interior-point allocation is, in fact, also \emph{necessary} for robustness to nonunique best responses. We shall postpone the discussion on necessity until Theorem~\ref{prop:best responses to existence}. 

\subsection{Robustness to uncertain reward perception of the follower}

Besides robustness to nonunique best responses, one may be interested in other notions of robustness motivated by practical considerations. One such notion is robustness to uncertainty in how the follower perceives the modified reward function. The original formulation in (\ref{eq:allocation_problem_om}) assumes that the modified reward function of the follower is exactly $r_{2}^{x}$ when the allocation is $x$. However, the follower may perceive the modified reward differently as $r_{2}^{x+\delta}$, where $\delta\in\reals^{|\Sd|}$ represents uncertainty in how the follower perceives reward modifications. In light of such uncertainty, one reasonable goal is to seek an allocation $x^{\star}$ that is robust to any uncertainty $\delta$ up to a certain magnitude. 

The following proposition shows that an optimal interior-point allocation offers robustness to uncertainty in reward perception of the follower. 
\begin{prop}[Robustness to reward perception of the follower]
\label{prop:robustness_reward}If $\xopt$ is an optimal interior-point allocation with margin $c>0$, then $\optival(\xopt+\delta)=\pessval(\xopt+\delta)=v_{1}^{\star}$ for all $\norm{\delta}_{1}<c$.
\end{prop}

\begin{proof}
Since $\xopt$ is an optimal interior-point allocation, there exists $m$ such that $x^{\star}+cv\in\cP_{m}$ for all $\norm{v}_{1}\leq1$ and $\iprod{r_{1}}{m}=v_{1}^{\star}$. Then $x^{\star}+\delta$ is in the interior of $\cP_{m}$ when $\norm{\delta}_{1}<c$. According to Proposition~\ref{prop:robustness_nonunique}, $\optival(\xopt+\delta)=\pessval(\xopt+\delta)=v_{1}^{\star}$.
\end{proof}
It can be seen that Proposition~\ref{prop:robustness_reward} implies Proposition~\ref{prop:robustness_nonunique}: By setting $\delta=0$, the result in Proposition~\ref{prop:robustness_reward} recovers the one in Proposition~\ref{prop:robustness_nonunique}. In other words, robustness to uncertainty in the follower's perceived reward is a stronger notion than robustness to nonunique best responses.

\subsection{Robustness to a boundedly rational follower\label{subsec:Robustness to boundedly rational attacker}}

In practice, the follower may not be able to solve the MDP to optimality and may instead produce a response that is only near-optimal, a phenomenon known as \emph{bounded rationality} \citep{simonBehavioralModelRational1955}. One common model of bounded rationality is \emph{quantal response}~\citep{luceIndividualChoiceBehavior1959}, which assumes that the decision-maker takes suboptimal actions with probabilities that diminish exponentially as the corresponding payoffs decrease~\citep{mcfaddenQuantalChoiceAnalysis1976}. More concretely, under a reward allocation $x$, a boundedly rational follower attempts to choose a policy that maximizes the payoff
\begin{equation}
\iprod{r_{2}^{x}}{m}-\tau\cdot\sum_{(s,a)\in\mathcal{S}\times\mathcal{A}}m(s,a)\log\frac{m(s,a)}{\sum_{a'\in\mathcal{A}}m(s,a')},\label{eq:standard ent-regularized}
\end{equation}
where $m$ is the occupancy measure induced by the policy of the follower, $\tau>0$ is a constant. The form of the payoff in~(\ref{eq:standard ent-regularized}) is inspired by the objectives used in entropy-regularized MDPs~\citep{neuUnifiedViewEntropyregularized2017}. The level of irrationality is modeled through the constant $\tau$. When $\tau=0$, the model recovers the rational case, where the follower will choose a best response. At the other extreme, as $\tau\rightarrow\infty$, the role of $r_{2}^{x}$ diminishes, and the follower is inclined to make the occupancy measure spread uniformly among all state-action pairs regardless of $r_{2}^{x}$. Since $\tau\cdot\sum_{(s,a)\in\mathcal{S}\times\mathcal{A}}m(s,a)\log\frac{m(s,a)}{\sum_{a'\in\mathcal{A}}m(s,a')}$ is strictly convex~\citep[Proposition~1]{neuUnifiedViewEntropyregularized2017} in $m$, the function in (\ref{eq:standard ent-regularized}) admits a unique maximizer.

It can be shown that an optimal interior-point allocation is robust to \emph{unmodeled} bounded rationality of the follower. 

\begin{prop}[Robustness to bounded rationality]
\label{prop:robustness of ent-regularized-1}Suppose that $(x^{\star},m^{\star})$ is an optimal solution to the reward design problem in~(\ref{eq:allocation_problem_om}), and $x^{\star}$ is an interior point of $\mathcal{P}_{m^{\star}}$. Let
\begin{equation}
    m_{\tau}^{\star}=\argmax_{m\in\cM}\left\{ \iprod{r_{2}^{x^{\star}}}{m}-\tau\cdot\sum_{(s,a)\in\mathcal{S}\times\mathcal{A}}m(s,a)\log\frac{m(s,a)}{\sum_{a'\in\mathcal{A}}m(s,a')}\right\} 
    \label{eq:def_m_tau_star}
\end{equation}
and $\Mdet{}^{\star}=\Mdet\cap\BR(x^{\star})$. Then for any $\tau>0$, it holds that 
\begin{equation}
\iprod{r_{1}}{m_{\tau}^{\star}}\geq\left(1-\frac{2\tau}{b(1-\gamma)}\log\vert\mathcal{A}\vert\right)\iprod{r_{1}}{m^{\star}},
\label{eq:bounded_rationality_lb}
\end{equation}
where $\gamma$ is the discount factor, and $b=\langle r_{2}^{x^{\star}},m^{\star}\rangle-\max_{m\in\Mdet\backslash\Mdet^{\star}}\langle r_{2}^{x^{\star}},m\rangle$.
\end{prop}

\begin{proof}
See Appendix~\ref{subsec:Transformation-of-Entropy-Regulared}. 
\end{proof}
Since $\cP_{\mopt}$ is an optimal allocation region, the allocation $\xopt$ is an optimal interior-point allocation. The left side of (\ref{eq:bounded_rationality_lb}) is the expected payoff of the leader when the follower is boundedly rational at level $\tau$. Proposition~\ref{prop:robustness of ent-regularized-1} provides a quantitative performance characterization on any optimal interior-point allocation in the presence of a near-rational follower, i.e., when $\tau$ is small. (The bound in~(\ref{eq:bounded_rationality_lb}) becomes vacuous when $\tau$ is large enough for the right side to become negative.) Specifically, the leader is guaranteed to receive an expected payoff not much worse than the optimal payoff $v_{1}^{\star}$ against a completely rational follower. In other words, even when incorrectly treating a near-rational follower as completely rational, the leader will still receive a payoff not far from her prediction. 

The discrepancy between the predicted payoff and the actual payoff depends on $b$, which measures the gap between the optimal payoff and the second-best payoff of a follower who only plays deterministic policies. The constant $b$ depends not only on the configuration of the MDP but also on $x^{\star}$. As $b$ decreases, the follower becomes more indifferent between a best response and a suboptimal response. This implies that a boundedly rational follower is less likely to play an optimal policy as desired by the leader, leading to a potential decrease in the payoff of the leader. 

A similar analysis can be carried out for other models of bounded rationality. For instance, another model of the payoff function of the follower, which is also used in entropy-regularized MDPs~\citep{neuUnifiedViewEntropyregularized2017}, is given by

\begin{equation}
\iprod{r_{2}^{x}}{m}+\tau\cdot\ent(m),\label{eq:boundedly rational reward fucntion}
\end{equation}
where $\ent(m)\deq-\sum_{(s,a)\in\mathcal{S}\times\mathcal{A}}m(s,a)\log m(s,a)$ is the entropy of $m$. For the payoff function in (\ref{eq:boundedly rational reward fucntion}), an optimal interior-point allocation also provably provides a robustness guarantee similar to the one in Proposition~\ref{prop:robustness of ent-regularized-1}; see Proposition~\ref{prop:Robustness to ent-regularized} in Appendix~\ref{subsec: proof of robustness to ent-regularize} for details. The third model of bounded rationality assumes that the follower can take any $\delta$-optimal response, which is a response that yields a payoff within $\delta$ from the optimal payoff \citep{ganRobustStackelbergEquilibria2023}. Appendix~\ref{subsec:delta-optimal-response-as} discusses how an optimal interior-point allocation is robust to $\delta$-optimal responses.

%% file: regions.tex
When an optimal interior-point allocation exists, some optimal allocation region $\cP^{\star}$ must have a nonempty interior. If $\cP^{\star}$ can be identified, then an optimal interior-point allocation can be found by picking any interior point of $\cP^{\star}$. However, is it possible to choose an arbitrary optimal allocation region? If not, is it necessary to examine all optimal allocation regions or only a subset? These questions will be answered in this section.

\subsection{Optimal allocation regions may have an empty interior}

Since an optimal occupancy measure $\mopt$ can be computed by solving the optimization problem in~(\ref{MILP}), it may be tempting to use the corresponding region $\cP_{\mopt}$ as the candidate optimal allocation region $\cP^{\star}$. However, this procedure may fail because not all optimal allocation regions have a nonempty interior.
\begin{example}
\label{exa: deterministic occupancy measure}Consider a similar setting as Example~\ref{exa:opti_pess_values}. As shown in Figure~\ref{fig: two decoys}, aside from $(3,3)$, we also allow reward allocated at $(3,1)$. The transition kernel at $(1,2)$ is modified as follows: If the follower chooses \textsf{right}, he will arrive at $(1,3)$ or $(1,1)$, each with probability $0.5$. If the follower chooses \textsf{up} (resp. \textsf{down}), he will arrive at $(1,3)$ (resp. $(1,1)$). If the follower chooses \textsf{left} or \textsf{stay}, he will remain at $(1,2)$. For any other state, the transition kernel remains deterministic. Denote by $x_{1}$ and $x_{2}$ the amount of resource allocated to the reward at $(3,3)$ and $(3,1)$, respectively. The total allocation budget $C$ satisfies $C>1$. Here, we also assume that the transition kernel at the states with the allocated reward only allows the follower to keep going right. Therefore, the follower cannot go back after arriving at the state with allocated reward.

Define a policy $\pi_{\mathrm{r}}$ of the follower such that $\pi_{\mathrm{r}}$ always chooses \textsf{right} in any state. As a result, because the follower starts at $(1,2)$, by following $\pi_{\mathrm{r}}$ and taking \textsf{right}, the follower may end up in two possible paths: 1) The follower arrives at $(1,3)$ with probability $0.5$ and subsequently follows Path 1; 2) the follower arrives at $(1,1)$ with probability 0.5 and subsequently follows Path 2. 

Let $\cX_{1}=\{x\mid x\succeq0,\ x_{1}=x_{2}+1\}$ and $m_{\mathrm{r}}$ be the occupancy measure induced by $\pi_{\mathrm{r}}$. We will show that $\cP_{m_{\mathrm{r}}}=\cX_{1}$. For any $x\in\cX_{1}$, the payoff for following Path 1 and that for following Path 2 are identical: In either case, the follower will receive a payoff of $2+x_{1}$. Thus, although the follower does not follow either Path 1 or 2 deterministically under $\pi_{\mathrm{r}}$, the follower is still always guaranteed to receive a payoff of $2+x_{1}$. This is also the largest payoff possible because the follower get into absorbing states after 4 steps. Therefore, $\pi_{\mathrm{r}}$ is an optimal policy under $x$, or equivalently $x\in\cP_{m_{\mathrm{r}}}$. On the other hand, for any $x\notin\cX_{1}$, the corresponding payoff of Path 1 will differ from that of Path 2. Let $\pi_{\mathrm{u}}$ (resp. $\pi_{\mathrm{d}}$) be a policy that chooses \textsf{up} (resp. \textsf{down}) at $(2,1)$ and follows Path 1 (resp. Path 2) onward, and $m_{\mathrm{u}}$ (resp. $m_{\mathrm{d}}$) be the induced occupancy measure. The payoff under $\pi_{\mathrm{r}}$ will be strictly lower than the payoff of the better policy between $\pi_{\mathrm{u}}$ and $\pi_{\mathrm{d}}$. The suboptimality of $\pi_{\mathrm{r}}$ implies $x\notin\cP_{m_{\mathrm{r}}}$. 

Consider instead $\cX_{2}=\{x\mid x\succeq0,\ x_{1}\geq x_{2}+1\}$. It is not difficult to verify that $\cP_{m_{\mathrm{u}}}=\cX_{2}$.

Both $m_{\mathrm{r}}$ and $m_{\mathrm{u}}$ are optimal occupancy measures because the follower is guaranteed to reach a state with allocated reward with probability $1$, giving the leader the maximum payoff of $1$. However, while $\cP_{m_{\mathrm{u}}}$ has a nonempty interior, the interior of $\cP_{m_{\mathrm{r}}}$ is empty because its dimension is $1$. 

\begin{figure}
\begin{centering}
\includegraphics[width=0.5\textwidth]{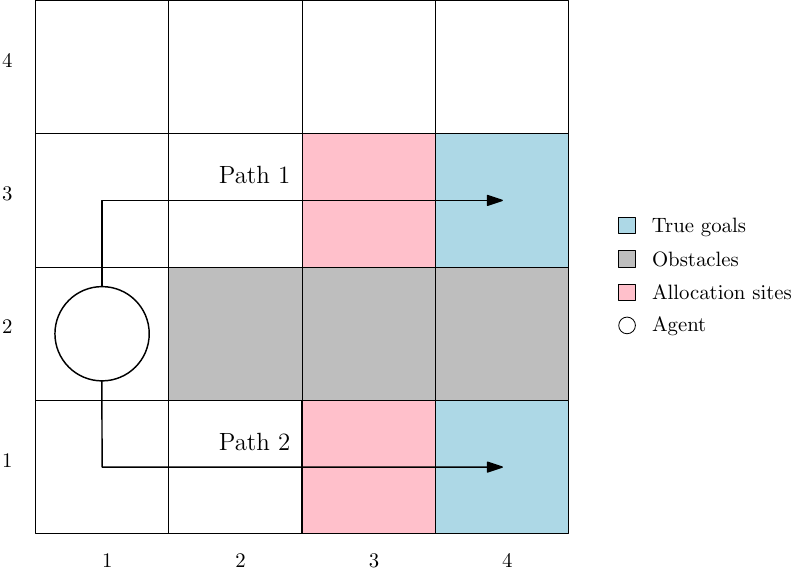}
\par\end{centering}
\caption{A reward design problem with two states that allow reward allocation, and the budget is $C>1$. There are 3 obstacles. All the transitions are deterministic except at $(1,2)$. If the follower chooses \textsf{right} at $(2,2)$, he will end up in $(1,3)$ or $(1,1)$ with equal probability; otherwise, the follower will move deterministically according to his action.\label{fig: two decoys}}

\end{figure}
\end{example}

\subsection{Allocation regions of deterministic occupancy measures}

Since not all optimal allocation regions have a nonempty interior, which region should be chosen to produce an optimal interior-point allocation? At first glance, it appears that one needs to examine every optimal allocation region or, equivalently, every optimal occupancy measure until a region with a nonempty interior is identified. Nevertheless, we will show that it suffices to focus on \emph{deterministic} optimal occupancy measures. 

Let $(x^{\star},m^{\star})$ be an optimal solution of the reward design problem in (\ref{eq:allocation_problem_om}). Let $\pi^{\star}$ be a policy that induces $m^{\star}$. This immediately implies that $\pi^{\star}$ is an optimal policy of the follower. The following proposition ensures that any $\pi\in\Pidet(\pistar)$ is also an optimal policy of the follower.

\begin{prop}
\label{lem:optimal_action} Let $(x^{\star},m^{\star})$ be an optimal solution of the reward design problem in (\ref{eq:allocation_problem_om}) and $\pi^{\star}$ be a policy that induces $m^{\star}$, i.e., $\pistar\in\Pi(\mopt)$. For any $\pi\in\Pidet(\pi^{\star})$, the induced occupancy measure $m^{\pi}$ is a best response of $x^{\star}$, i.e. $m^{\pi}\in\BR(x^{\star})$ or, equivalently, $x^{\star}\in\mathcal{P}_{m^{\pi}}$. 
\end{prop}

\begin{proof}
See Appendix~\ref{subsec:Proof-of-Lemma_optimal_action}.
\end{proof}
Not only is any policy $\pi\in\Pidet(\pistar)$ optimal for the follower, but the policy is also optimal for the leader even though the reward function $r_{1}$ of the leader generally differs from $r_{2}^{\xopt}$.
\begin{prop}
\label{prop:all determinisitc m are optimal}Under the same conditions in Proposition~\ref{lem:optimal_action}, the induced occupancy measure $m^{\pi}$ also satisfies $\iprod{r_{1}}{m^{\pi}}=v_{1}^{\star}$ . 
\end{prop}

\begin{proof}
See Appendix~\ref{subsec:Proof-of-Proposition all  deterministic oc are}.
\end{proof}
Proposition~\ref{prop:all determinisitc m are optimal} not only shows the existence of deterministic optimal occupancy measures but also suggests an algorithm to find them. If the given optimal occupancy measure $\mopt$ is randomized, one can first recover a randomized policy $\pistar$ from $\mopt$. Then, a deterministic policy $\pi\in\Pidet(\pistar)$ can be constructed by examining the actions taken by $\pistar$ at each $s\in\cS$: Whenever more than one action is taken by $\pistar$ with nonzero probability, an arbitrary one of these actions will be assigned to $\pi$ to ensure $\pi\in\Pidet(\pistar)$. The induced deterministic occupancy measure $m^{\pi}$ from $\pi$ is guaranteed to be optimal according to Proposition~\ref{prop:all determinisitc m are optimal}. 
\begin{thm}
\label{thm:deterministic region contains non}Let $m^{\star}$ be a randomized optimal occupancy measure and $\pi^{\star}\in\Pi(m^{\star})$. For any $m\in\Mdet(\pi^{\star})$, it holds that $\mathcal{P}_{\mopt}\subseteq\mathcal{P}_{m}$, and $\cP_{m}$ is an optimal allocation region.
\end{thm}

\begin{proof}
Consider any $x\in\mathcal{P}_{m^{\star}}$. Because $m\in\Mdet(\pistar)$, $m$ is induced by a policy in $\Pidet(\pistar)$ by definition. From Proposition~\ref{lem:optimal_action}, we know that $x\in\mathcal{P}_{m}$, implying $\mathcal{P}_{\mopt}\subseteq\mathcal{P}_{m}$. In addition, Proposition~\ref{prop:all determinisitc m are optimal} ensures that $m$ is an optimal occupancy measure, implying that $\mathcal{P}_{m}$ is an optimal allocation region.
\end{proof}
As implied by Theorem~\ref{thm:deterministic region contains non}, if $\cP_{\mopt}$ has a nonempty interior, so does $\cP_{m}$ for any (deterministic) occupancy measure $m\in\Mdet(\pi^{\star})$, where $\pi^{\star}$ is an arbitrary policy that induces $m^{\star}$. Thus, only allocation regions of deterministic optimal occupancy measures need to be examined to find an optimal interior-point allocation. Recall that the set $\cM$ of occupancy measures is generally uncountably infinite. Theorem~\ref{thm:deterministic region contains non} reduces the relevant occupancy measures to a finite set $\Mdet$. Such a reduction has important theoretical implications and will be used in Section~\ref{sec:Condition-on-Existence} to establish conditions for the existence of optimal interior-point allocations.

It is worth noting that Theorem~\ref{thm:deterministic region contains non} does not imply that the allocation region of a deterministic occupancy measure always has a nonempty interior. For instance, although $\pi_{\mathrm{r}}$ in Example~\ref{exa: deterministic occupancy measure} is a deterministic policy, its corresponding allocation region $\cP_{m_{\mathrm{r}}}$ has an empty interior.

%% file: existence_of_interior.tex
While Section~\ref{sec:ipa_robust} shows that optimal interior-point allocations offer robustness, it remains a question whether an optimal interior-point allocation is guaranteed to exist. Indeed, as shown in Section~\ref{sec:choosing_region}, although it suffices to examine deterministic occupancy measures, the allocation region of a deterministic occupancy measure may still have an empty interior. As will be shown in this section, the allocation budget plays an important role in the existence of an optimal interior-point allocation. Moreover, the existence of an optimal interior-point allocation is not only sufficient (Proposition~\ref{prop:robustness_nonunique}) but also necessary for the existence of a robust allocation against nonunique best responses of the follower.

\subsection{Influence of allocation budget}

Before presenting the general result, we would like to use an example to illustrate how the existence of an optimal interior-point allocation can be affected by the allocation budget. 
\begin{example}
\label{exa:budget}Consider the environment in Example~\ref{exa:opti_pess_values}. Let $\pi_{\mathrm{u}}$ be the policy that follows Path 1 and $m_{\mathrm{u}}$ the occupancy measure induced by $\pi_{\mathrm{u}}$. Without considering the allocation budget, we will show that an allocation $x$ is optimal for the leader if and only if $x\geq1$. When $x\geq1$, from the perspective of the follower, the total payoff of Path 1 is not less than that of Path 2. Thus, $\pi_{\mathrm{u}}$ is optimal for the follower and will give the leader the maximum payoff of $1$. On the other hand, when $x<1$, the follower will choose Path 2 over Path 1 and give the leader a suboptimal payoff of $0$. This also shows that $m_{\mathrm{u}}$ is the only optimal occupancy measure and that the corresponding optimal allocation region is given by $\cP_{m_{\mathrm{u}}}=\{x\mid x\geq1\}$. 

Next, we will show how the allocation budget $C$ affects the existence of an optimal interior-point allocation. When $C>1$, the leader can choose any $x\in(1,C]$, which lies within the interior of $\cP_{m_{\mathrm{u}}}$. Because $\cP_{m_{\mathrm{u}}}$ is an optimal allocation region, $x$ must be an optimal interior-point allocation. However, an optimal interior-point allocation does not exist when $C=1$: The only optimal allocation the leader can choose is $x=1$, which does not belong to the interior of $\cP_{m_{\mathrm{u}}}$. 
\end{example}

Example~\ref{exa:budget} shows that the allocation budget is an important factor in the existence of an optimal interior-point allocation. It suggests that an optimal interior-point allocation may fail to exist when any optimal allocation must exhaust the allocation budget. To examine whether the budget needs to be exhausted, consider the following problem:

\begin{equation}
\begin{alignedat}{2} & \optmax_{x\in\mathcal{X},\ m\in\mathcal{M}} & \quad & C-\sum_{i=1}^{\vert\mathcal{S}_{d}\vert}x_{i}\\
 & \optst &  & \iprod{r_{1}}{m}=v_{1}^{\star}\\
 &  &  & m\in\BR(x).
\end{alignedat}
\label{eq:margin problem}
\end{equation}
The constraints in~(\ref{eq:margin problem}) ensure that $(x,m)$ is an optimal solution to the reward design problem in~(\ref{eq:allocation_problem_om}). Thus, the optimal value of~(\ref{eq:margin problem}) is $0$ if and only if any optimal allocation must exhaust the budget. The following theorem shows that the optimal value of problem~(\ref{eq:margin problem}) determines the existence of an optimal interior-point allocation.
\begin{thm}
\label{Thm:Sufficiency}The reward design problem in~(\ref{eq:allocation_problem_om}) admits at least one optimal interior-point allocation if and only if the optimal value of problem~(\ref{eq:margin problem}) is strictly positive.
\end{thm}

\begin{proof}
See Appendix~\ref{subsec:Proof-of-Theorem_Sufficiency}.
\end{proof}

\subsection{Necessity for robustness to nonunique best responses}

Example~\ref{exa:budget} also hints at a relationship between the existence of an optimal interior-point allocation and robustness to nonunique best responses (introduced in Section~\ref{subsec:robustness_nonunique}). When $C=1$, under the only optimal allocation $x=1$, the best responses of the follower is not unique: It is optimal for the follower to choose either Path 1 or Path 2. However, the former will give the leader a payoff of $1$, whereas the latter will give a payoff of $0$. In other words, an optimal allocation that is robust to nonunique best responses does not exist. Such nonexistence happens to coincide with the nonexistence of optimal interior-point allocations.

In fact, this is more than a coincidence. The existence of an optimal interior-point allocation is a sufficient and necessary condition for the existence of an allocation robust to nonunique best responses. Notice that the sufficiency has already been established in Proposition~\ref{prop:robustness_nonunique}. The following theorem focuses only on the necessity.
\begin{thm}[Necessity of optimal interior-point allocations]
\label{prop:best responses to existence}If there exists an optimal allocation $\xopt$ satisfying $\optival(\xopt)=\pessval(\xopt)=v_{1}^{\star}$, then an optimal interior-point allocation must exist.
\end{thm}

\begin{proof}
See Appendix~\ref{subsec:Proof-of-Theorem_best_response_to_existence}.
\end{proof}
Theorem~\ref{prop:best responses to existence} does not, however, guarantee that $\xopt$ is an optimal interior-point allocation when $\xopt$ satisfies $\optival(\xopt)=\pessval(\xopt)=v_{1}^{\star}$. For instance, consider any allocation $\xopt\in\cX_{1}$ in Example~\ref{exa: deterministic occupancy measure}. Due to Theorem~\ref{thm:deterministic region contains non}, only optimal allocation regions of deterministic occupancy measures need to be examined. The allocation $\xopt$ belongs to three optimal allocation regions of deterministic occupancy measures: $\cP_{m_{\mathrm{r}}}$, $\cP_{m_{\mathrm{u}}}$, and $\cP_{m_{\mathrm{d}}}$. It is not difficult to see that $\xopt$ is robust to nonunique best responses. However, $\xopt$ is not an interior point of any optimal allocation region. First, $\xopt$ cannot be an interior point of $\cP_{m_{\mathrm{r}}}$ because Example~\ref{exa: deterministic occupancy measure} has already shown that $\cP_{m_{\mathrm{r}}}$ has an empty interior. Second, $\xopt$ is not in the interior of $\cP_{m_{\mathrm{u}}}$ or $\cP_{m_{\mathrm{d}}}$ since it is possible to perturb $\xopt$ arbitrarily small to make $m_{\mathrm{u}}$ or $m_{\mathrm{d}}$ no longer a best response of the follower.

%% file: milp.tex
For finding an optimal interior-point allocation, Theorem~\ref{thm:deterministic region contains non} shows that only a finite number of optimal allocation regions (i.e., the ones of deterministic occupancy measures) need to be examined. Nevertheless, the number of regions can be as large as the number of deterministic policies of the follower, which is given by $|\cA|^{|\cS|}$. Thus, it may be impractical to enumerate all the allocation regions of deterministic occupancy measures. This section presents a practical method for computing an optimal interior-point allocation via MILP. 

Suppose that the optimal value $v_{1}^{\star}$ of problem~(\ref{eq:allocation_problem_om}) has been obtained using the procedure described in Section~\ref{subsec:Computing-an-optimal}. Our goal is to find $(x,m)$ that satisfies the following conditions: 
\begin{enumerate}
\item $x$ is an admissible allocation: $x\in\cX$.
\item $m$ is an optimal occupancy measure: $m\in\cM$, and $\iprod{r_{1}}{m}=v_{1}^{\star}$.
\item $x$ is an interior point of $\cP_{m}$: There exists $c>0$ such that (\ref{eq:interior_pt_condition}) holds.
\end{enumerate}
An issue with condition~(\ref{eq:interior_pt_condition}) is that the condition would lead to infinitely many constraints because it requires checking every $v$ satisfying $\norm{v}_{1}\leq1$. Denote by $\{e_{i}\}_{i=1}^{|\Sd|}$ the standard basis of $\reals^{|\Sd|}$. Recall that the unit $\ell_{1}$-norm ball in $\reals^{|\Sd|}$ is the convex hull of $\{\pm e_{i}\}_{i=1}^{|\Sd|}$. Since $\cP_{m}$ is a convex set, condition~(\ref{eq:interior_pt_condition}) is equivalent to

\begin{equation}
x+ce_{i}\in\cP_{m},\quad x-ce_{i}\in\cP_{m},\quad i=1,\dots,|\Sd|\label{eq:1-norm transformation}
\end{equation}
or, equivalently,
\begin{equation}
m\in\BR(x+ce_{i}),\quad m\in\BR(x-ce_{i}),\quad i=1,\dots,|\mathcal{S}_{d}|.\label{eq:br}
\end{equation}
According to the KKT conditions in~(\ref{eq:KKT}), condition~(\ref{eq:br}) holds if and only if  there exist $\nu_{i}^{+}$ and $\nu_{i}^{-}$ for $i=1,\dots,|\mathcal{S}_{d}|$ such that 
\begin{align}
 & Am=\rho,\qquad m\succeq0\label{eq:primal_feasibility}\\
 & A^{T}\nu_{i}^{+}-r_{2}^{x+ce_{i}}\succeq0,\quad A^{T}\nu_{i}^{-}-r_{2}^{x-ce_{i}}\succeq0,\quad i=1,\dots,|\mathcal{S}_{d}|\label{eq:dual_feasibility}\\
 & m\perp A^{T}\nu_{i}^{+}-r_{2}^{x+ce_{i}},\quad m\perp A^{T}\nu_{i}^{-}-r_{2}^{x-ce_{i}},\quad i=1,\dots,|\mathcal{S}_{d}|.\label{eq:comp_constr}
\end{align}

Rather than an arbitrary optimal interior-point allocation, it is actually possible to find an optimal interior-point allocation with the \emph{maximum} margin by solving the following optimization problem: 
\begin{equation}
\begin{alignedat}{2} & \optmax_{x,m,c,\nu_{i}^{+},\nu_{i}^{-}} & \quad & c\\
 & \optst &  & x\in\mathcal{X},\quad\iprod{r_{1}}{m}=v_{1}^{\star},\quad\textrm{\eqref{eq:primal_feasibility}--\eqref{eq:comp_constr}}.
\end{alignedat}
\label{eq:prob_cc}
\end{equation}
Let $(x^{\star},m^{\star},c^{\star})$ be an optimal solution of problem~(\ref{eq:prob_cc}). If $c^{\star}>0$, then $x^{\star}$ is an interior point of the optimal allocation region $\mathcal{P}_{m^{\star}}$ and hence is an optimal interior-point allocation with margin $c^{\star}$.

Similar to (\ref{eq:MILP_c}), the complementarity constraints in (\ref{eq:comp_constr}) can be reformulated as affine constraints with integer variables. Since all the remaining constraints in (\ref{eq:prob_cc}) are affine, problem (\ref{eq:prob_cc}) can be reformulated as an MILP and solved by off-the-shelf solvers including Gurobi~\citep{gurobi} and CPLEX~\citep{CPLEX}.

%% file: reward_regions.tex
\subsection{Optimal interior-point reward functions\label{subsec:oipa-incentive}}

In Section~\ref{sec:ipa_robust}, we introduced the concept of an optimal interior-point allocation and proved its robustness when the leader's reward function is in a special form given in Section~\ref{subsec:prelim-setup}. In this section, we will extend the definition of the optimal interior-point allocation in order to study robust reward design when the leader uses a general reward function. When $r_{1}$ is chosen arbitrarily, the following example shows that an interior-point allocation may no longer offer robustness to nonunique best responses, let alone the other two stronger notions of robustness. 
\begin{example}
Consider a $4\times4$ grid world in Figure~\ref{fig:counter_example}. The follower starts from $(1,2)$. The action space of the follower and the transition kernel are the same as Example~\ref{exa:opti_pess_values}. We assume a discount factor of $1$. There is no true goal for the follower. The leader is interested in the state $(3,3)$: When the follower enters $(3,3)$, the leader receives a payoff of $1$. The leader's reward is $0$ in all other states. The reward allocation is only allowed at $(4,2)$, and the allocation budget is $1$. Once the follower enters $(4,2)$, the follower cannot leave the state. 

Suppose a positive reward is allocated at $(4,2)$. In this case, both Path 1 and Path 2 lead the follower to $(4,2)$, the only state that the follower can get rewarded. Denote the reward allocation by $x$. Denote the occupancy measures induced by following Path 1 and Path 2 by $m_{1}$ and $m_{2}$, respectively. Then, $\cP_{m_{1}}=\cP_{m_{2}}=\{x\mid0<x\leq1\}$, and both $\cP_{m_{1}}$ and $\cP_{m_{2}}$ are optimal allocation regions. Consider an optimal interior-point allocation, for instance, $x=0.5$. When $x=0.5$, the follower is indifferent to taking Path 1 or Path 2. However, the leader's payoff will be $1$ if the follower chooses Path 1 and $0$ otherwise. This means $\pessval(x)=0\neq\optival(x)=1$ when $x$ is an optimal interior-point allocation. The example, however, does not contradict with Proposition~\ref{prop:robustness_nonunique} because $r_{1}$ differs from the one defined in Section~\ref{subsec:prelim-setup}.

\begin{figure}
\begin{centering}
\includegraphics[width=0.5\textwidth]{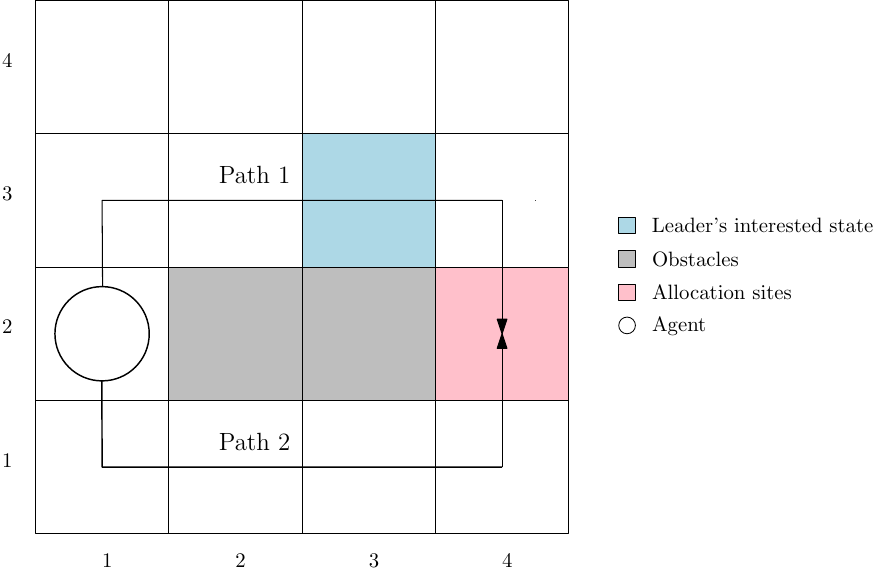}
\par\end{centering}
\caption{The reward design problem with a general leader's reward function. The leader is interested in the state $(3,3)$, and the leader's reward is $1$ in this state and $0$ in all other states. The reward allocation is only admissible at $(4,2)$. There is no true goal for the follower.\label{fig:counter_example}}

\end{figure}
\end{example}

We first examine how the induced occupancy measure $m$ of the follower is affected by its reward function. For a given follower's reward function $\mr\in\mathbb{R}^{\vert\mathcal{S}\vert\vert\mathcal{A}\vert}$, define the set of best responses of $\mr$ by
\[
\BR(\mr)\deq\argmax_{m\in \mathcal{M}}\iprod{\mr}{m}.
\]
It is possible to divide the space $\mathbb{R}^{\vert\mathcal{S}\vert\vert\mathcal{A}\vert}$ of reward functions into (possibly overlapping) regions based on the corresponding best response. Each region is called a \emph{reward region} and is identified by a unique occupancy measure.
\begin{defn}[Reward region]
\label{def:allocation_region-1}Let $m\in\cM$ be an occupancy measure. The \emph{reward region} of $m$ is defined by
\begin{equation*}
\mathcal{\tilde{P}}_{m}=\{\mr\in\mathbb{R}^{|\mathcal{S}||\mathcal{A}|}\mid m\in\BR(\mr)\}=\{\mr\in\mathbb{R}^{|\mathcal{S}||\mathcal{A}|}\mid\iprod{\mr}{m}\geq\iprod{\mr}{m'}\text{ for all }m'\in\mathcal{M}\}.
\end{equation*}
Furthermore, the reward region of an optimal occupancy measure of the reward design problem~(\ref{eq:allocation_problem_om}) is called an \emph{optimal reward region}. 
\end{defn}

The name reward region stands for the fact that the region consists of reward functions. Similarly to Definition~\ref{def:interior-point}, it is possible to define \emph{interior-point reward functions} and \emph{optimal interior-point reward functions}.
\begin{defn}[Interior-point reward function]
\label{def:interior-point-1}Let $m\in\cM$ be an occupancy measure. A reward function $\mr\in\mathbb{R}^{\vert\mathcal{S}\vert\vert\mathcal{A}\vert}$ is called an\emph{ interior-point reward function of $\tilde{\cP}_{m}$} if there exists $\tilde{c}>0$ such that
\begin{equation}
\mr+\tilde{c}v\in\tilde{\cP}_{m}\quad\text{for all }\norm{v}_{1}\leq1.\label{eq:interior_pt_condition-1}
\end{equation}
The largest value of $\tilde{c}$ for (\ref{eq:interior_pt_condition-1}) to hold is called the \emph{reward margin} of $\mr$. An interior-point of an optimal reward region is called an \emph{optimal interior-point reward}\textit{ function}.
\end{defn}

\subsection{Robustness of \textit{\emph{optimal interior-point reward functions}}}

The properties of optimal interior-point reward functions are similar to those of optimal interior-point allocations shown in Section~\ref{sec:ipa_robust}. Previously, we defined $\optival(x)$ and $\pessval(x)$ as the optimal value of (\ref{eq:optimistic problem}) and (\ref{eq:pessimistic problem}), respectively. Given a reward function $\mr$, we abuse the notation and similarly define $\optival(\mr)$ as the optimal value of
\begin{equation*}
\begin{alignedat}{2} & \optmax_{m\in\mathcal{M}} & \quad & \iprod{r_{1}}{m}\\
 & \optst &  & m\in\BR(\mr),
\end{alignedat}
\end{equation*}
and $\pessval(\mr)$ as the optimal value of 
\begin{equation*}
\begin{alignedat}{2} & \optmin_{m\in\mathcal{M}} & \quad & \iprod{r_{1}}{m}\\
 & \optst &  & m\in\BR(\mr).
\end{alignedat}
\end{equation*}
Denote the interior of a reward region $\tildep$ by $\mathrm{int}(\tildep)$.

\begin{prop}
\label{prop:unique best response}If $\mr\in\mathrm{int}(\tildep_{m})$ for some $m$, then $\BR(\mr)=\{m\}$.
\end{prop}

\begin{proof}
Since $\mr\in\mathrm{int}(\tildep_{m})$, for any standard basis vector $e_{i}$ ($i=1,\dots,|\cS|\cdot|\cA|$) of $\mathbb{R}^{\vert\mathcal{S}\vert\cdot\vert\mathcal{A}\vert}$, there exists some $c>0$ such that $\mr+ce_{i},\mr-ce_{i}\in\tildep_{m}$. Therefore, given any $m'\in\mathcal{M}$, it holds that $\langle\mr+ce_{i},m\rangle\geq\langle\mr+ce_{i},m'\rangle$ or equivalently 
\begin{equation}
\langle\mr,m\rangle\geq\langle\mr,m'\rangle+\langle ce_{i},m'-m\rangle.\label{eq:m_diff-1}
\end{equation}
Similarly, it holds that $\langle\mr-ce_{i},m\rangle\geq\langle\mr-ce_{i},m'\rangle$ or equivalently 
\begin{equation}
\langle\mr,m\rangle\geq\langle\mr,m'\rangle+\langle ce_{i},m-m'\rangle.\label{eq:m_diff-2}
\end{equation}
Combine~(\ref{eq:m_diff-1}) and (\ref{eq:m_diff-2}) to obtain 
\begin{equation}
\langle\mr,m\rangle\geq\langle\mr,m'\rangle+|\langle ce_{i},m-m'\rangle|.\label{eq:m_diff-3}
\end{equation}
When $m'\neq m$, there must exist $(s,a)$ such that $m(s,a)\neq m'(s,a)$. Thus, there exists some $i\in\{1,2,\ldots,|\mathcal{S}|\cdot|\mathcal{A}|\}$ such that $\langle e_{i},m-m'\rangle=m(s,a)-m'(s,a)\neq0$. It then follows from~(\ref{eq:m_diff-3}) that $\langle\mr,m\rangle>\langle\mr,m'\rangle$. 
\end{proof}
If $\mr$ is an optimal interior-point reward function, then $\optival(\mr)=v_{1}^{\star}$. Furthermore, Proposition~\ref{prop:unique best response} shows that the best response of $\mr$ is unique, which immediately implies $\optival(\mr)=\pessval(\mr)=v_{1}^{\star}$. 
\begin{cor}[Robustness to nonunique best responses]
\label{coro:robustness_nonunique}For a reward allocation $x$, if $r_{2}^{x}$ is an optimal interior-point reward function, then $\optival(r_{2}^{x})=\pessval(r_{2}^{x})=v_{1}^{\star}$.
\end{cor}

Using the notion of optimal interior-point reward functions, Propositions~\ref{prop:robustness_nonunique} and~\ref{prop:robustness of ent-regularized-1} can also be extended to the case of general $r_{1}$. 
\begin{prop}[Robustness to reward perception of the follower]
\label{prop:robustness_reward-1}For a reward allocation $x$, if $r_{2}^{x}$ is an optimal interior-point reward function with margin $\tilde{c}>0$, then $\optival(r_{2}^{x}+\delta)=\pessval(r_{2}^{x}+\delta)=v_{1}^{\star}$ for all $\norm{\delta}_{1}<\tilde{c}$.
\end{prop}

\begin{proof}
Similar to the proof of Proposition~\ref{prop:robustness_nonunique}.
\end{proof}
\begin{prop}[Robustness to bounded rationality]
\label{prop:Robustness to ent-regularized-1}Suppose that $\iprod{r_{1}}{\mopt}=v_{1}^{\star}$ and $r_{2}^{\xopt}$ is an interior-point reward function of $\tilde{\mathcal{P}}_{m^{\star}}$. Denote by $\Mdet^{\star}\deq\Mdet\cap\BR(r_{2}^{\xopt})$ the set of deterministic optimal occupancy measures under $r_{2}^{\xopt}$. Then for any $\tau>0$, it holds that
\[
\iprod{r_{1}}{m_{\tau}^{\star}}\geq\left(1-\frac{2\tau}{b(1-\gamma)}\log\vert\mathcal{A}\vert\right)v_{1}^{\star},
\]
where $m_{\tau}^{\star}$ is given by \eqref{eq:def_m_tau_star}, $\gamma$ is the discount factor, and $b=\langle r_{2}^{\xopt},m^{\star}\rangle-\max_{m\in\Mdet\setminus\Mdet^{\star}}\langle r_{2}^{\xopt},m\rangle$.
\end{prop}

\begin{proof}
Similar to the proof of Proposition~\ref{prop:robustness of ent-regularized-1}.
\end{proof}
Results similar to Proposition~\ref{prop:Robustness to ent-regularized-1} hold when the bounded rationality of the follower is characterized by other models, including an alternative entropy-regularized MDP (Appendix~\ref{subsec: proof of robustness to ent-regularize}) and the $\delta$-optimal responses (Appendix~\ref{subsec:delta-optimal-response-as}).

\subsection{Existence of an optimal interior-point reward function}

Although the existence of an optimal interior-point reward function cannot be shown similarly to Theorem~\ref{Thm:Sufficiency}, some insights can still be gained on the relationship between different reward regions. The following corollary follows immediately from Proposition~\ref{prop:unique best response} on the uniqueness of best response.
\begin{cor}
\label{corno intersection of int and bd}For any $m_{1},m_{2}\in\Mdet$ with $m_{1}\neq m_{2}$, $\mathrm{int}(\mathcal{\tildep}_{m_{1}})\cap\mathcal{\tildep}_{m_{2}}=\emptyset$.
\end{cor}

\begin{proof}
Consider $\mr\in\mathrm{int}(\mathcal{\tildep}_{m_{1}})$. It follows from Proposition~\ref{prop:unique best response} that $\BR(\mr)=\{m_{1}\}$. This implies $m_{2}\notin\BR(\mr)$ or, equivalently, $\mr\notin\tildep_{m_{2}}$.
\end{proof}
Theorem~\ref{thm:deterministic region contains non} can be generalized to the case of general $r_{1}$ to characterize the relationship between reward regions of optimal deterministic and randomized occupancy measures. 
\begin{prop}
\label{thm:deterministic region contains non-1 }Let $m^{\star}$ be a randomized optimal occupancy measure and $\pi^{\star}\in\Pi(m^{\star})$. For any $m\in\Mdet(\pi^{\star})$, it holds that $\tilde{\mathcal{P}}_{\mopt}\subseteq\mathcal{\tilde{P}}_{m}$, and $\tilde{\cP}_{m}$ is an optimal reward region.
\end{prop}

\begin{proof}
The proof is based on Lemma~\ref{lem:optimal_action-2} and Lemma~\ref{lem: optimal_action-3} in Appendix~\ref{sec:appendix-reward} and is similar to the proof of Theorem~\ref{thm:deterministic region contains non}.
\end{proof}
\begin{cor}
Let $m^{\star}$ be an optimal occupancy measure and $\pi^{\star}\in\Pi(m^{\star})$. For any $m\in\Mdet(\pi^{\star})$, if $m\neq m^{\star}$, it holds that $\tilde{\mathcal{P}}_{\mopt}\subseteq\mathrm{bd}(\mathcal{\tilde{P}}_{m})$, where $\mathrm{bd}(\mathcal{\tilde{P}}_{m})$ is the boundary of $\mathcal{\tilde{P}}_{m}$.
\end{cor}

\begin{proof}
Since $m\neq m^{\star}$, from Corollary~\ref{corno intersection of int and bd}, it holds that $\mathrm{int}(\mathcal{\tilde{P}}_{m})\cap\tilde{\mathcal{P}}_{\mopt}=\emptyset$. Furthermore, $\mathcal{\tilde{P}}_{m}$ is a closed set by definition, which implies $\mathcal{\tilde{P}}_{m}=\mathrm{bd}(\mathcal{\tilde{P}}_{m})\cup\mathrm{int}(\mathcal{\tilde{P}}_{m})$. Since $\tilde{\mathcal{P}}_{\mopt}\subseteq\mathcal{\tilde{P}}_{m}$ according to Proposition~\ref{thm:deterministic region contains non-1 }, it follows that $\tilde{\mathcal{P}}_{\mopt}\subseteq\mathrm{bd}(\mathcal{\tilde{P}}_{m})$.
\end{proof}
Proposition~\ref{thm:deterministic region contains non-1 } states the following fact: Suppose the reward region $\tildep_{\mopt}$ of some optimal randomized occupancy measure $m^{\star}$ contains an optimal interior-point reward function with margin $\tilde{c}$. Let $\pi^{\star}\in\Pi(m^{\star})$. Then for any $m\in\Mdet(\pi^{\star})$, the corresponding reward region $\tildep_{m}$ must also contain an optimal interior-point reward function whose margin is no smaller than $\tilde{c}$. This reveals the special role of deterministic occupancy measures for the problem.

%% file: experiments.tex
In this section, we numerically validate the robustness of the maximum-margin optimal interior-point allocation given by (\ref{eq:prob_cc}) and compare it with an arbitrary optimal allocation given by (\ref{MILP}). Denote by $x_{\mathrm{IP}}$ the optimal interior-point allocation given by (\ref{eq:prob_cc}), and denote by $x_{\mathrm{MILP}}$ an arbitrary optimal allocation given by (\ref{MILP}). We sometimes refer to $x_{\mathrm{IP}}$ as the \emph{interior-point solution} and $x_{\mathrm{MILP}}$ as the\emph{ MILP solution}. The robustness of allocation is evaluated in three different environments: a $6\times6$ grid world, a $10\times10$ grid world, and a probabilistic attack graph. We will check two notions of robustness: robustness to nonunique best responses and robustness to a boundedly rational attacker. The remaining notion of robustness, i.e., robustness to uncertain reward perception of the attacker, will be demonstrated by the computed margin of allocation. All numerical experiments were performed on a Macbook Air laptop computer with an Apple M2 processor and 8~GB RAM running macOS Sonoma 14.3.1. The interior-point solutions and MILP solutions in different environments are computed using the Python MIP package with Gurobi 11.0.0.

\subsection{Environments\label{subsec:environments}}

\subsubsection{$6\times6$ grid world}

The $6\times6$ grid world illustrated in Figure~\ref{fig:6x6-grid-world} is modeled as an MDP with finite state and action spaces. Each cell is associated with a 2-tuple that represents the horizontal and vertical coordinates of the cell: The coordinate of the lower-left cell is $(1,1)$, and that of the upper-right cell is $(6,6)$. The grid world consists of a defender and an attacker, who play the roles of the leader and the follower, respectively. The attacker always starts from the state $(1,3)$ and can choose one of the four directions, \textsf{left}, \textsf{right}, \textsf{up}, and \textsf{down}, as the desired direction to move. The attacker will transit to the adjacent cell in the desired direction with probability $0.8$ if the new cell remains within the boundary. The attacker may also move laterally to a new cell in one of the two directions, each with probability $0.1$. The attacker will never move opposite to the desired direction. In any case, if the new cell is out of bounds, then the attacker will stay at the current cell. The grid world has two true goals at $(1,6)$ and $(5,4)$ and two allocated rewards set up by the defender at $(5,2)$ and $(6,5)$. From the perspective of the attacker, the allocated rewards are indistinguishable from the true targets. The perceived reward of the attacker for reaching any of the true goals is $1$. The perceived reward for reaching an allocated reward is equal to the allocated resource therein by the defender. The allocated resource at each allocated reward must be nonnegative, and the total budget of allocation is $4$. The MDP ends when the attacker arrives at a true goal or an allocated reward. The goal of the defender is to find an allocation strategy that maximizes the probability for the attacker to receive an allocated reward. This goal is captured by the reward function of the defender: $r_{1}(s,a)=1\ \forall(s,a):s\in\mathcal{S}_{d}$, and $r_{1}(s,a)=0\ \forall(s,a):s\notin\mathcal{S}_{d}$, which is consistent with the form of $r_1$ introduced in Section~\ref{subsec:prelim-setup}. Using the notation from Section~\ref{sec:Problem mdp}, the problem setup is summarized as follows:
\begin{itemize}
\item State space $\text{\ensuremath{\mathcal{S}}}=\{1,2,\cdots,6\}^{2}$ represents the cells in the $6\times6$ grid world.
\item Action space $\cA=\{\textsf{left},\textsf{right},\textsf{up},\textsf{down}\}$ represents the four desired directions of movement. 
\item The set $\mathcal{S}_{d}=\{s_{1},s_{2}\}$, where $s_{1}=(1,4)$ and $s_{2}=(4,5)$, represents the cells at which the allocated rewards are located.
\item The reward function of the attacker is given by $r$, and that of the defender is given by $r_{1}$. Both functions are defined in Section~\ref{subsec:prelim-setup}.
\end{itemize}
\begin{figure}
\begin{centering}
\includegraphics[width=0.6\textwidth]{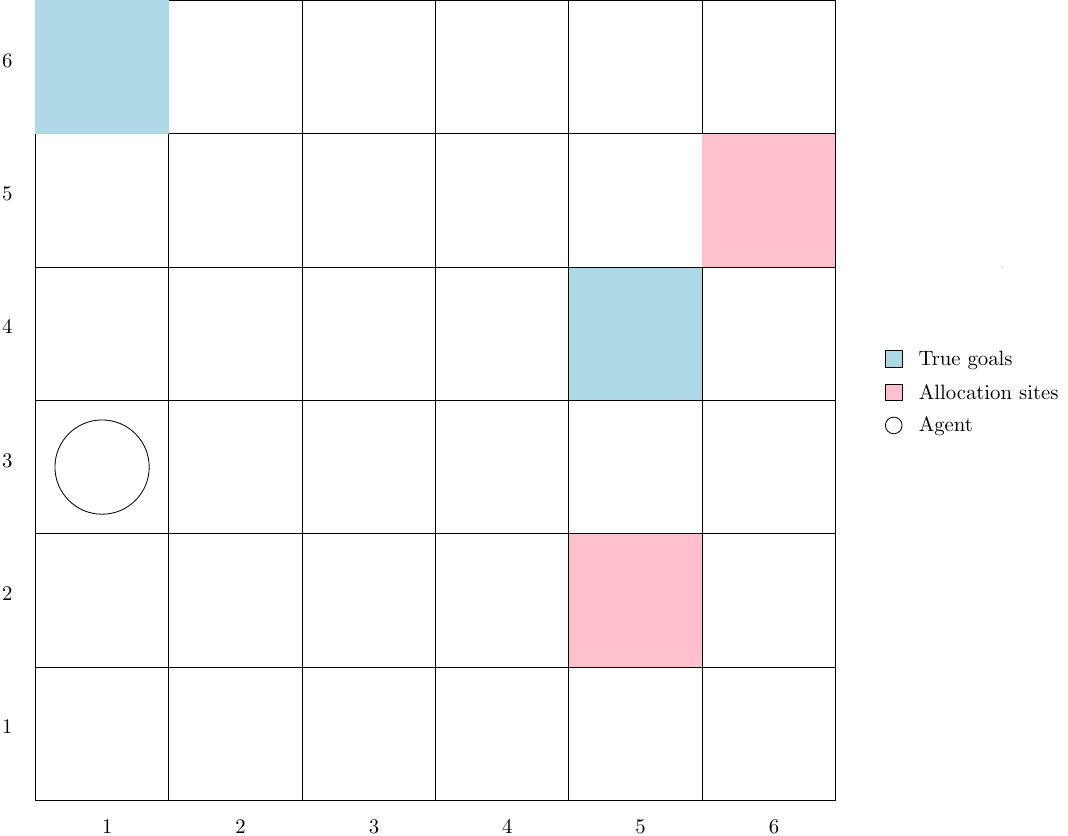}
\par\end{centering}
\caption{The $6\times6$ grid world. The true goals are at $(1,6)$ and $(5,4)$, and the defender can allocate rewards at $(5,2)$ and $(6,5)$. The attacker always starts from $(1,3)$. \label{fig:6x6-grid-world}}
\end{figure}

\subsubsection{$10\times10$ grid world}

The transition kernel, the total budget of allocation, and actions of the attacker in the $10\times10$ grid world (Figure~\ref{fig:10x10-grid-world}) are the same as those in the $6\times6$ grid world. The true goals are at $(6,9)$, $(8,5)$, and $(10,3)$, and the allocated rewards are at $s_{1}=(8,1)$, $s_{2}=(7,4)$, and $s_{3}=(6,9)$. The attacker always starts from the state $(2,6)$. 

\begin{figure}
\begin{centering}
\includegraphics[width=0.8\textwidth]{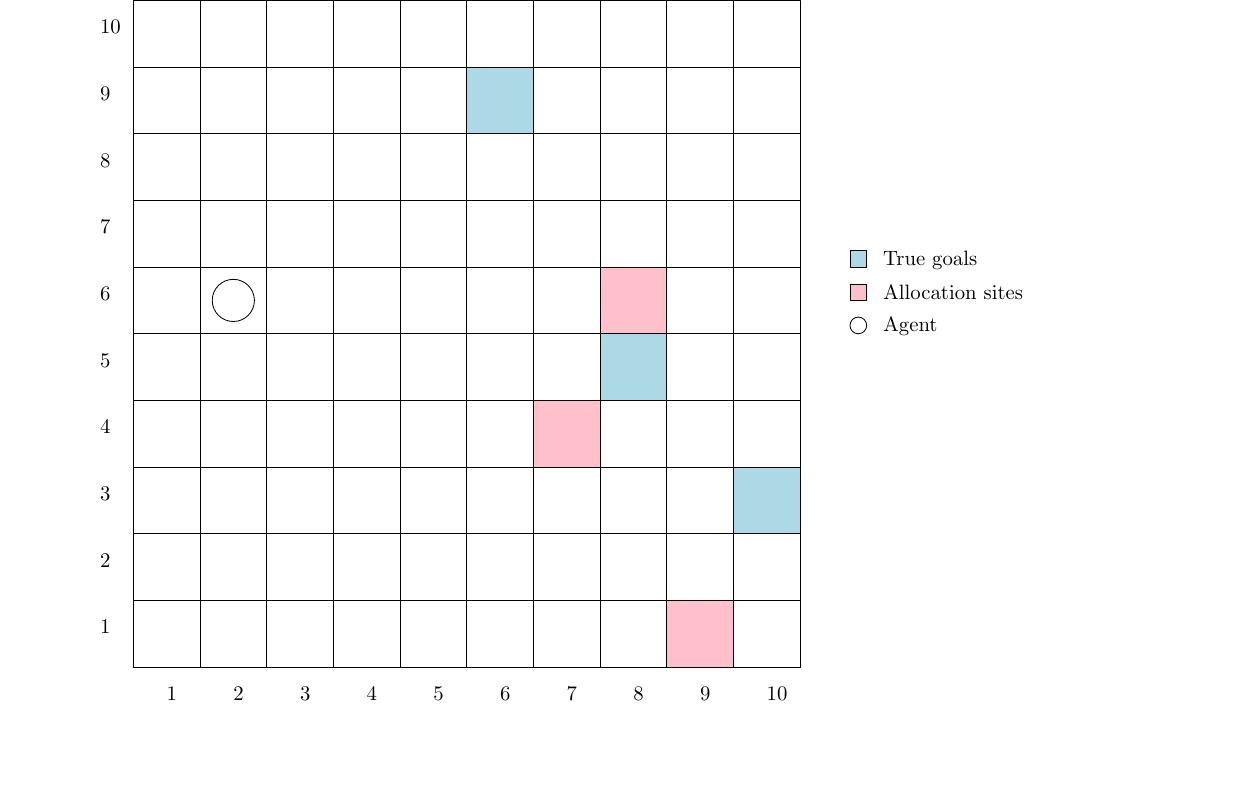}
\par\end{centering}
\caption{The $10\times10$ grid world. The true goals are at $(6,9)$, $(8,5)$ and $(10,3)$, and the defender can allocate rewards at $(8,1)$, $(7,4)$ and $(6,9)$. The attacker always starts from $(2,6)$. \label{fig:10x10-grid-world} }
\end{figure}

\subsubsection{Probabilistic attack graph}

\begin{figure}[h]
	\centering
	\resizebox{0.5\linewidth}{!}{
		\begin{tikzpicture}[->,>=stealth',shorten >=1pt,auto,node distance=2  cm, scale =0.5,transform shape, line width = 0.5pt]
		\node[state, initial] (0) {$0$};
		\node[state] (2) [above right of=0] {$2$};
		\node[state] (3) [below right of=0] {$3$};
		\node[state] (1) [above of=2] {$1$};
		\node[state] (4) [below of=3] {$4$};
		\node[state] (5) [right of=1] {$5$};
		\node[state] (6) [right of=4] {$6$};
		\node[state] (7) [right of=3] {$7$};
		\node[state] (8) [right of=2] {$8$};
		\node[state] (9) [above right of=7] {$9$};
		\node[state,fill={rgb,255:red,173; green,216; blue,230}] (10) [below right of=7] {$10$};
		\node[state,fill={rgb,255:red,255; green,192; blue,203}] (12) [above right of=9] {$12$};
		\node[state,fill={rgb,255:red,255; green,192; blue,203}] (11) [below right of=9] {$11$};
		
		\path (0) edge[line width = 1.2pt, color = black] node {} (1)
		(0) edge[] node {} (2)
		(0) edge[] node {} (3)
		(0) edge[]  node {} (4)
		(1) edge[line width = 1.2pt, color = black] node {} (5)
		(1) edge[] node {} (8)
		(1) edge[] node {} (6)
		(2) edge[line width = 1.2pt, color = black] node {} (6)
		(2) edge[] node {} (7)
		(3) edge[] node {} (5)
		(3) edge[] node {} (7)
		(4) edge[]  node {} (5)
		(4) edge[] node {} (7)
		(6) edge[] node {} (9)
		(6) edge[] node {} (10)
		(7) edge[line width = 1.2pt, color = black] node {} (9)
		(7) edge[] node {} (8)
		(9) edge[] node {} (11)
		(9) edge[] node {} (12)
		;
		\end{tikzpicture}
	}  
	\caption{A probabilistic attack graph.}
	\label{fig: simple probabilistic attack graph}
\end{figure}
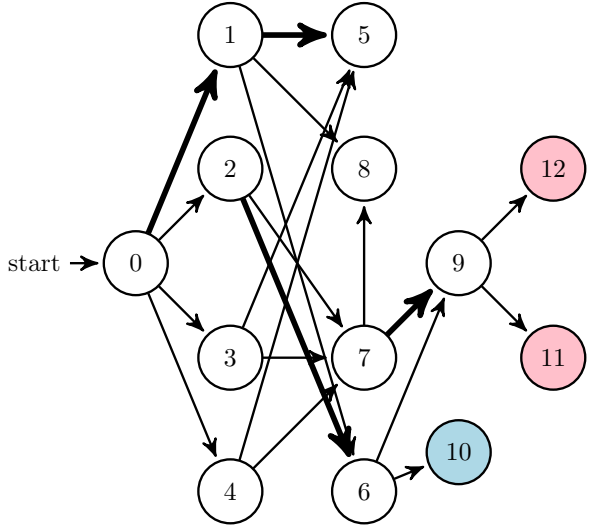

Figure~\ref{fig: simple probabilistic attack graph} shows the probabilistic attack graph used in our numerical experiments. The probabilistic attack graph is an MDP whose state space consists of all the nodes in the graph. The attacker has four actions $\{a,b,c,d\}$.  The transition probabilities of the MDP are defined by the edges of the graph. For clarity, the graph only shows the possible transitions given action $a$, where a thick (resp. thin) arrow represents a high (resp. low) transition probability. For example, when the attacker is at node $0$ and takes action $a$, the transition probabilities are given by $\cT(0, a, 1) =0.7$ and $\cT(0, a, i) = 0.1$ for $i = 2,3,4$. A complete description of the transition function $\cT$ is provided at the URL in Footnote\footnote{\url{https://www.dropbox.com/s/nyycf57vdry139j/MDPTransition.pdf}}. The attacker always starts from state $0$. The only true goal is node $10$. The reward allocations are set at $\{11, 12\}$. Similar to the grid world environments, the perceived reward of the attacker for reaching the true goal is $1$. The perceived reward for reaching an allocated reward is equal to the allocated resource therein by the defender. The total budget of allocation is $4$. The reward function $r_1$ of the defender is the same as the one used in the grid world environments.

\subsection{Robustness to nonunique best responses\label{subsec:Robustness-to-nonunique experiment}}

\subsubsection{Method of validation}

Our goal is to verify Proposition~\ref{prop:robustness_nonunique} that the interior-point solution $x_{\mathrm{IP}}$ has the property that all its best responses are optimal, i.e., $\optival(x_{\mathrm{IP}})=\pessval(x_{\mathrm{IP}})=v_{1}^{\star}$. Since $\optival(x_{\mathrm{IP}})=v_{1}^{\star}$ by the definition of $x_{\mathrm{IP}}$, it remains to show that $\pessval(x_{\mathrm{IP}})=v_{1}^{\star}$. Although $\pessval(x_{\mathrm{IP}})$ can be obtained by solving the problem in (\ref{eq:pessimistic problem}) in theory, the constraint may become infeasible due to finite numerical precision. Instead, given an optimal allocation $x^{\star}$ of the reward design problem in (\ref{eq:allocation_problem_om}), we solve the relaxed problem 
\begin{equation}
\begin{alignedat}{2} & \optmin_{m\in\mathcal{M}} & \quad & \iprod{r_{1}}{m}\\
 & \optst &  & \max_{m'\in\mathcal{M}}\iprod{r_{2}^{x^{\star}}}{m'}-\iprod{r_{2}^{x^{\star}}}{m}\leq\epsilon
\end{alignedat}
\label{eq: modified robust check}
\end{equation}
for different values of $\epsilon$. Denote the optimal value of (\ref{eq: modified robust check}) by $v_{\epsilon}(\xopt)$. If $\pessval(\xopt)=v_{1}^{\star}$, then $v_{\epsilon}(\xopt)$ can be made arbitrarily close to $v_{1}^{\star}$ by choosing $\epsilon$ sufficiently small. 
\begin{prop}
\label{prop: modified robust check}Let $x^{\star}$ be an optimal allocation. Then $\lim_{\epsilon\to0^{+}}v_{\epsilon}(x^{\star})=v_{1}^{\star}$ if and only if $x^{\star}$ is robust to nonunique best responses, i.e., $\optival(x^{\star})=\pessval(x_{\mathrm{}}^{\star})=v_{1}^{\star}$.
\end{prop}

\begin{proof}
See Appendix~\ref{subsec:Modified robust check}.
\end{proof}

\subsubsection{$6\times6$ grid world\label{subsec:6_6 grid world nonunique best responses}}

For the $6\times6$ grid world, the MILP defined in (\ref{MILP}) gives an optimal reward allocation $x_{\mathrm{MILP}}=(1.946,\ 1.774)$. The defender's expected payoff is $0.433$. In comparison, the maximum-margin optimal interior-point allocation given by (\ref{eq:prob_cc}) is $x_{\mathrm{IP}}=(2.122,\ 1.869)$. The defender's expected payoff under $x_{\mathrm{IP}}$ is the same as the MILP solution. The margin of the interior-point solution is $0.087$. 

We then computed $v_{\epsilon}(x_{\mathrm{MILP}})$ and $v_{\epsilon}(x_{\mathrm{IP}})$ under different values of $\epsilon$ by solving problem (\ref{eq: modified robust check}). The results are shown in Table~\ref{tab:pess and opt value} and Figure~\ref{fig:pess and opt value}. As $\epsilon$ approaches $0$, the value of $v_{\epsilon}(x_{\mathrm{IP}})$ also approaches $v_{1}^{\star}=0.433$. By Proposition~\ref{prop: modified robust check}, the results support Proposition~\ref{prop:robustness_nonunique} that the interior-point solution is robust to nonunique best responses. In comparison, $x_{\mathrm{MILP}}$ is not an interior-point allocation. The value of $v_{\epsilon}(x_{\mathrm{MILP}})$ appears to converge as $\epsilon$ approaches $0$, with the gap $v_{1}^{\star}-v_{\epsilon}(x_{\mathrm{MILP}})>0.03$ even when $\epsilon=10^{-10}$, showing that $x_{\mathrm{MILP}}$ is not robust to nonunique best responses.

\begin{table}
\begin{centering}
\begin{tabular}{|c|c|c|c|c|c|}
\hline 
$\epsilon$ & $10^{-3}$ & $10^{-4}$ & $10^{-5}$ & $10^{-6}$ & $10^{-7}$\tabularnewline
\hline 
$v_{\epsilon}(x_{\mathrm{MILP}})$ & $0.107$ & $0.195$ & $0.378$ & $0.396$ & $0.398$\tabularnewline
\hline 
$v_{\epsilon}(x_{\mathrm{IP}})$ & $0.379$ & $0.425$ & $0.432$ & $0.432$ & $0.433$\tabularnewline
\hline 
\end{tabular}
\par\end{centering}
\caption{$6 \times 6$ grid world case: The optimal values of (\ref{eq: modified robust check}) for the MILP solution and the interior-point solution under different values of $\epsilon$.\label{tab:pess and opt value}}

\end{table}

\begin{figure}
\centering{}\includegraphics[width=0.7\textwidth]{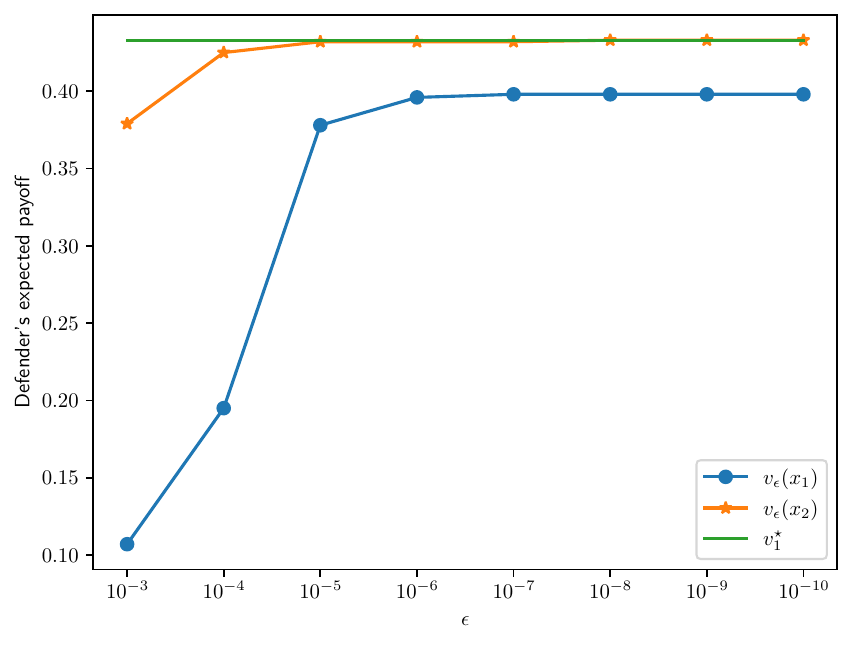}\caption{$6\times6$ grid world case: The optimal values of (\ref{eq: modified robust check}) for the MILP solution and the interior-point solution under different values of $\epsilon$. \label{fig:pess and opt value}}
\end{figure}

\subsubsection{$10\times10$ grid world\label{subsec:10*10-grid-world nonunique BR}}

For the $10\times10$ grid world, the optimal solution given by the MILP is $x_{\mathrm{MILP}}=(1.429,\ 0.000,\ 1.466)$, and the interior-point solution is $x_{\mathrm{IP}}=(1.836,\ 0.000,\ 2.164)$. %
The optimal expected payoffs of the defender are both $0.495$ given two reward allocations $x_{\mathrm{MILP}}$ and $x_{\mathrm{IP}}$. The margin of the interior-point solution is $0.061$.

The optimal values of (\ref{eq: modified robust check}) for $x_{\mathrm{MILP}}$ and $x_{\mathrm{IP}}$ under different $\epsilon$ are given in (\ref{tab:table of expected payoffs 10x10}). In this case, the interior-point solution remains robust to nonunique best responses according to the numerical results. The MILP solution may also be an optimal interior-point allocation, but it is difficult to conclude definitively due to finite numerical precision. If the defender's payoff finally approaches $0.495$ under $x_{\mathrm{MILP}}$, then $x_{\mathrm{MILP}}$ must be robust to nonunique best responses due to Proposition~\ref{prop: modified robust check}. 

\begin{table}
\begin{centering}
\begin{tabular}{|c|c|c|c|c|c|}
\hline 
$\epsilon$ & $10^{-3}$ & $10^{-4}$ & $10^{-5}$ & $10^{-6}$ & $10^{-6}$\tabularnewline
\hline 
$v_{\epsilon}(x_{\mathrm{MILP}})$ & $0.491$ & $0.494$ & $0.494$ & $0.494$ & $0.494$\tabularnewline
\hline 
$v_{\epsilon}(x_{\mathrm{IP}})$ & $0.494$ & $0.495$ & $0.495$ & $0.495$ & $0.495$\tabularnewline
\hline 
\end{tabular}
\par\end{centering}
\caption{$10\times10$ grid world case: The optimal values of (\ref{eq: modified robust check}) for the MILP solution and the interior-point solution under different values of $\epsilon$. \label{tab:table of expected payoffs 10x10}}
\end{table}

\subsubsection{Attack graph\label{subsec:Attack-graph-grid-world nonqunique BR}}

For the probabilistic attack graph, the MILP solution is $x_{\mathrm{MILP}}=(1.218,\ 0)$, and the interior-point solution is $x_{\mathrm{IP}}=(2.667,1.333)$ with a margin of $1.333$. Both solutions lead to a payoff of $0.655$ for the defender. Table~\ref{tab:table of expected payoffs attack graph} shows the optimal values of (\ref{eq: modified robust check}) for $x_{\mathrm{MILP}}$ and $x_{\mathrm{IP}}$ under different $\epsilon$. The numerical results imply that the interior-point solution is robust to nonunique best responses, whereas the MILP solution is not.

\begin{table}
\begin{centering}
\begin{tabular}{|c|c|c|c|c|c|}
\hline 
$\epsilon$ & $10^{-3}$ & $10^{-4}$ & $10^{-5}$ & $10^{-6}$ & $10^{-6}$\tabularnewline
\hline 
$v_{\epsilon}(x_{\mathrm{MILP}})$ & $0.247$ & $0.248$ & $0.248$ & $0.248$ & $0.248$\tabularnewline
\hline 
$v_{\epsilon}(x_{\mathrm{IP}})$ & $0.654$ & $0.654$ & $0.655$ & $0.655$ & $0.655$\tabularnewline
\hline 
\end{tabular}
\par\end{centering}
\caption{Attack graph case: The optimal values of (\ref{eq: modified robust check}) for the MILP solution and the interior-point solution under different values of $\epsilon$.  \label{tab:table of expected payoffs attack graph}}
\end{table}

\subsection{Robustness to a boundedly rational attacker}

For each environment introduced in Section~\ref{subsec:environments}, we computed the defender's payoff against a boundedly rational attacker. Our model of bounded rationality assumes that the attacker solves an entropy-regularized MDP in (\ref{eq:standard ent-regularized}). The parameter $\tau$ in (\ref{eq:standard ent-regularized}) reflects the level of rationality of the attacker: The smaller the value of $\tau$, the more rational the attacker. 


For the $6\times6$ grid world, Table~\ref{tab:Table-of-expected-payoffs} gives the expected payoffs of the defender when facing attackers with different levels of rationality for the interior-point solution $x_\mathrm{IP}$ (Row~3) and the MILP solution $x_\mathrm{MILP}$ (Row~4), respectively. From Proposition~\ref{prop:robustness of ent-regularized-1}, when the interior-point solution is used for allocation, a lower bound for the defender's payoff should approach $v_{1}^{\star} = \iprod{r_1}{\mopt}$ as $\tau \to 0$. Recall from Section~\ref{subsec:6_6 grid world nonunique best responses} that the optimal expected payoff against a rational attacker is given by $v_{1}^{\star}=0.433$. It can be seen that the theoretical result given by Proposition~\ref{prop:robustness of ent-regularized-1} is consistent with the numerical results in Row~3 of Table~\ref{tab:Table-of-expected-payoffs}. In comparison, the MILP solution does not enjoy the same payoff guarantee. The payoff for the MILP solution remains noticeably below $v_{1}^{\star}$ even when $\tau$ is as small as $10^{-5}$. In addition, the MILP solution performs consistently worse than the interior-point solution for all values of $\tau$. 

\begin{table}
\centering{}%
\begin{tabular}{|c|c|c|c|c|c|}
\hline 
$\tau$ & $10^{-1}$ & $10^{-2}$ & $10^{-3}$ & $10^{-4}$ & $10^{-5}$\tabularnewline
\hline 
Optimal solution & $1.31\times10^{-3}$ & $0.429$ & $0.433$ & $0.433$ & $0.433$\tabularnewline
\hline 
Interior-point solution & $1.29\times10^{-3}$ & $0.429$ & $0.433$ & $0.433$ & $0.433$\tabularnewline
\hline 
MILP solution & $8.63\times10^{-4}$ & $0.335$ & $0.324$ & $0.385$ & $0.414$\tabularnewline
\hline 
\end{tabular}\caption{$6\times6$ grid world case: Expected payoffs of the defender for the interior-point solution and the MILP solution when facing attackers with different levels of rationality. The optimal payoff against a rational attacker ($\tau = 0$) is given by $v_{1}^{\star} = 0.433$.\label{tab:Table-of-expected-payoffs}}
\end{table}

For comparison, Row~2 of Table~\ref{tab:Table-of-expected-payoffs} shows the optimal expected payoffs of the defender against a boundedly rational attacker, where the defender is assumed to know the actual value of $\tau$ and choose the reward allocation accordingly. The expected payoff of the defender for the interior-point solution is found to be close to the optimal payoff. In other words, the interior-point solution leads to a near-optimal payoff even when the defender incorrectly assumes a rational attacker. 

Similar relationships between the payoffs were observed for the $10\times10$ grid world and the probabilistic attack graph. The numerical results are shown in Table~\ref{tab:Table-of-expected-payoffs 10*10} and Table~\ref{tab:Table-of-expected-payoffs attack graph}.

\begin{table}
\centering{}%
\begin{tabular}{|c|c|c|c|c|c|}
\hline 
$\tau$ & $10^{-1}$ & $10^{-2}$ & $10^{-3}$ & $10^{-4}$ & $10^{-5}$\tabularnewline
\hline 
Optimal solution & $0.194$ & $0.489$ & $0.495$ & $0.495$ & $0.495$\tabularnewline
\hline 
Interior point solution & $1.21\times10^{-3}$ & $0.489$ & $0.495$ & $0.495$ & $0.495$\tabularnewline
\hline 
MILP solution & $1.78\times10^{-4}$ & $0.484$ & $0.495$ & $0.495$ & $0.495$\tabularnewline
\hline 
\end{tabular}\caption{$10\times10$ grid world case: Expected payoffs of the defender for the interior-point solution and the MILP solution when facing attackers with different levels of rationality. The optimal payoff against a rational attacker ($\tau = 0$) is given by $v_{1}^{\star} = 0.495$.\label{tab:Table-of-expected-payoffs 10*10}}
\end{table}

\begin{table}
\centering{}%
\begin{tabular}{|c|c|c|c|c|c|}
\hline 
$\tau$ & $10^{-1}$ & $10^{-2}$ & $10^{-3}$ & $10^{-4}$ & $10^{-5}$\tabularnewline
\hline 
Optimal solution & $0.620$ & $0.655$ & $0.655$ & $0.655$ & $0.655$\tabularnewline
\hline 
Interior-point solution & $0.583$ & $0.655$ & $0.655$ & $0.655$ & $0.655$\tabularnewline
\hline 
MILP solution & $0.407$ & $0.446$ & $0.446$ & $0.446$ & $0.500$\tabularnewline
\hline 
\end{tabular}\caption{Attack graph case: Expected payoffs of the defender for the interior-point solution and the MILP solution when facing attackers with different levels of rationality. The optimal payoff against a rational attacker ($\tau = 0$) is given by $v_{1}^{\star} = 0.655$.\label{tab:Table-of-expected-payoffs attack graph}}
\end{table}

\subsection{Scalability}
 This subsection discusses the scalability of solving the optimization problem in (\ref{eq:prob_cc}). Instead of solving (\ref{eq:prob_cc}) directly, we obtained the optimal value of (\ref{eq:prob_cc}) via the bisection method. Each step of the bisection checks whether a proposed margin $c$ is feasible. The initial upper bound for $c$ is given by $C$, the total budget of allocation, and the initial lower bound for $c$ is $0$. For both environments, the bisection procedure typically ended after $10$ to $20$ steps. Each bisection step ranged from $0.2$ to $354$~seconds for the $6 \times 6$ grid world and from $1.9$ to $85$~seconds for the $10 \times 10$ grid world. We suspect that the longer running times for the $6 \times 6$ grid world were instance-dependent; the MILP in \eqref{eq:prob_cc} for the $6 \times 6$ grid world was possibly a harder problem instance than the one for the $10 \times 10$ grid world. The running time was also found to depend on whether the proposed margin is feasible. Longer running times were observed when the proposed margin was infeasible, especially when the proposed margin was close to feasibility.

%% file: conclusions.tex
We study the problem of reward design for Markov decision processes in a leader-follower setup, where the leader is tasked with modifying the follower's reward function to induce a policy favorable to the leader. Existing methods for reward design typically rely on exact knowledge of the follower's reward function and can be sensitive to modeling errors. Motivated by the issue of sensitivity, we present a new method of reward design with robustness guarantees when the follower's reward function is unknown. Our method can be viewed as a formal algorithmic counterpart of the folklore solution for achieving robustness based on perturbing the leader's strategy. The reward modification in our method is based on the concept of \emph{optimal interior-point allocation}. We show that an optimal interior-point allocation provably offers robustness to three types of uncertainties, including 1) how the follower breaks ties in the presence of nonunique best responses, 2) inexact knowledge of how the follower perceives reward modifications, and 3) bounded rationality of the follower. 

 We also study the existence of an optimal interior-point allocation when the leader's reward function adopts a special form. One complication in finding an optimal interior-point allocation is that the corresponding optimal allocation region cannot be predetermined because some may have an empty interior. Nevertheless, our result shows that it suffices to focus on optimal allocation regions of deterministic occupancy measures. This characterization of the optimal allocation regions leads to two sufficient and necessary conditions for the existence of an optimal interior-point allocation. One examines whether the allocation budget needs to be exhausted to achieve the optimal leader's payoff, which can be verified numerically. The other requires the existence of an allocation robust to nonunique best responses of the follower, which establishes the central role that optimal interior-point allocations play in achieving robustness. When the leader uses a general reward function, most results can be generalized by replacing the optimal interior-point allocation with the \emph{optimal interior-point reward function}. However, a sufficient and necessary condition for the existence of an optimal interior-point reward function remains an open question. 

We further show that an optimal interior-point allocation can be found by mixed-integer linear programming if such an allocation exists. In fact, the mixed-integer linear program can be used to find an optimal interior-point allocation with the largest margin. Numerical experiments have been conducted in several simulation environments inspired by problems in cybersecurity. The experiments not only validate the theoretical guarantees on robustness but also show that our method scales to problems of practical size.

%% file: appendix.tex
\section{Basic results about MDPs}
\subsection{Policies and occupancy measures}

\subsubsection{Proof of Proposition~\ref{prop:policies induce occupancy measure}\label{subsec:Proof of polices induce occupancy measure}}
\begin{proof}
By definition
\begin{align*}
m(s,a) & =\mathbb{E}_{\pi,s_{0}\sim\rho}\left[\sum_{t=0}^{\infty}\gamma^{t}\mathbb{P}(s_{t}=s,a_{t}=a)\right]\\
 & =\mathbb{E}_{\pi,s_{0}\sim\rho}\left[\sum_{t=0}^{\infty}\gamma^{t}\mathbb{P}(s_{t}=s)\cdot\pi(s,a)\right]\\
 & =\pi(s,a)\cdot\mathbb{E}_{\pi,s_{0}\sim\rho}\left[\sum_{t=0}^{\infty}\gamma^{t}\mathbb{P}(s_{t}=s)\right]\\
 & =\pi(s,a)\cdot\mathbb{E}_{\pi,s_{0}\sim\rho}\left[\sum_{t=0}^{\infty}\sum_{a\in\mathcal{A}}\gamma^{t}\mathbb{P}(s_{t}=s,a_{t}=a)\right]\\
 & =\pi(s,a)\cdot\sum_{a\in\mathcal{A}}\mathbb{E}_{\pi,s_{0}\sim\rho}\left[\sum_{t=0}^{\infty}\gamma^{t}\mathbb{P}(s_{t}=s,a_{t}=a)\right]\\
 & =\pi(s,a)\cdot\sum_{a\in\mathcal{A}}m(s,a).
\end{align*}
The exchangeability of the summation is due to the absolute convergence of $\sum_{t=0}^{\infty}\gamma^{t}\mathbb{P}(s_{t}=s,a_{t}=a)$. If $\sum_{a}m(s,a)\neq0$, the definition requires that $\pi(s,a)=\frac{m(s,a)}{\sum_{a}m(s,a)}$. Otherwise, $\pi(s,\cdot)$ can be an arbitrary probability distribution over $\mathcal{A}$.
\end{proof}

\subsubsection{Occupancy measures induced by deterministic policies\label{subsec:proof of induce same set of occupancy measure}\label{subsec: proof of Lem:Dterministic}}
\begin{prop}
\label{prop:Induce same set of occupancy measure-1}For any $\pi_{1},\pi_{2}\in\Pi(m)$, it holds that $\Mdet(\pi_{1})=\Mdet(\pi_{2})$. 
\end{prop}

\begin{proof}
Let $\mathcal{S}_{0}=\{s\in\mathcal{S}\mid\sum_{a\in\mathcal{A}}m(s,a)=0\}$. Observe that for any $\pi_{1},\pi_{2}\in\Pi(m)$, 
\begin{equation}
\pi_{1}(s,a)=\pi_{2}(s,a)=\frac{m(s,a)}{\sum_{a}m(s,a)}\quad \forall s\in\mathcal{S}\setminus\mathcal{S}_{0}.\label{eq:pi equal}
\end{equation}
Suppose a policy $\pidet^{1}\in\Pidet(\pi_{1})$ induces the occupancy measure $\mdet^{1}\in\Mdet(\pi_{1})$. Consider another policy $\pidet^{2}$ defined by
\begin{equation}
\pidet^{2}(s,\cdot)\triangleq\begin{cases}
\pidet^{1}(s,\cdot) & \text{if }s\in\mathcal{S}\setminus\mathcal{S}_{0},\\
\pidet(s,\cdot) & \text{if }s\in\mathcal{S}_{0},
\end{cases}\label{eq:def of pi2}
\end{equation}
where $\pidet$ is an arbitrary deterministic policy in $\Pidet(\pi_{2})$. Since $\pidet^{1}$ and $\pidet$ are both deterministic policies, it is not hard to see that $\pidet^{2}$ is also a deterministic policy. 

We shall first show that $\pidet^{2}\in\Pidet(\pi_{2})$. By definition, it suffices to show that $\pidet^{2}(s,a)=0$ when $\pi_{2}(s,a)=0$. When $\pi_{2}(s,a)=0$ and $s\in\mathcal{S}_{0}$, it holds that $\pidet(s,a)=0$ by the definition of $\Pidet(\pi_{2})$, which implies that $\pidet^{2}(s,a)=0$. On the other hand, when $\pi_{2}(s,a)=0$ and $s\in\mathcal{S}\setminus\mathcal{S}_{0}$, it follows from \eqref{eq:pi equal} that $\pi_{1}(s,a)=0$. From the definition of $\Pidet(\pi_{1})$, we know $\pidet^{1}(s,a)=0$, which implies that $\pidet^{2}(s,a)=\pidet^{1}(s,a)=0$ according to \eqref{eq:def of pi2}. 

Meanwhile, since $\pidet^{2}(s,\cdot)=\pidet^{1}(s,\cdot)$ when $s\in\mathcal{S}\setminus\mathcal{S}_{0}$, it follows from Proposition~\ref{prop:policies induce occupancy measure} that the occupancy measures induced by $\pidet^{2}$ and $\pidet^{1}$ are equal. Because $\pidet^{1}$ is selected arbitrarily from $\Pidet(\pi_{1})$, it follows that $\Mdet(\pi_{1})\subseteq\Mdet(\pi_{2})$. The same procedure can be used to show $\Mdet(\pi_{2})\subseteq\Mdet(\pi_{1})$, which completes the proof.
\end{proof}
\begin{lem}
\label{lem:Deterministic occupancy measures are verticies}$\Mdet$ consists of all the vertices of $\cM$.
\end{lem}

Before the proof, recall that for a polyhedron $P$, a point $x \in P$ is an extreme point if one cannot find two points $y,z\in P$ that are different from $x$, and a constant $\alpha \in (0,1)$ such that $x=\alpha y+(1-\alpha)z$. A point $x\in P$ is a vertex if there exists a vector $c$ such that $\iprod{c}{x}>\iprod{c}{y}$ for all $y\in P \setminus \{x\}$. A polytope is a bounded polyhedron. It is well known that for a polytope, a vertex is equivalent to an extreme point. 
\begin{proof}
Consider any $\mdet\in\Mdet$. We shall first show that $\mdet$ cannot be expressed as any convex combination of two other occupancy measures and hence is an extreme point of $\cM$. Assume, for the sake of contradiction, that 
\[
\mdet=\alpha m_{1}+(1-\alpha)m_{2}
\]
for some $m_{1},m_{2}\in\mathcal{M}\setminus\{\mdet\}$, where $m_{1}\neq m_{2}$, and $\alpha\in(0,1)$. Denote by $\pidet$, $\pi_{1}$, and $\pi_{2}$ any policies that induce $\mdet$, $m_{1}$, and $m_{2}$, respectively. Consider any $s\in\mathcal{S}$ such that $\sum_{a\in\mathcal{A}}\mdet(s,a)\neq0$. Because $\pidet$ is a deterministic policy, there exists a unique action $a^{\star}\in\mathcal{A}$ such that $\pidet(s,a^{\star})=1$ and $\pidet(s,a)=0$ for all $a\in\mathcal{A}\setminus\{a^{\star}\}$. Consider any $a\in\mathcal{A}\setminus\{a^{\star}\}$. Since $\pidet(s,a)=\mdet(s,a)/\sum_{a}\mdet(s,a)$, it follows that $\mdet(s,a)=0$. Because $m_{1},m_{2}\succeq0$, this implies $m_{1}(s,a)=m_{2}(s,a)=0$. It follows that $\pi_{1}$ and $\pi_{2}$ satisfy $\pi_{1}(s,a)=\pi_{2}(s,a)=0$. Since $\sum_{a'\in\mathcal{A}}\pi_{i}(s,a')=1$ for $i\in\{1,2\}$, it must hold that $\pi_{1}(s,a^{\star})=\pi_{2}(s,a^{\star})=1$, which implies that $\pi_{1}(s,a)=\pi_{2}(s,a)=\pidet(s,a)$ for all $a\in\mathcal{A}$. Since $s\in\mathcal{S}$ is arbitrarily selected, we know that $\pi_{1}(s,a)=\pi_{2}(s,a)=\pidet(s,a)$ holds for all $(s,a)$ such that $\sum_{a\in\mathcal{A}}\mdet(s,a)\neq0$. From Proposition~\ref{prop:policies induce occupancy measure}, the two policies $\pi_{1},\pi_{2}$ induce the same occupancy measures as $\pidet$, i.e., $m_{1}=m_{2}=\mdet$, which leads to a contradiction.

Suppose some $m\notin\Mdet$ is a vertex of $\mathcal{M}$. From the definition of vertices, there is a vector $y$, viewed as a reward function, such that $m$ is the unique occupancy measure that achieves the maximum expected reward. However, this is impossible since for any reward function in an MDP, there must exist a deterministic optimal policy.
\end{proof}

\subsection{Expected payoff and occupancy measure\label{subsec:Proof-of expected reward }}
\begin{lem}
Let $x$ be a reward allocation. Suppose the follower uses a policy $\pi$, which induces an occupancy measure $m$. The expected payoff of the follower satisfies $\mathbb{E}_{s\sim\rho}\left[V_{2}^{\pi}(s;x)\right]=\langle r_{2}^{x},m\rangle$.
\label{Expected-accumulated-reward}
\end{lem}
\begin{proof}
The expected cumulative reward of the follower satisfies 

\begin{align*}
\mathbb{E}_{s_{0}\sim\rho}\left[V_{2}^{\pi}(s)\right] & =\mathbb{E}_{\pi,s_{0}\sim\rho}\left[\sum_{t=0}^{\infty}\gamma^{t}r_{2}^{x}(s_{t},a_{t})\right]\\
 & =\mathbb{E}_{\pi,s_{0}\sim\rho}\left[\sum_{t=0}^{\infty}\sum_{(s,a)\in\mathcal{S}\times\mathcal{A}}\gamma^{t}\mathbb{P}(s_{t}=s,a_{t}=a)r_{2}^{x}(s,a)\right]\\
 & =\sum_{(s,a)\in\mathcal{S}\times\mathcal{A}}\mathbb{E}_{\pi,s_{0}\sim\rho}\left[\sum_{t=0}^{\infty}\gamma^{t}\mathbb{P}(s_{t}=s,a_{t}=a)\right]r_{2}^{x}(s,a)\\
 & =\sum_{(s,a)\in\mathcal{S}\times\mathcal{A}}m(s,a)r_{2}^{x}(s,a).
\end{align*}
The exchangeability of the summation is from the absolute convergence of $\sum_{t=0}^{\infty}\gamma^{t}\mathbb{P}(s_{t}=s,a_{t}=a)$.
\end{proof}

\subsection{Useful properties of occupancy measures}
\begin{lem}
\label{lem: sum_of_oc}Let $m$ be an occupancy measure. Then $\sum_{(s,a)\in\mathcal{S}\times\mathcal{A}}m(s,a)=1/(1-\gamma)$.
\end{lem}

\begin{proof}
By the definition of occupancy measures,
\[
m(s,a)=\mathbb{E}_{s_{0}\sim\rho}\left[\sum_{t=0}^{\infty}\gamma^{t}\mathbb{P}(s_{t}=s,a_{t}=a\mid s_{0})\right].
\]
Sum over the states and actions to obtain
\begin{align*}
\sum_{(s,a)\in\mathcal{S}\times\mathcal{A}}m(s,a) & =\sum_{(s,a)\in\mathcal{S}\times\mathcal{A}}\mathbb{E}_{s_{0}\sim\rho}\left[\sum_{t=0}^{\infty}\gamma^{t}\mathbb{P}(s_{t}=s,a_{t}=a\mid s_{0})\right]\\
 & =\mathbb{E}_{s_{0}\sim\rho}\left[\sum_{(s,a)\in\mathcal{S}\times\mathcal{A}}\sum_{t=0}^{\infty}\gamma^{t}\mathbb{P}(s_{t}=s,a_{t}=a\mid s_{0})\right]\\
 & =\mathbb{E}_{s_{0}\sim\rho}\left[\sum_{t=0}^{\infty}\gamma^{t}\sum_{(s,a)\in\mathcal{S}\times\mathcal{A}}\mathbb{P}(s_{t}=s,a_{t}=a\mid s_{0})\right]\\
 & =\mathbb{E}_{s_{0}\sim\rho}\left[\sum_{t=0}^{\infty}\gamma^{t}\right]\\
 & =\frac{1}{1-\gamma}.
\end{align*}
It was possible to change the order of summation because $\sum_{t=0}^{\infty}\gamma^{t}\mathbb{P}(s_{t}=s,a_{t}=a\mid s_{0})\leq1/(1-\gamma)$, which is an absolutely convergent series. 
\end{proof}

\section{Proofs on the robustness to nonunique best response}

\subsection{Proof of Proposition~\ref{prop:robustness_nonunique}\label{subsec: proof of Robustness 1}}
\begin{lem}
\label{lem:OC difference}Let $v$ be a vector in $\mathbb{R}^{\vert\mathcal{S}_{d}\vert}$. Suppose $\bar{x}$ is an interior point of an allocation region $\mathcal{P}_{m}$, i.e., $\bar{x}+cv\in\mathcal{P}_{m}$ for some constant $c>0$ and for all $\Vert v\Vert_{1}\leq1$. For any $m'\in\mathcal{M}$, it holds that
\[
\iprod{r_{2}^{\bar{x}}}{m-m'}\geq c\cdot\max_{s\in\mathcal{S}_{d}}\left|\sum_{a\in\cA}\left(m(s,a)-m'(s,a)\right)\right|.
\]
\end{lem}

\begin{proof}
Consider any $m'\in\mathcal{M}$. It follows from the definition of $\cP_m$ that
\begin{align*}
0 & \leq\min_{\|v\|_{1}=1}\left\{ \iprod{r_{2}^{\bar{x}+cv}}{m}-\iprod{r_{2}^{\bar{x}+cv}}{m'}\right\} \\
 & =\min_{\|v\|_{1}=1}\iprod{r_{2}^{\bar{x}+cv}}{m-m'}\\
 & =\min_{\|v\|_{1}=1}\left\{ \iprod{r_{2}^{\bar{x}}}{m-m'}+c\iprod{\delta^{v}}{m-m'}\right\} \\
 & =\iprod{r_{2}^{\bar{x}}}{m-m'}+c\min_{\|v\|_{1}=1}\iprod{\delta^{v}}{m-m'},
\end{align*}
where $\delta^{v}$ satisfies $\delta^{v}(s,a)=v(s)$ for all $s\in\mathcal{S}_{d}$ and $\delta^{v}(s,a)=0$ otherwise. Therefore, we have
\begin{align*}
\iprod{r_{2}^{\bar{x}}}{m-m'} & \geq-c\min_{\norm{v}_{1}=1}\iprod{\delta^{v}}{m-m'}\\
 & =c\max_{\norm{v}_{1}=1}\iprod{\delta^{v}}{m'-m}\\
 & \geq c\cdot\max_{s\in\mathcal{S}_{d}}\left|\sum_{a\in\cA}\left(m(s,a)-m'(s,a)\right)\right|,
\end{align*}
which completes the proof.
\end{proof}
\begin{proof}[Proof of Proposition~\ref{prop:robustness_nonunique}]
Since $x^{\star}$ is an optimal interior-point allocation, there exists an optimal occupancy measure $m^{\star}$ such that $\xopt$ is an interior point of $\cP_{\mopt}$. By Lemma~\ref{lem:OC difference}, there exists $c>0$ such that
\[
\iprod{r_{2}^{x^{\star}}}{m^{\star}-m}\geq c \cdot \max_{s\in\mathcal{S}_{d}}\left|\sum_{a\in\mathcal{A}}\left(m(s,a)-m^{\star}(s,a)\right)\right|\quad\forall m\in\mathcal{M}.
\]
When $m\in\BR(x^{\star})$, it holds that $\iprod{r_{2}^{x^{\star}}}{m}=\iprod{r_{2}^{x^{\star}}}{m^{\star}}$, i.e., $\iprod{r_{2}^{x^{\star}}}{m^{\star}-m}=0$. We obtain $0\geq c \cdot \max_{s\in\mathcal{S}_{d}}\left|\sum_{a\in\mathcal{A}}\left(m(s,a)-m^{\star}(s,a)\right)\right|.$ This implies 
\[
\sum_{a\in\mathcal{A}}\left(m(s,a)-m^{\star}(s,a)\right)=0\quad\forall s\in\mathcal{S}_{d}.
\]
By the definition of $r_{1}$, it holds that $\iprod{r_{1}}{m-m^{\star}}=\sum_{s\in\mathcal{S}_{d}}\sum_{a\in\mathcal{A}}(m(s,a)-m^{\star}(s,a))=0$, implying $\iprod{r_{1}}{m}=\iprod{r_{1}}{m^{\star}}$.
\end{proof}

\subsection{Proof of Proposition~\ref{prop: modified robust check}\label{subsec:Modified robust check}}
\begin{proof}
We first show the \emph{if }direction. Suppose that $x^{\star}$ satisfies $\iprod{r_{1}}{m}=v_{1}^{\star}$ for all $m\in\BR(x^{\star})$. Consider any $\epsilon>0$ and $m$ satisfying $\iprod{r_{2}^{x^{\star}}}{m^{\star}}-\iprod{r_{2}^{x^{\star}}}{m}\leq\epsilon$. Let $b=\langle r_{2}^{x^{\star}},m^{\star}\rangle-\max_{m\in\Mdet\setminus\BR(\xopt)}\langle r_{2}^{x^{\star}},m\rangle$ and $\Mdet^{\star}=\Mdet\cap\BR(x^{\star})$. By Lemma~\ref{lem:Deterministic occupancy measures are verticies}, we can express $m$ as a convex combination of deterministic occupancy measures: $m=\sum_{i=1}^{h}\lambda_{i}m_{i}+\sum_{j=1}^{l}\lambda_{j}m_{j}'$, where $m_{i}\in\Mdet^{\star}$ for $i=1,2,\ldots,h$, $m_{j}'\in\Mdet\setminus\Mdet^{\star}$ for $j=1,2,\ldots,l$, $\lambda_{i},\lambda_{j}\geq0$ and $\sum_{i=1}^{h}\lambda_{i}+\sum_{j=1}^{l}\lambda_{j}=1$. For any $i \in \{1,2,\dots,h\}$, because $m_{i}\in\Mdet^{\star}$, it holds that $\iprod{r_{2}^{x^{\star}}}{m_{i}}=\iprod{r_{2}^{x^{\star}}}{m^{\star}}$. For any $j=\{1,2,\dots,l\}$, because $m_{j}'\in\Mdet\setminus\Mdet^{\star}=\Mdet\setminus\BR(\xopt)$, it holds that $\iprod{r_{2}^{x^{\star}}}{m^{\star}}-\iprod{r_{2}^{x^{\star}}}{m_{j}'}\geq b$. Therefore,
\begin{align*}
\iprod{r_{2}^{x^{\star}}}{m^{\star}}-\iprod{r_{2}^{x^{\star}}}{m} & =\iprod{r_{2}^{x^{\star}}}{m^{\star}}-\iprod{r_{2}^{x^{\star}}}{\sum_{i=1}^{h}\lambda_{i}m_{i}}-\iprod{r_{2}^{x^{\star}}}{\sum_{j=1}^{l}\lambda_{j}m_{j}'}\\
 & =\iprod{r_{2}^{x^{\star}}}{m^{\star}}-\sum_{i=1}^{h}\lambda_{i}\iprod{r_{2}^{x^{\star}}}{m_{i}}-\sum_{j=1}^{l}\lambda_{j}\iprod{r_{2}^{x^{\star}}}{m_{j}'}\\
 & =\sum_{i=1}^{h}\lambda_{i}\iprod{r_{2}^{x^{\star}}}{m^{\star}-m_{i}}+\sum_{j=1}^{l}\lambda_{j}\iprod{r_{2}^{x^{\star}}}{m^{\star}-m_{j}'}\\
 & =\sum_{j=1}^{l}\lambda_{j}\iprod{r_{2}^{x^{\star}}}{m^{\star}-m_{j}'}\\
 & \geq b\cdot\sum_{j=1}^{l}\lambda_{j}.
\end{align*}
This implies that $\epsilon\geq b\cdot\sum_{j=1}^{l}\lambda_{j}$, i.e., $\sum_{j=1}^{l}\lambda_{j}\leq \epsilon/b$. In the meantime, because $m_{i} \in \in\Mdet^{\star}$, it follows from the given assumption that $\iprod{r_{1}}{m_{i}}=v_{1}^{\star}=\iprod{r_{1}}{m^{\star}}$. Thus,
\begin{align*}
\iprod{r_{1}}{m^{\star}}-\iprod{r_{1}}{m} & =\iprod{r_{1}}{m^{\star}}-\iprod{r_{1}}{\sum_{i=1}^{h}\lambda_{i}m_{i}}-\iprod{r_{1}}{\sum_{j=1}^{l}\lambda_{j}m_{j}'}\\
 & =\sum_{i=1}^{h}\lambda_{i}\iprod{r_{1}}{m^{\star}-m_{i}}+\sum_{j=1}^{l}\lambda_{j}\iprod{r_{1}}{m^{\star}-m_{j}'}\\
 & =\sum_{j=1}^{l}\lambda_{j}\iprod{r_{1}}{m^{\star}-m_{j}'}.
\end{align*}
When $\epsilon \to 0^{+}$, it follows from $\sum_{j=1}^{l}\lambda_{j}\leq \epsilon/b$ that $\sum_{j=1}^{l}\lambda_{j} \to 0^{+}$, which further implies that $\iprod{r_{1}}{m^{\star}}-\iprod{r_{1}}{m} \to 0^{+}$. Since $m$ is selected arbitrarily such that $\iprod{r_{2}^{x^{\star}}}{m^{\star}}-\iprod{r_{2}^{x^{\star}}}{m}\leq\epsilon$, this implies that $v_{1}^{\star}-v_{\epsilon}(x^{\star})\rightarrow0^{+}$ when $\epsilon\rightarrow0^{+}$.

For the \emph{only if }direction, we assume for the sake of contradiction that there exists $m\in\BR(x^{\star})$ such that $\iprod{r_{1}}{m^{\star}}-\iprod{r_{1}}{m}=c>0$. Then $m$ is a feasible solution to (\ref{eq: modified robust check}), which implies $v_{1}^{\star}-v_{\epsilon}(x^{\star})\geq c$ for any $\epsilon>0$. This contradicts the given assumption that $v_{1}^{\star}-v_{\epsilon}(x^{\star})\rightarrow0^{+}$ when $\epsilon\rightarrow0^{+}$.
\end{proof}

\section{Proofs on the robustness to bounded rationality}

\subsection{Model of bounded rationality\label{subsec:Transformation-of-Entropy-Regulared}}


In the original entropy-regularized MDP, the follower chooses a policy $\pi$ to maximize
\[
V_{\tau}^{\pi}(\rho)\triangleq V^{\pi}(\rho)+\tau\cdot\mathcal{H}(\rho,\pi),
\]
where $V^{\pi}(\rho)=\mathbb{E}_{s\sim\rho}\left[V_{2}^{\pi}(s)\right]$ and
\[
\mathcal{H}(\rho,\pi) \triangleq\mathbb{E}_{\pi, s_{0}\sim\rho}\left[\sum_{t=0}^{\infty}-\gamma^{t}\log\pi(s_{t},a_{t})\right].
\]
Let $m$ be the occupancy measure induced by $\pi$. It then holds that
\[
\mathcal{H}(\rho,\pi) = -\sum_{(s,a)\in\mathcal{S}\times\mathcal{A}}m(s,a)\log\pi(s,a)=-\sum_{(s,a)\in\mathcal{S}\times\mathcal{A}}m(s,a)\log\frac{m(s,a)}{\sum_{a'\in\mathcal{A}}m(s,a')},
\]
where the last equality follows from Proposition~\ref{prop:policies induce occupancy measure}.
Combine Lemma~\ref{Expected-accumulated-reward} to obtain the payoff function in~\eqref{eq:standard ent-regularized}.

\begin{proof}[Proof of Proposition~\ref{prop:robustness of ent-regularized-1}]
Notice that 
\begin{align*}
 & -\sum_{(s,a)\in\mathcal{S}\times\mathcal{A}}\left(m^{\star}(s,a)\log\frac{m^{\star}(s,a)}{\sum_{a'\in\mathcal{A}}m^{\star}(s,a')}\right)\\
 & \qquad = -\sum_{(s,a)\in\mathcal{S}\times\mathcal{A}}\left(\left(\sum_{a\in\mathcal{A}}m^{\star}(s,a)\right)\left(\frac{m^{\star}(s,a)}{\sum_{a'\in\mathcal{A}}m^{\star}(s,a')}\right)\log\frac{m^{\star}(s,a)}{\sum_{a'\in\mathcal{A}}m^{\star}(s,a')}\right)\\
& \qquad = -\sum_{s\in\mathcal{S}}\left(\vert\mathcal{A}\vert\sum_{a\in\mathcal{A}}m^{\star}(s,a)\right)\sum_{a\in\mathcal{A}}\left(\left(\frac{m^{\star}(s,a)}{\sum_{a'\in\mathcal{A}}m^{\star}(s,a')}\right)\log\frac{m^{\star}(s,a)}{\sum_{a'\in\mathcal{A}}m^{\star}(s,a')}\right)\\
& \qquad \leq \sum_{s\in\mathcal{S}}\left(\vert\mathcal{A}\vert\sum_{a\in\mathcal{A}}m^{\star}(s,a)\cdot\frac{1}{\vert\mathcal{A}\vert}\log\vert\mathcal{A}\vert\right)\\
& \qquad = \sum_{s\in\mathcal{S}}\sum_{a\in\mathcal{A}}m^{\star}(s,a)\log\vert\mathcal{A}\vert\\
& \qquad = \frac{1}{1-\gamma}\log\vert\mathcal{A}\vert.
\end{align*}
This implies
\begin{equation}
\begin{aligned}
\tau\cdot\left(\sum_{(s,a)\in\mathcal{S}\times\mathcal{A}}m^{\star}(s,a)\log\frac{m^{\star}(s,a)}{\sum_{a'\in\mathcal{A}}m^{\star}(s,a')}-\sum_{(s,a)\in\mathcal{S}\times\mathcal{A}}m_{\tau}^{\star}(s,a)\log\frac{m_{\tau}^{\star}(s,a)}{\sum_{a'\in\mathcal{A}}m_{\tau}^{\star}(s,a')}\right)\\
\leq2\tau\cdot\frac{1}{1-\gamma}\cdot\log\vert\mathcal{A}\vert.
\end{aligned}
\label{eq: inequality of entropy regularized part 1}
\end{equation}
It follows from the optimality of $m_{\tau}^{\star}$ that
\begin{equation}
\begin{aligned}
\tau\cdot\left(\sum_{(s,a)\in\mathcal{S}\times\mathcal{A}}m^{\star}(s,a)\log\frac{m^{\star}(s,a)}{\sum_{a'\in\mathcal{A}}m^{\star}(s,a')}-\sum_{(s,a)\in\mathcal{S}\times\mathcal{A}}m_{\tau}^{\star}(s,a)\log\frac{m_{\tau}^{\star}(s,a)}{\sum_{a'\in\mathcal{A}}m_{\tau}^{\star}(s,a')}\right)\\
\geq\iprod{r_{2}^{x^{\star}}}{m^{\star}}-\iprod{r_{2}^{x^{\star}}}{m_{\tau}^{\star}}.
\end{aligned}
\label{eq: inequality of entropy regularized part 2}
\end{equation}
Combine \eqref{eq: inequality of entropy regularized part 1} and \eqref{eq: inequality of entropy regularized part 2} to obtain
\begin{equation}
\iprod{r_{2}^{x^{\star}}}{m^{\star}}-\iprod{r_{2}^{x^{\star}}}{m_{\tau}^{\star}}\leq2\tau\cdot\frac{1}{1-\gamma}\cdot\log\vert\mathcal{A}\vert.
\label{eq: inequality of entropy regularized part 3}
\end{equation}

Let $h = \vert\Mdet^{\star}\vert$ and $l = \vert\Mdet\setminus\Mdet^{\star}\vert$. Since $m_{\tau}^{\star} \in \cM$, according to Lemma~\ref{lem:Deterministic occupancy measures are verticies}, we can write $m_{\tau}^{\star}=\sum_{i=1}^{h}\lambda_{i}m_{i}+\sum_{j=1}^{l}\lambda_{j}m_{j}'$, where $m_{i}\in\Mdet^{\star}$ for $i=1,2,\ldots,h$, $m_{j}'\in\Mdet\setminus\Mdet^{\star}$ for $j=1,2,\ldots,l$, $\lambda_{i},\lambda_{j}\geq0$ and $\sum_{i=1}^{h}\lambda_{i}+\sum_{j=1}^{l}\lambda_{j}=1$. We can rewrite the left side of (\ref{eq: inequality of entropy regularized part 3}) as 
\[
\begin{aligned}\iprod{r_{2}^{x^{\star}}}{m^{\star}}-\iprod{r_{2}^{x^{\star}}}{m_{\tau}^{\star}} & =\iprod{r_{2}^{x^{\star}}}{m^{\star}}-\iprod{r_{2}^{x^{\star}}}{\sum_{i=1}^{h}\lambda_{i}m_{i}+\sum_{j=1}^{l}\lambda_{j}m_{j}'}\\
 & =(1-\sum_{i=1}^{h}\lambda_{i})\iprod{r_{2}^{x^{\star}}}{m^{\star}}-\sum_{j=1}^{l}\lambda_{j}\iprod{r_{2}^{x^{\star}}}{m_{j}'}\\
 & =\sum_{j=1}^{l}\lambda_{j}\iprod{r_{2}^{x^{\star}}}{m^{\star}}-\sum_{j=1}^{l}\lambda_{j}\iprod{r_{2}^{x^{\star}}}{m_{j}'}\\
 & =\sum_{j=1}^{l}\lambda_{j}\iprod{r_{2}^{x^{\star}}}{m^{\star}-m_{j}'}.
\end{aligned}
\]
The second equality is because $\iprod{r_{2}^{x^{\star}}}{m^{\star}}=\iprod{r_{2}^{x^{\star}}}{m_{i}}$ for $i=1,2,\ldots,h$, which is obtained by the assumption that $m_{i}\in\Mdet^{\star}=\Mdet\cap\BR(x^{\star})\subseteq\BR(x^{\star})$. It then follows that
\begin{equation}
\sum_{j=1}^{l}\lambda_{j}\iprod{r_{2}^{x^{\star}}}{m^{\star}-m_{j}'}\leq2\tau\cdot\frac{1}{1-\gamma}\cdot\log\vert\mathcal{A}\vert.\label{eq:bound of r_2*m}
\end{equation}
Since $\Mdet\setminus\Mdet^{\star}$ is a finite set and hence $\max_{m\in\Mdet\setminus\Mdet^{\star}}\langle r_{2}^{x^{\star}},m\rangle=b$ exists, one can obtain
\begin{align*}
 & b\sum_{j=1}^{l}\lambda_{j}\leq\sum_{j=1}^{l}\lambda_{j}\iprod{r_{2}^{x^{\star}}}{m^{\star}-m_{j}'}.
\end{align*}
Combine with (\ref{eq:bound of r_2*m}) to obtain $b\sum_{j=1}^{l}\lambda_{j}\leq2\tau\cdot\frac{1}{1-\gamma}\cdot\log\vert\mathcal{A}\vert$, i.e., $\sum_{j=1}^{l}\lambda_{j}\leq2\tau\cdot\frac{1}{b(1-\gamma)}\cdot\log\vert\mathcal{A}\vert$. Therefore,
\begin{align*}
\iprod{r_{1}}{m_{\tau}^{\star}} & =\iprod{r_{1}}{\sum_{i=1}^{h}\lambda_{i}m_{i}+\sum_{j=1}^{l}\lambda_{j}m_{j}'}\\
 & =\sum_{i=1}^{h}\lambda_{i}\iprod{r_{1}}{m^{\star}}+\sum_{j=1}^{l}\lambda_{j}\iprod{r_{1}}{m_{j}'}\\
 & \geq\sum_{i=1}^{h}\lambda_{i}\iprod{r_{1}}{m^{\star}}\\
 & =(1-\sum_{j=1}^{l}\lambda_{j})\iprod{r_{1}}{m^{\star}}\\
 & \geq\left(1-2\tau\cdot\frac{1}{b(1-\gamma)}\cdot\log\vert\mathcal{A}\vert\right)\iprod{r_{1}}{m^{\star}}.
\end{align*}
The second equality follows from the assumption that $x^{\star}$ is an optimal interior-point allocation. From Proposition~\ref{prop:robustness_nonunique}, any best response of $x^{\star}$ is an optimal occupancy measure. Thus, $\iprod{r_{1}}{m_{i}}=\iprod{r_{1}}{m^{\star}}$ since $m_{i}\in\BR(x^{\star})$.

\end{proof}

\subsection{Alternative model of bounded rationality\label{subsec: proof of robustness to ent-regularize}}

\begin{prop}
\label{prop:Robustness to ent-regularized}Suppose that $(x^{\star},m^{\star})$ is an optimal solution to the reward design problem in~(\ref{eq:allocation_problem_om}), and $x^{\star}$ is an interior point of $\mathcal{P}_{m^{\star}}$. Define $\Mdet^{\star}\deq\Mdet\cap\BR(x^{\star})$, the set of optimal deterministic occupancy measures under $\xopt$. Let
\[
    m_{\mathrm{\tau}}^{\star}=\argmax_{m\in\cM}\left\{ \iprod{r_{2}^{x^{\star}}}{m}-\tau\cdot \ent(m)\right\}.
\]
Then for any $\tau>0$, it holds that 
\[
\iprod{r_{1}}{m_{\mathrm{\tau}}^\star}\geq\left(1-\frac{2\tau}{b(1-\gamma)}\lvert\log\left(\vert\mathcal{S}\vert\vert\mathcal{A}\vert\cdot(1-\gamma)\right)\rvert\right)\iprod{r_{1}}{m^{\star}},
\]
where $\gamma$ is the discount factor, and $b=\langle r_{2}^{\xopt},m^{\star}\rangle-\max_{m\in\Mdet\setminus\Mdet^{\star}}\langle r_{2}^{\xopt},m\rangle$.
\end{prop}

\begin{proof}
Notice that 
\begin{align*}
\tau\cdot\left(\ent(m^{\star})-\ent(m_{\tau}^{\star})\right) & \leq\tau\lvert \ent(m^{\star})+\ent(m_{\tau}^{\star})\rvert\\
 & =2\tau\left(\frac{1}{1-\gamma}\lvert\log\left(\vert\mathcal{S}\vert\vert\mathcal{A}\vert\cdot(1-\gamma)\right)\rvert\right).
\end{align*}
The upper bound of $\ent(m^{\star})+\ent(m_{\tau}^{\star})$ is derived from the facts that $m\in\mathcal{M}$ is non-negative and the sum over states and actions is bounded by $\frac{1}{1-\gamma}$ by Lemma~\ref{lem: sum_of_oc}. The rest of the proof is similar to the proof of Proposition~\ref{prop:robustness of ent-regularized-1}.
\end{proof}

\subsection{$\delta$-optimal response as a model for bounded rationality\label{subsec:delta-optimal-response-as}}

Consider a follower who chooses a worst response for the leader among all the $\delta$-optimal responses. Denote by $\BR_{\delta}(x)\deq\left\{m \mid \max_{m'\in\mathcal{M}}\iprod{r_2^x}{m'}-\iprod{r_2^x}{m}\leq\delta\right\}$ the set of $\delta$-optimal responses under a reward allocation $x$. The payoff of the leader under $x$ is given by
\[
\min_{m\in\BR_{\delta}(x)}\iprod{r_{1}}{m}.
\]
The following proposition shows that an optimal interior-point allocation is robust against $\delta$-optimal responses. 
\begin{prop}
\label{prop:robustness to delta optimal response}Suppose that $(x^{\star},m^{\star})$ is an optimal solution to the reward design problem in~(\ref{eq:allocation_problem_om}), and $x^{\star}$ is an interior point of $\mathcal{P}_{m^{\star}}$. Let $m_{\delta}^{\star}\in\argmin_{m\in\BR_{\delta}(x^{\star})}\iprod{r_{1}}{m}$ and $\Mdet^{\star}=\Mdet\cap\BR(x^{\star})$. Then for any $\delta>0$, the value $\iprod{r_{1}}{m_{\delta}^{\star}}$ satisfies 
\[
\iprod{r_{1}}{m_{\delta}^{\star}}\geq\left(1-\frac{\delta}{b}\right)\iprod{r_{1}}{m^{\star}},
\]
where $b=\langle r_{2}^{x^{\star}},m^{\star}\rangle-\max_{m\in\Mdet\setminus\Mdet^{\star}}\langle r_{2}^{x^{\star}},m\rangle$.
\end{prop}

\begin{proof}
Notice that
\[
\iprod{r_{2}^{x^{\star}}}{m^{\star}}-\iprod{r_{2}^{x^{\star}}}{m_{\delta}^{\star}}\leq\delta
\]
according to the definition of $\BR_{\delta}(x^{\star})$. The rest of the proof is similar to the proof of Proposition~\ref{prop:robustness of ent-regularized-1}.
\end{proof}

\section{Proofs on the optimality of deterministic occupancy measures}

\subsection{Proof of Proposition~\ref{lem:optimal_action}\label{subsec:Proof-of-Lemma_optimal_action}}
\begin{lem}
\label{lem:optimal_action-1}Let $\Qopt$ be the optimal Q-function of an MDP, and $\piopt$ be any deterministic greedy policy with respect to $\Qopt$, i.e., for all $s\in\mathcal{S}$, it holds that $\piopt(s,a)=1$ for some $a\in\argmax_{a'\in\mathcal{A}}\Qopt(s,a')$. For any optimal policy $\pi$ (not necessarily deterministic) of the MDP, it holds that
\[
\mathbb{E}_{a\sim\pi(s,\cdot)}\left[Q^{\mathrm{opt}}(s,a)\right]=\max_{a'\in\mathcal{A}}Q^{\mathrm{opt}}(s,a')\quad\forall s\in\mathcal{S}.
\]
\end{lem}

\begin{proof}
Let $d_{s}^{\pi}(s')=(1-\gamma)\sum_{t=0}^{\infty}\gamma^{t}\mathbb{P}(s_{t}=s'\mid s_{0}=s)$. By the performance difference lemma~\citep{kakadeApproximatelyOptimalApproximate2002}, 
\begin{align*}
V^{\pi}(s)-V^{\mathrm{opt}}(s) & =\frac{1}{1-\gamma}\mathbb{E}_{s'\sim d_{s}^{\pi}}\left[\mathbb{E}_{a\sim\pi(s', \cdot)}Q^{\mathrm{opt}}(s',a)-V^{\mathrm{opt}}(s')\right]\\
 & =\frac{1}{1-\gamma}\mathbb{E}_{s'\sim d_{s}^{\pi}}\left[\mathbb{E}_{a\sim\pi(s', \cdot)}Q^{\mathrm{opt}}(s',a)-\max_{a'\in\mathcal{A}}Q^{\mathrm{opt}}(s',a')\right].
\end{align*}
Meanwhile, it follows from the optimality of $\pi$ that $V^{\pi}(s)-V^{\mathrm{opt}}(s)=0$, which implies that 
\[
\mathbb{E}_{s'\sim d_{s}^{\pi}}\left[\mathbb{E}_{a\sim\pi(s', \cdot)}Q^{\mathrm{opt}}(s',a)-\max_{a'\in\mathcal{A}}Q^{\mathrm{opt}}(s',a')\right]=0.
\]
Because $\mathbb{E}_{a\sim\pi(s', \cdot)}Q^{\mathrm{opt}}(s',a)-\max_{a'\in\mathcal{A}}Q^{\mathrm{opt}}(s',a')\leq0$ for all $s'\in\mathcal{S}$, it must hold that
\[
\mathbb{E}_{a\sim\pi(s', \cdot)}Q^{\mathrm{opt}}(s',a)-\max_{a'\in\mathcal{A}}Q^{\mathrm{opt}}(s',a')=0
\]
for all $s'\in\mathcal{S}$.
\end{proof}
\begin{proof}[Proof of Proposition~\ref{lem:optimal_action}]
If $\pi^{\star}\in\Pidet$, then $\Pidet(\pi^{\star})=\{\pi^{\star}\}$ by definition. Therefore, any $\pi\in\Pidet(\pi^{\star})$ trivially satisfies $\langle r_{2}^{x^{\star}},m^{\pi}\rangle=\iprod{r_{2}^{x^{\star}}}{m^{\star}}=\max_{m'\in\mathcal{M}}\langle r_{2}^{x^{\star}},m'\rangle$. 

Consider the case when $\pi^{\star}\notin\Pidet$. Let $V_{2}^{\mathrm{opt}}$ and $\Qopt_2$ be the optimal value function and the optimal Q-function of the follower's MDP, respectively. Because $\pistar$ is an optimal policy, it follows from Lemma~\ref{lem:optimal_action-1} that 
\begin{equation}
	\mathbb{E}_{a\sim\pi^{\star}(s,\cdot)}Q_{2}^{\mathrm{opt}}(s,a)=\max_{a'\in\mathcal{A}}Q_{2}^{\mathrm{opt}}(s,a')\quad\forall s\in\mathcal{S}.\label{eq:optimal policy has same Q function}
\end{equation}
Hence, if $\pi^{\star}(s,a)\neq0$ for some $(s,a)\in\mathcal{S}\times\mathcal{A}$, then 
\begin{equation}
	Q_{2}^{\mathrm{opt}}(s,a)=\max_{a'\in\mathcal{A}}Q_{2}^{\mathrm{opt}}(s,a')=V_{2}^{\mathrm{opt}}(s).\label{eq:same Q value for non-det actions}
\end{equation}
For any $s\in\mathcal{S},$ let
\[
\mathcal{A}_{s}^{\star}=\{a\mid\pi^{\star}(s,a)\in(0,1)\},
\]
and consider the following two cases.

Case 1: Consider $s\in\mathcal{S}$ such that $\mathcal{A}_{s}^{\star}$ is nonempty. Since $\pi\in\Pidet(\pistar)$ is deterministic, there exists some $a^{\star}\in\cA$ such that $\pi(s,a^{\star})=1$. From the definition of $\Pidet(\pistar)$, it follows that $\pistar(s,a^{\star})\neq0$. Use~(\ref{eq:same Q value for non-det actions}) to obtain $Q_{2}^{\mathrm{opt}}(s,a^{\star})=V_{2}^{\mathrm{opt}}(s)$. This implies 
\[
\E_{a\sim\pi(s,\cdot)}\Qopt_{2}(s,a)=\Qopt_{2}(s,a^{\star})=V_{2}^{\mathrm{opt}}(s).
\]

Case 2: Consider $s\in\mathcal{S}$ such that $\mathcal{A}_{s}^{\star}$ is empty, i.e., $\pistar(s,a)\in\{0,1\}$ for all $a\in\cA$. Since $\pi\in\Pidet(\pistar)$, one must have $\pi(s,a)=\pi^{\star}(s,a)$ for all $a\in\mathcal{A}$. Use~(\ref{eq:optimal policy has same Q function}) to obtain
\[
\mathbb{E}_{a\sim\pi(s,\cdot)}Q_{2}^{\mathrm{opt}}(s,a)=\mathbb{E}_{a\sim\pistar(s,\cdot)}Q_{2}^{\mathrm{opt}}(s,a)=\max_{a'\in\mathcal{A}}Q_{2}^{\mathrm{opt}}(s,a')=V_{2}^{\mathrm{opt}}(s).
\]

Therefore, for all $s\in\mathcal{S},$ it holds that $\mathbb{E}_{a\sim\pi(s,\cdot)}Q_{2}^{\mathrm{opt}}(s,a)=V_{2}^{\mathrm{opt}}(s)$. Applying the performance difference lemma again, one can obtain that for all $s_0\in\mathcal{S}$,
\[
V_{2}^{\pi}(s_0)-V_{2}^{\mathrm{opt}}(s_0)=\frac{1}{1-\gamma}\mathbb{E}_{s\sim d_{s_0}^{\pi}}\left[\mathbb{E}_{a\sim\pi(s,\cdot)}Q_{2}^{\mathrm{opt}}(s,a)-V_{2}^{\mathrm{opt}}(s)\right]=0.
\]
This implies $\mathbb{E}_{s_{0}\sim\rho}\left[V_{2}^{\pi}(s_{0};\xopt)\right]=\mathbb{E}_{s_{0}\sim\rho}\left[V_{2}^{\mathrm{opt}}(s_{0};\xopt)\right]$ or, equivalently, $\langle r_{2}^{x^{\star}},m^{\pi}\rangle=\langle r_{2}^{x^{\star}},m^{\star}\rangle=\max_{m'\in\mathcal{M}}\langle r_{2}^{x^{\star}},m'\rangle$, from which one obtains $m^{\pi}\in\BR(x^{\star})$.
\end{proof}

\subsection{Proof of Proposition~\ref{prop:all determinisitc m are optimal}\label{subsec:Proof-of-Proposition all  deterministic oc are}}
\begin{proof}
If $\pi^{\star}\in\Pidet$, then $\Pidet(\pi^{\star})=\left\{ \pi^{\star}\right\} $, and the proposition holds trivially. Consider the case when $\pi^{\star}\notin\Pidet$. For any $s\in\cS$, let
\[
\mathcal{A}_{s}=\{a\in\mathcal{A}\mid\pi^{\star}(s,a)\in(0,1)\}.
\]
Because $\pistar\notin\Pidet$, there exists $s'\in\cS$ such that $\cA_{s'}$ is nonempty. Choose $a'\in\mathcal{A}_{s'}$ arbitrarily. Define a policy $\pi'$ as follows: $\pi'(s,\cdot)=\pi^{\star}(s,\cdot)$ when $s\neq s'$, and $\pi'(s,a')=1$ when $s=s'$. We shall first show that $\pi'\in\BR(x^{\star})$. By Lemma~\ref{lem:optimal_action-1}, since $\pi^{\star}$ is optimal, 
\begin{align*}
	\mathbb{E}_{a\sim\pi^{\star}(s',\cdot)}\left[Q_{2}^{\mathrm{opt}}(s',a)\right] & =\max_{a\in\mathcal{A}}Q_{2}^{\mathrm{opt}}(s',a).
\end{align*}
This implies that $Q_{2}^{\mathrm{opt}}(s',a)=\max_{a\in\mathcal{A}}Q_{2}^{\mathrm{opt}}(s',a)$ for all $a\in\mathcal{A}_{s'}$. Because $a'\in\cA_{s'}$, it holds that 
\[
\mathbb{E}_{a\sim\pi'(s',\cdot)}\left[Q_{2}^{\mathrm{opt}}(s',a)\right]=Q_{2}^{\mathrm{opt}}(s',a')=\max_{a\in\mathcal{A}}Q_{2}^{\mathrm{opt}}(s',a).
\]
Meanwhile, for $(s,a)$ such that $s\neq s'$, since $\pi'(s,a)=\pi^{\star}(s,a)$,
\[
\mathbb{E}_{a\sim\pi'(s,\cdot)}\left[Q_{2}^{\mathrm{opt}}(s,a)\right]=\mathbb{E}_{a\sim\pi^{\star}(s,\cdot)}\left[Q_{2}^{\mathrm{opt}}(s,a)\right]=\max_{a\in\mathcal{A}}Q_{2}^{\mathrm{opt}}(s,a).
\]
By the performance difference lemma, $\pi'$ is an optimal policy of the follower. 

Define $V_{1}^{\star}\deq V_{1}^{\pistar}$ and $Q_{1}^{\star}\deq Q_{1}^{\pistar}$. Because $\pi'$ is an optimal policy of the follower and hence a feasible solution to~(\ref{eq:original problem}), it follows from the optimality of $\pistar$ that 
\begin{equation}
	V_{1}^{\pi'}(s_{0})-V_{1}^{\star}(s_{0})\leq0.\label{eq:V_compare-1}
\end{equation}
By the performance difference lemma,
\begin{equation}
	V_{1}^{\pi'}(s_{0})-V_{1}^{\star}(s_{0})=\frac{1}{1-\gamma}\mathbb{E}_{s\sim d_{s_{0}}^{\pi'}}\left[\mathbb{E}_{a\sim\pi'(s,\cdot)}Q_{1}^{\star}(s,a)-V_{1}^{\star}(s)\right].\label{eq:V_compare-2}
\end{equation}
Since $\pi'(s,\cdot)=\pi^{\star}(s,\cdot)$ for any $s\neq s'$, we know that $\mathbb{E}_{a\sim\pi'(s,\cdot)}Q_{1}^{\star}(s,a)-V_{1}^{\star}(s)=0$ for $s\neq s'$. It then follows from (\ref{eq:V_compare-1}) and (\ref{eq:V_compare-2}) that $\mathbb{E}_{a\sim\pi'(s,\cdot)}Q_{1}^{\star}(s,a)-V_{1}^{\star}(s)\leq0$ when $s=s'$, i.e.,
\[
Q_{1}^{\star}(s',a')\leq V_{1}^{\star}(s').
\]
Since $s'\in\cS$ and $a'\in\cA_{s'}$ were chosen arbitrarily, it holds that $Q_{1}^{\star}(s,a)\leq V_{1}^{\star}(s)$ for all $(s,a)\in\mathcal{S}\times\mathcal{A}$ such that $\pi^{\star}(s,a)\in(0,1)$. However, the inequality can never be achieved, else 
\[
\mathbb{E}_{a\sim\pi^{\star}(s,\cdot)}Q_{1}^{\star}(s,a)-V_{1}^{\star}(s)<0,
\]
which violates the Bellman consistency equation. In other words, $Q_{1}^{\star}(s,a)=V_{1}^{\star}(s)$ when $\pi^{\star}(s,a)\in(0,1)$. In addition, $Q_{1}^{\star}(s,a)=V_{1}^{\star}(s)$ when $\pistar(s,a)=1$. Combine the two conditions to obtain $Q_{1}^{\star}(s,a)=V_{1}^{\star}(s)$ when $\pi^{\star}(s,a)\neq0$. Consider any $\pi\in\Pidet(\pistar)$. When $\pi(s,a)=1$, it must hold that $\pistar(s,a)\neq0$, and consequently $Q_{1}^{\star}(s,a)=V_{1}^{\star}(s)$. For any $s_{0}\in\cS$, by the performance difference lemma, 
\[
V_{1}^{\pi}(s_{0})-V_{1}^{\star}(s_{0})=\frac{1}{1-\gamma}\mathbb{E}_{s\sim d_{s_{0}}^{\pi}}\left[\mathbb{E}_{a\sim\pi(s,\cdot)}Q_{1}^{\star}(s,a)-V_{1}^{\star}(s)\right]=\frac{1}{1-\gamma}\mathbb{E}_{s\sim d_{s_{0}}^{\pi}}\left[V_{1}^{\star}(s)-V_{1}^{\star}(s)\right]=0.
\]
This implies $\mathbb{E}_{s_{0}\sim\rho}\left[V_{1}^{\pi}(s_{0})\right]=\mathbb{E}_{s_{0}\sim\rho}\left[V_{1}^{\star}(s_{0})\right]$ or, equivalently, $\iprod{r_{1}}{m^{\pi}}=v_{1}^{\star}$. 
\end{proof}

\section{Proofs of the existence of optimal interior-point allocation}

\subsection{Proof of Theorem~\ref{Thm:Sufficiency}\label{subsec:Proof-of-Theorem_Sufficiency}}
\begin{lem}
\label{lem:intersection of convex and open}Let $X$ and $Y$ be two sets. Suppose that $X$ is convex with a nonempty interior and that $Y$ is open. Then either $X\cap Y=\emptyset$, or $X\cap Y$ has a nonempty interior. 
\end{lem}

\begin{proof}
The proof is reproduced from a post on StackExchange (see Footnote\footnote{\url{https://math.stackexchange.com/q/2701244}}). Within this proof, for a given set $X$, we shall denote its interior by $X^{\circ}$, its boundary by $\partial X$, its complement by $X^C$, and its closure by $\overline{X}$. 

We first show that $\partial X=\partial(X^{\circ})$. Let $c\in X^{\circ}$ and $x\in\partial X$. Without loss of generality, we assume that $x=0$ by a change of coordinates. Then $\lambda c\in X$ for $0\leq\lambda\leq1$ by convexity. Because $c\in X^{\circ},$ there is an open set $U\subseteq X^{\circ}$ that contains $c$. Define $U_{\lambda}\deq\{y\mid y=\lambda u,\ u\in U\}$. It is not hard to verify that when $0<\lambda\leq1$, the set $U_{\lambda}$ is open, contains $\lambda c$, and $U_{\lambda}\subseteq X$ by the convexity of $X$, which implies $\lambda c\in X^{\circ}$ when $0<\lambda\leq1$. This shows $x\in\overline{X^{\circ}}$. Meanwhile, since $x\in\partial X$, it holds that $x\in\overline{X^{C}}\subseteq\overline{X^{\circ C}}$. We obtain that $x\in\overline{X^{\circ}}\cap\overline{X^{\circ C}}=\partial(X^{\circ})$, which implies $\partial X\subseteq\partial(X^{\circ})$. For the other direction, recall that $\overline{X^{C}}=X^{\circ C}$ for any $X$. Use this result to obtain $\overline{X}=\overline{(X^{C})^{C}}=X^{C\circ C}$, $\overline{X^{\circ}}=\overline{(X^{\circ C})^{C}}=X^{\circ C\circ C}$, and $\overline{X^{\circ C}}=X^{\circ\circ C}=X^{\circ C}$. Thus,
\[
\partial X=\overline{X}\cap\overline{X^{C}}=X^{\circ C}\cap X^{C \circ C},
\]
and 
\[
\partial (X^{\circ})=\overline{X^{\circ}}\cap\overline{X^{\circ C}}=X^{\circ C}\cap X^{\circ C \circ C}.
\]
From the fact $X^{C\circ C}\supseteq X^{\circ C \circ C}$, we obtain $\partial X\supseteq\partial (X^{\circ})$.

Suppose that $X\cap Y$ is nonempty, and let $x\in X\cap Y$. If $x\in X^{\circ}$, then $X^{\circ}\cap Y\neq\emptyset$ is a nonempty open set in $X\cap Y$. If $x\in\partial X$, then it follows from the previous discussion that $x \in \partial (X^{\circ})$. Because $Y$ is open, there exists an open ball $B(x) \subseteq Y$ centered at $x$. Moreover, $B(x)\cap X^{\circ}\neq\emptyset$ since $x \in \partial(X^{\circ})$. Therefore, $B(x)\cap X^{\circ}$ is a nonempty open set in $X\cap Y$.
\end{proof}
\begin{lem}
\label{lem: inequality of delta^s}Given a function $v \colon \mathcal{S}_{d}\rightarrow\mathbb{R}$ and an occupancy measure $m$, define 
\begin{equation}
\delta^{v}(s,a)=\begin{cases}
v(s) & \text{if }s\in\mathcal{S}_{d},\\
0 & \text{otherwise},
\label{eq:def_deltav}
\end{cases}
\end{equation}
and $\nu(s)=\sum_{a\in\mathcal{A}}m(s,a)$. Then $|\iprod{\delta^{v}}{m}| \leq\Vert v\Vert_{1}\max_{s\in\mathcal{S}_{d}}\nu(s)$.
\end{lem}

\begin{proof}
According to the definition,
\begin{align*}
\iprod{\delta^{v}}{m} & =\sum_{s\in\mathcal{S}}\sum_{a\in\mathcal{A}}m(s,a)\delta^{v}(s,a)  =\sum_{s\in\mathcal{S}_{d}}\sum_{a\in\mathcal{A}}v(s)m(s,a)\\
 & =\sum_{s\in\mathcal{S}_{d}}v(s)\sum_{a\in\mathcal{A}}m(s,a) =\sum_{s\in\mathcal{S}_{d}}v(s)\nu(s). \numberthis \label{eq:deltav_m_iprod}
\end{align*}
Apply H\"{o}lder's inequality to complete the proof.
\end{proof}
\begin{lem}
\label{lem:nonempty interior 2}Suppose that $x\in\cX$ and $m \in \cM$ satisfy $\iprod{r_{2}^{x}}{m}>\iprod{r_{2}^{x}}{\mdet}$ for all $\mdet\in\Mdet\setminus\{m\}$. Then $\mathcal{P}_{m}$ has a nonempty interior, and $x$ is an interior-point allocation of $\mathcal{P}_{m}$.
\end{lem}

\begin{proof}
Let $\mdet\in\Mdet$. Define $\sigma(m,\mdet) \deq \iprod{r_{2}^{x}}{m-\mdet}$. Let $x'(v) \deq x + \epsilon v$, where $\epsilon=\min_{\mdet'\in\Mdet\setminus\{m\}}\sigma(m,\mdet')$, and $v$ is any vector satisfying $\Vert v\Vert_{1}\leq\frac{1}{2}(1-\gamma)$. It follows that $0<\epsilon\leq\sigma(m,\mdet)$. Define $\nu_{\mathrm{det}}(s) \deq \sum_{a\in\mathcal{A}}\mdet(s,a)$ and $\nu(s) \deq \sum_{a\in\mathcal{A}}m(s,a)$. It holds that 
\begin{align*}
& \iprod{r_{2}^{x'(v)}}{\mdet} -\iprod{r_{2}^{x'(v)}}{m}=\iprod{r_{2}^{x+\epsilon v}}{\mdet-m}\\
 & \qquad =\iprod{r_{2}^{x}}{\mdet-m}+\epsilon\cdot\iprod{\delta^{v}}{\mdet-m}\\
 & \qquad \leq-\epsilon+\epsilon\cdot\iprod{\delta^{v}}{\mdet}-\epsilon\cdot\iprod{\delta^{v}}{m}\\
 & \qquad \leq-\epsilon+\epsilon\cdot\frac{1}{2}(1-\gamma)\cdot\max_{s\in\mathcal{S}_{d}}\nu_{\mathrm{det}}(s)+\epsilon\cdot\frac{1}{2}(1-\gamma)\cdot\max_{s\in\mathcal{S}_{d}}\nu(s)\\
 & \qquad \leq-\epsilon+\epsilon\\
 & \qquad = 0 \numberthis \label{eq:mdet_compare},
\end{align*}
where $\delta^{v}$ is defined in \eqref{eq:def_deltav}. The second inequality is due to Lemma~\ref{lem: inequality of delta^s}. The third inequality holds because $\sum_{(s,a)\in\mathcal{S}\times\mathcal{A}}m(s,a)$ is bounded by $\frac{1}{1-\gamma}$ according to Lemma~\ref{lem: sum_of_oc}. Because any MDP must have a deterministic optimal occupancy measure, \eqref{eq:mdet_compare} implies $m\in\BR(x'(v))$ for all $\Vert v\Vert_{1}\leq\frac{1}{2}(1-\gamma)$. In other words, the open ball 
\[
\{x+\epsilon v\mid\Vert v\Vert_{1}<\frac{1}{2}(1-\gamma),\ \epsilon=\min_{\mdet'\in\Mdet\setminus\{m\}}\sigma(m,\mdet')\}
\]
contains $x\in\mathcal{X}$ and is in $\mathcal{P}_{m}$, proving that $x$ is an interior point of $\cP_m$.
\end{proof}

\begin{proof}[Proof of Theorem~\ref{Thm:Sufficiency}]
(Sufficiency) Let $(x^{0},m^{0})$ be an optimal solution of (\ref{eq:margin problem}) satisfying $C-\sum_{s\in\mathcal{S}_{d}}x^{0}(s)>0$. Define
\begin{equation}
\mathcal{M}^{\star}(x^{0})\deq\{m\mid\iprod{r_{1}}{m}=\iprod{r_{1}}{m^{0}}\} \cap (\Mdet \cap \BR(x^{0})) \label{eq:def_Mstar_x0}
\end{equation}
and
\[
b \deq \iprod{r_{2}^{x^{0}}}{m^{0}}-\max_{m'\in\Mdet\setminus\BR(x^{0})}\iprod{r_{2}^{x^{0}}}{m'}>0.
\]
Consider $x'=x^{0}+\epsilon\cdot\frac{(1-\gamma)}{2\vert\mathcal{S}_{d}\vert}\cdot\mathbf{1}$, where $\epsilon=\min\left\{ b,\frac{1}{(1-\gamma)}\left(C-\sum_{s\in\mathcal{S}_{d}}x^{0}(s)\right)\right\} >0$. It follows that $x'\in\cX$ because
\begin{align*}
C-\sum_{s\in\mathcal{S}_{d}}x'(s) & =C-\sum_{s\in\mathcal{S}_{d}}x^{0}(s)-\vert\mathcal{S}_{d}\vert\cdot\epsilon\cdot\frac{(1-\gamma)}{2\vert\mathcal{S}_{d}\vert}\\
 & \geq C-\sum_{s\in\mathcal{S}_{d}}x^{0}(s)-\frac{1}{(1-\gamma)}\left(C-\sum_{s\in\mathcal{S}_{d}}x^{0}(s)\right)\cdot\frac{(1-\gamma)}{2}\\
 & =\frac{1}{2}\left(C-\sum_{s\in\mathcal{S}_{d}}x^{0}(s)\right)\\
 & >0.
\end{align*}
For any $m \in \cM$, it holds that
\begin{align*}
\iprod{r_{2}^{x'}}{m} & =\iprod{r_{2}^{x^{0}+\epsilon\cdot\frac{(1-\gamma)}{2\vert\mathcal{S}_{d}\vert}\cdot\mathbf{1}}}{m}\\
 & =\iprod{r_{2}^{x^{0}}}{m}+\epsilon\cdot\frac{(1-\gamma)}{2\vert\mathcal{S}_{d}\vert}\sum_{(s,a):s\in\mathcal{S}_{d}}m(s,a)\\
 & =\iprod{r_{2}^{x^{0}}}{m}+\epsilon\cdot\frac{(1-\gamma)}{2\vert\mathcal{S}_{d}\vert}\cdot\iprod{r_{1}}{m}. \numberthis \label{eq:r2xprime_m}
\end{align*}
Furthermore, when $m\in\Mdet\setminus\BR(x^{0})$, it follows from \eqref{eq:r2xprime_m} that
\begin{align*}
\iprod{r_{2}^{x'}}{m} & =\iprod{r_{2}^{x^{0}}}{m}+\epsilon\cdot\frac{(1-\gamma)}{2\vert\mathcal{S}_{d}\vert}\cdot\iprod{r_{1}}{m}\\
 & \leq\iprod{r_{2}^{x^{0}}}{m}+\epsilon\cdot\frac{(1-\gamma)}{2\vert\mathcal{S}_{d}\vert}\cdot\frac{1}{1-\gamma}\\
 & \leq\iprod{r_{2}^{x^{0}}}{m}+\epsilon\cdot\frac{1}{2\vert\mathcal{S}_{d}\vert}\\
 & \leq\iprod{r_{2}^{x^{0}}}{m}+\frac{b}{2}\\
 & \leq\iprod{r_{2}^{x^{0}}}{m^{0}}-\frac{b}{2}\\
 & \leq\iprod{r_{2}^{x'}}{m^{0}}-\frac{b}{2}, \numberthis \label{eq:m in M_det=00005CBR}
\end{align*}
where the first inequality is from Lemma~\ref{lem: sum_of_oc}. In the meantime, consider any $\tilde{m}\in\left(\Mdet\cap\BR(x^{0})\right)\setminus\mathcal{M}^{\star}(x^{0})$. It follows that $\iprod{r_{1}}{\tilde{m}}\neq\iprod{r_{1}}{m^{0}}$ and $\iprod{r_{2}^{x^{0}}}{\tilde{m}}=\iprod{r_{2}^{x^{0}}}{m^{0}}$ from the definition of $\mathcal{M}^{\star}(x^{0})$. Since $(x^{0},m^{0})$ is an optimal solution of the reward design problem in (\ref{eq:allocation_problem_om}), it holds that $\iprod{r_{1}}{m^{0}}>\iprod{r_{1}}{\tilde{m}}$. Hence, for any $\tilde{m}\in\left(\Mdet\cap\BR(x^{0})\right)\setminus\mathcal{M}^{\star}(x^{0})$, it follows from \eqref{eq:r2xprime_m} that
\begin{equation}
\begin{aligned}\iprod{r_{2}^{x'}}{\tilde{m}} & =\iprod{r_{2}^{x^{0}}}{\tilde{m}}+\epsilon\cdot\frac{(1-\gamma)}{2\vert\mathcal{S}_{d}\vert}\cdot\iprod{r_{1}}{\tilde{m}}\\
 & <\iprod{r_{2}^{x^{0}}}{\tilde{m}}+\epsilon\cdot\frac{(1-\gamma)}{2\vert\mathcal{S}_{d}\vert}\cdot\iprod{r_{1}}{m^{0}}\\
 & =\iprod{r_{2}^{x^{0}}}{m^{0}}+\epsilon\cdot\frac{(1-\gamma)}{2\vert\mathcal{S}_{d}\vert}\cdot\iprod{r_{1}}{m^{0}}\\
 & =\iprod{r_{2}^{x'}}{m^{0}}.
\end{aligned}
\label{eq:m in (M_det intersect BR)=00005CM^star}
\end{equation}

Let $c=\iprod{r^{x'}}{m^{0}}-\max_{m'\in\left(\Mdet\cap\BR(x^{0})\right)\setminus\mathcal{M}^{\star}(x^{0})}\iprod{r^{x'}}{m'}$. It follows from~(\ref{eq:m in (M_det intersect BR)=00005CM^star}) that $c>0$. Let $d=\min\left\{ \frac{1}{2}b,c\right\} >0$. Recall that $\Mdet=\left(\Mdet\setminus\BR(x^{0})\right)\cup\left(\left(\Mdet\cap\BR(x^{0})\right)\setminus\mathcal{M}^{\star}(x^{0})\right)\cup\mathcal{M}^{\star}(x^{0})$. For any $m'\in\Mdet\setminus\mathcal{M}^{\star}(x^{0})=\left(\Mdet\setminus\BR(x^{0})\right)\cup\left(\left(\Mdet\cap\BR(x^{0})\right)\setminus\mathcal{M}^{\star}(x^{0})\right)$, from (\ref{eq:m in M_det=00005CBR}) and (\ref{eq:m in (M_det intersect BR)=00005CM^star}), we obtain that 
\begin{align*}
\iprod{r_{2}^{x'}}{m'} & \leq\iprod{r_{2}^{x'}}{m^{0}}-\min\left\{ \frac{1}{2}b,c\right\} \\
 & =\iprod{r_{2}^{x'}}{m^{0}}-d.
\end{align*}

We will continue the proof by considering two mutually exclusive cases.

Case 1: Suppose for all $s\in\mathcal{S}_{d}$ and $m\in\mathcal{M}^{\star}(x^{0})$, it holds that
\begin{equation}
\sum_{a\in\mathcal{A}}m(s,a)=\sum_{a\in\mathcal{A}}m^{0}(s,a).\label{eq:equal for all s,a}
\end{equation}
For any $m\in\mathcal{M}^{\star}(x^{0})$, it follows from \eqref{eq:def_Mstar_x0} that $\iprod{r_{2}^{x^{0}}}{m^{0}}=\iprod{r_{2}^{x^{0}}}{m}$ and $\iprod{r_{1}}{m^{0}}=\iprod{r_{1}}{m}$. Use \eqref{eq:r2xprime_m} to obtain
\begin{equation}
\begin{alignedat}{1}\iprod{r_{2}^{x'}}{m^{0}} & =\iprod{r_{2}^{x^{0}}}{m^{0}}+\epsilon\cdot\frac{(1-\gamma)}{2\vert\mathcal{S}_{d}\vert}\cdot\iprod{r_{1}}{m^{0}}\\
 & =\iprod{r_{2}^{x^{0}}}{m}+\epsilon\cdot\frac{(1-\gamma)}{2\vert\mathcal{S}_{d}\vert}\cdot\iprod{r_{1}}{m}\\
 & =\iprod{r_{2}^{x'}}{m}.
\end{alignedat}
\label{eq: equality of x', mo and m}
\end{equation}
Given a function $v \colon \mathcal{S}_{d}\rightarrow\mathbb{R}$, define $\delta^{v}$ as in \eqref{eq:def_deltav}. In the following, we sometimes abuse the notation and treat $v$ as a vector in $\reals^{|\Sd|}$. Define $x'(v)\deq x'+dv$. For any $m\in\mathcal{M}^{\star}(x^{0})$ and $v\in\reals^{|\Sd|}$,
\begin{equation}
	\iprod{r_{2}^{x'(v)}}{m^{0}} =\iprod{r_{2}^{x'+dv}}{m^{0}} = \iprod{r_{2}^{x'}}{m^{0}}+d\cdot\iprod{\delta^{v}}{m^{0}} =\iprod{r_{2}^{x'}}{m}+d\cdot\iprod{\delta^{v}}{m^{0}}.
	\label{eq: equalities of x'-1}
\end{equation}
From \eqref{eq:deltav_m_iprod} and \eqref{eq:equal for all s,a}, 
\begin{equation}
	\iprod{\delta^{v}}{m^{0}}=\sum_{s\in\Sd}v(s)\sum_{a\in\cA}m^{0}(s,a)=\sum_{s\in\Sd}v(s)\sum_{a\in\cA}m(s,a)=\iprod{\delta^{v}}{m}.\label{eq:iprod_deltav_m0}
\end{equation}
Substitute \eqref{eq:iprod_deltav_m0} into \eqref{eq: equalities of x'-1} to obtain
\begin{equation}
\iprod{r_{2}^{x'(v)}}{m^{0}}=\iprod{r_{2}^{x'}}{m}+d\cdot\iprod{\delta^{v}}{m}=\iprod{r_{2}^{x'(v)}}{m}. \label{eq: equalities of x'}
\end{equation}
Furthermore, when $\Vert v\Vert_{1}\leq\frac{1}{2}(1-\gamma)$, for any $m'\in\Mdet\setminus\mathcal{M}^{\star}(x^{0})$, it holds that
\begin{align*}
\iprod{r_{2}^{x'(v)}}{m'} & =\iprod{r_{2}^{x'+dv}}{m'}\\
 & =\iprod{r_{2}^{x'}}{m'}+d\cdot\iprod{\delta^{v}}{m'}\\
 & \leq\iprod{r_{2}^{x'}}{m'}+d\cdot\frac{1}{2}(1-\gamma)\cdot\max_{s\in\mathcal{S}_{d}}\sum_{a\in\mathcal{A}}m'(s,a)\\
 & \leq\iprod{r_{2}^{x'}}{m'}+\frac{1}{2}d\\
 & \leq\iprod{r_{2}^{x'}}{m^{0}}-\frac{1}{2}d\\
 & \leq\iprod{r_{2}^{x'}}{m^{0}}-d\cdot\frac{1}{2}(1-\gamma)\cdot\max_{s\in\mathcal{S}_{d}}\sum_{a\in\mathcal{A}} m^{0}(s,a)\\
 & \leq\iprod{r_{2}^{x'}}{m^{0}}+d\cdot\iprod{\delta^{v}}{m}\\
 & =\iprod{r_{2}^{x'+dv}}{m^{0}}\\
 & =\iprod{r_{2}^{x'(v)}}{m^{0}}. \numberthis \label{eq: inequalities of x'}
\end{align*}
The first inequality is from Lemma~\ref{lem: inequality of delta^s}. From (\ref{eq: equalities of x'}) and (\ref{eq: inequalities of x'}), it follows that $\iprod{r_{2}^{x'(v)}}{m^{0}}\geq\iprod{r_{2}^{x'(v)}}{m}$ for any $m\in\Mdet$ and $\Vert v\Vert_{1}<(1-\gamma)/2$. This implies that the open ball 
\[
\{x'+dv\mid\Vert v\Vert_{1}<\frac{1}{2}(1-\gamma)\}
\]
is contained in $\mathcal{P}_{m^{0}}$. Since $x'\in\mathcal{X}$, and $m^{0}$ is an optimal occupancy measure, $x'$ must be an optimal interior-point allocation. 

Case 2: Consider $m\in\mathcal{M}^{\star}(x^{0})$. When condition (\ref{eq:equal for all s,a}) in Case 1 does not hold, there must exist some $s^{\star}\in\mathcal{S}_{d}$ such that 
\begin{equation}
\min_{m\in\mathcal{M}^{\star}(x^{0})}\sum_{a\in\mathcal{A}}m(s^{\star},a)\neq\sum_{a\in\mathcal{A}}m^{0}(s^{\star},a).\label{eq: inequality for some s,a}
\end{equation}
The minimum on the left side of (\ref{eq: inequality for some s,a}) exists because $\mathcal{M}^{\star}(x^{0})$ is finite. Moreover, there must exist $m_{\star}\in\mathcal{M}^{\star}(x^{0})$ satisfying $\sum_{a\in\mathcal{A}}m_{\star}(s^{\star},a)=\min_{m\in\mathcal{M}^{\star}(x^{0})}\sum_{a\in\mathcal{A}}m(s^{\star},a)$. Consider $x''$ such that $x''(s)=x'(s)$ when $s\in\mathcal{S}_{d}\setminus\{s^{\star}\}$ and $x''(s^{\star})=x'(s^{\star})-\min\left\{ \frac{d}{2\sum_{a\in\mathcal{A}}m_{\star}(s^{\star},a)},\epsilon\cdot\frac{(1-\gamma)}{2\vert\mathcal{S}_{d}\vert}\right\} $. Notice that $x''(s^{\star})\geq0$ because $x'(s^{\star})\geq\epsilon\cdot\frac{(1-\gamma)}{2\vert\mathcal{S}_{d}\vert}$. Moreover, since $C-\sum_{s\in\mathcal{S}_{d}}x'(s)\geq0$, it holds that $C-\sum_{s\in\mathcal{S}_{d}}x''(s)\geq0$, implying that $x''\in\mathcal{X}$. Since $m_{\star}\in\mathcal{M}^{\star}(x^{0})$, it follows from~(\ref{eq: equality of x', mo and m}) that $\iprod{r_{2}^{x'}}{m_{\star}}=\iprod{r_{2}^{x'}}{m^{0}}$. Thus, for any $m\in\Mdet\setminus\mathcal{M}^{\star}(x^{0})$,
\begin{align*}\iprod{r_{2}^{x''}}{m_{\star}} & \geq\iprod{r_{2}^{x'}}{m_{\star}}-\frac{d}{2\sum_{a\in\mathcal{A}}m_{\star}(s^{\star},a)}\sum_{a\in\mathcal{A}}m_{\star}(s^{\star},a)\\
 & =\iprod{r_{2}^{x'}}{m_{\star}}-\frac{d}{2}\\
 & =\iprod{r_{2}^{x'}}{m^{0}}-\frac{d}{2}\\
 & >\iprod{r_{2}^{x'}}{m^{0}}-d\\
 & \geq\iprod{r_{2}^{x'}}{m}\\
 & \geq\iprod{r_{2}^{x''}}{m}. \numberthis \label{eq: inequality of x''}
\end{align*}
In the meantime, for any $m'\in\mathcal{M}^{\star}(x^{0})\setminus\{m_{\star}\}$,
\begin{equation}
\begin{alignedat}{1}\iprod{r_{2}^{x''}}{m_{\star}} & =\iprod{r_{2}^{x'}}{m_{\star}}-\min\left\{ \frac{d}{2\sum_{a\in\mathcal{A}}m_{\star}(s^{\star},a)},\epsilon\cdot\frac{(1-\gamma)}{2\vert\mathcal{S}_{d}\vert}\right\} \cdot\sum_{a\in\mathcal{A}}m_{\star}(s^{\star},a)\\
 & >\iprod{r_{2}^{x'}}{m_{\star}}-\min\left\{ \frac{d}{2\sum_{a\in\mathcal{A}}m_{\star}(s^{\star},a)},\epsilon\cdot\frac{(1-\gamma)}{2\vert\mathcal{S}_{d}\vert}\right\} \cdot\sum_{a\in\mathcal{A}}m'(s^{\star},a)\\
 & =\iprod{r_{2}^{x'}}{m'}-\min\left\{ \frac{d}{2\sum_{a\in\mathcal{A}}m_{\star}(s^{\star},a)},\epsilon\cdot\frac{(1-\gamma)}{2\vert\mathcal{S}_{d}\vert}\right\} \cdot\sum_{a\in\mathcal{A}}m'(s^{\star},a)\\
 & =\iprod{r_{2}^{x''}}{m'}.
\end{alignedat}
\label{eq: inequality of x'' 2}
\end{equation}
Hence we obtain $x''\in\mathcal{X}$ and $\iprod{r_{2}^{x''}}{m_{\star}}>\iprod{r_{2}^{x''}}{\mdet}$ for all $\mdet\in\Mdet\setminus\{m_{\star}\}$ from (\ref{eq: inequality of x''}) and (\ref{eq: inequality of x'' 2}). By Lemma~\ref{lem:nonempty interior 2}, we know that $x''$ is an interior-point allocation of $\cP_{m_{\star}}$. Since $m_{\star}$ is also an optimal occupancy measure of the reward design problem in (\ref{eq:allocation_problem_om}) by the definition of $\mathcal{M}^{\star}(x^{0})$, it follows that $x''$ is an optimal interior-point allocation.

(Necessity) Assume toward a contradiction that $\xopt$ is an optimal interior-point allocation in $\cP_{m}$ for some optimal occupancy measure $m$. Then there must exist an open set $B(\xopt)$ containing $\xopt$ with $B(\xopt)\subseteq\cP_{m}$. Since $\cX$ is convex, the set $B(x^{\star})\cap\mathcal{X}$ has a nonempty interior by Lemma~\ref{lem:intersection of convex and open}. On the other hand, from the given condition, any $x\in\mathcal{P}_{m}$ must satisfy $C-\sum_{i}x_{i}^{0}=0$ and is hence on the boundary of $\cX$. This implies that the interior of $\cP_{m}\cap\cX$ is empty. Since $B(\xopt)\subseteq\cP_{m}$, it follows that the interior of $B(x^{\star})\cap\mathcal{X}$ is also empty, leading to a contradiction.
\end{proof}

\subsection{Proof of Theorem~\ref{prop:best responses to existence}\label{subsec:Proof-of-Theorem_best_response_to_existence}}
\begin{proof}
Consider $m^{\star}\in\BR(x^{\star})$. Define
\[
\BRdet^{\star} \deq \Mdet\cap\BR(\xopt)=\left\{ m\in\Mdet\mid\langle r_{2}^{x^{\star}},m\rangle=\langle r_{2}^{x^{\star}},m^{\star}\rangle\right\}
\]
and $b \deq \langle r_{2}^{x^{\star}},m^{\star}\rangle-\max_{m\in\Mdet\setminus\BRdet^{\star}}\langle r_{2}^{x^{\star}},m\rangle$. Because any best response $m\in\BR(x^{\star})$ satisfies $\iprod{r_{1}}{m}=v_{1}^{\star}$, it holds that $\langle r_{1},m\rangle=\langle r_{1},m^{\star}\rangle$ for all $m\in\BRdet^{\star}$. 

Case 1: Suppose that 
\begin{equation}
\sum_{a\in\mathcal{A}}m(s,a)=\sum_{a\in\cA}m'(s,a)\label{eq:identical_m}
\end{equation}
for any $m,m'\in\BRdet^{\star}$ and any $s\in\mathcal{S}_{d}$. Given a function $v \colon \mathcal{S}_{d}\rightarrow\mathbb{R}$, define $\delta^{v}$ as in \eqref{eq:def_deltav}. In the following, we will sometimes abuse the notation and treat $v$ as a vector in $\reals^{|\Sd|}$. Consider any $x'\in\{x^{\star}+\epsilon v\mid\Vert v\Vert_{1}<1\}$. For any $m\in\BRdet^{\star}$, it holds that
\begin{equation}
\begin{alignedat}{1}\langle r_{2}^{x'},m^{\star}\rangle & =\langle r_{2}^{x^{\star}+\epsilon v},m^{\star}\rangle\\
 & =\langle r_{2}^{x^{\star}},m^{\star}\rangle+\epsilon\langle\delta^{v},m^{\star}\rangle\\
 & =\langle r_{2}^{x^{\star}},m\rangle+\epsilon\langle\delta^{v},m\rangle\\
 & =\langle r_{2}^{x'},m\rangle.
\end{alignedat}
\label{eq: equality of x'}
\end{equation}
The third equality is due to the condition in~(\ref{eq:identical_m}) and the fact that $m\in\BRdet^{\star}$. In addition, when $\epsilon\leq\frac{1-\gamma}{2\vert\mathcal{A}\vert}\cdot b$, it holds that
\begin{align*}
\langle r_{2}^{x'},m^{\star}\rangle & =\langle r_{2}^{x^{\star}},m^{\star}\rangle+\epsilon\langle\delta^{v},m^{\star}\rangle\\
 & >\langle r_{2}^{x^{\star}},m^{\star}\rangle-\epsilon\cdot|\mathcal{A}|\max_{(s,a):s\in\mathcal{S}_{d}}|m^{\star}(s,a)|\\
 & \geq\langle r_{2}^{x^{\star}},m^{\star}\rangle-\epsilon\cdot|\mathcal{A}|\cdot\frac{1}{1-\gamma}\\
 & \geq\langle r_{2}^{x^{\star}},m^{\star}\rangle-\frac{b}{2}\\
 & =\max_{m'\in\Mdet\setminus\BRdet^{\star}}\langle r_{2}^{x^{\star}},m'\rangle+\frac{b}{2}\\
 & \geq\max_{m'\in\Mdet\setminus\BRdet^{\star}}\left(\langle r_{2}^{x^{\star}},m'\rangle+\epsilon\cdot|\mathcal{A}|\cdot\frac{1}{1-\gamma}\right)\\
 & \geq\max_{m'\in\Mdet\setminus\BRdet^{\star}}\left(\langle r_{2}^{x^{\star}},m'\rangle+\epsilon\cdot|\mathcal{A}|\max_{(s,a):s\in\mathcal{S}_{d}}|m'(s,a)|\right)\\
 & \geq\max_{m'\in\Mdet\setminus\BRdet^{\star}}\left(\langle r_{2}^{x^{\star}},m'\rangle+\epsilon\langle\delta^{v},m'\rangle\right)\\
 & =\max_{m'\in\Mdet\setminus\BRdet^{\star}}\langle r_{2}^{x'},m'\rangle. \numberthis \label{eq: inequality of x'}
\end{align*}
It follows from (\ref{eq: equality of x'}) and (\ref{eq: inequality of x'}) that $\iprod{r_{2}^{x'}}{\mopt}\geq\iprod{r_{2}^{x'}}{m}$ for any $m\in\Mdet$ when $\epsilon\leq\frac{1-\gamma}{2\vert\mathcal{A}\vert}\cdot b$. Hence $\mopt\in\BR(x')$ or, equivalently, $x'\in\cP_{\mopt}$. Because $x'$ is chosen arbitrarily from $\{x^{\star}+\epsilon v\mid\Vert v\Vert_{1}<1\}$, it follows that $\{x^{\star}+\epsilon v\mid\Vert v\Vert_{1}<1\}\subset\mathcal{P}_{m^{\star}}$ when $\epsilon\leq\frac{1-\gamma}{2\vert\mathcal{A}\vert}\cdot b$ and $x^{\star}\in\mathcal{X}$, implying that $\xopt$ is an optimal interior-point allocation. 

Case 2: If (\ref{eq:identical_m}) does not hold, there must exist $s_{0}\in\mathcal{S}_{d}$ such that $\sum_{a\in\mathcal{A}}m(s_{0},a)\neq\sum_{a\in\mathcal{A}}m'(s_{0},a)$ for some $m,m'\in\BRdet^{\star}$, Because $\BRdet^{\star}$ is finite, there must exist $m_{2}\in\BRdet^{\star}$ such that $\sum_{a\in\mathcal{A}}m_{2}(s_{0},a)=\min_{m\in\BRdet^{\star}}\sum_{a\in\mathcal{A}}m(s_{0},a)$. Let $c>0$, and define $x'$ such that $x'(s_{0})=x^{\star}(s_{0})-c$ and $x'(s)=x^{\star}(s)$ for all $s\in\mathcal{S}_{d}\setminus\{s_{0}\}$. Then, for any $m\in\BRdet^{\star}$, it holds that
\begin{equation}
\begin{aligned}\langle r_{2}^{x'},m_{2}\rangle & =\langle r_{2}^{x^{\star}},m_{2}\rangle-c\sum_{a\in\mathcal{A}}m_{2}(s_{0},a)\\
 & =\langle r_{2}^{x^{\star}},m\rangle-c\sum_{a\in\mathcal{A}}m_{2}(s_{0},a)\\
 & \geq\langle r_{2}^{x^{\star}},m\rangle-c\sum_{a\in\mathcal{A}}m(s_{0},a)\\
 & =\langle r_{2}^{x'},m\rangle.
\end{aligned}
\label{eq:perturbed x and m in M_det}
\end{equation}
Meanwhile, choose $c<b(1-\gamma)$. This implies that $c\sum_{a\in\mathcal{A}}m_{2}(s_{0},a)\leq c/(1-\gamma)<b$. Thus,
\begin{equation}
\begin{aligned}\langle r_{2}^{x'},m_{2}\rangle & =\langle r_{2}^{x^{\star}},m_{2}\rangle-c\sum_{a\in\mathcal{A}}m_{2}(s_{0},a)\\
 & >\langle r_{2}^{x^{\star}},m_{2}\rangle-b\\
 & =\langle r_{2}^{x^{\star}},m^{\star}\rangle-b\\
 & =\max_{m'\in\Mdet\setminus\BRdet^{\star}}\langle r_{2}^{x^{\star}},m'\rangle\\
 & \geq\max_{m'\in\Mdet\setminus\BRdet^{\star}}\left(\langle r_{2}^{x^{\star}},m'\rangle-c\sum_{a\in\mathcal{A}}m'(s_{0},a)\right)\\
 & =\max_{m'\in\Mdet\setminus\BRdet^{\star}}\langle r_{2}^{x'},m'\rangle.
\end{aligned}
\label{eq:perturbed x for m in M=00005CM_det}
\end{equation}

It follows from (\ref{eq:perturbed x and m in M_det}) and (\ref{eq:perturbed x for m in M=00005CM_det}) that $\langle r_{2}^{x'},m_{2}\rangle\geq\max_{m\in\Mdet}\langle r_{2}^{x'},m\rangle$. Because any MDP has a deterministic optimal policy, it holds that $\max_{m\in\Mdet}\langle r_{2}^{x'},m\rangle=\max_{m\in\mathcal{M}}\langle r_{2}^{x'},m\rangle$. This implies $\langle r_{2}^{x'},m_{2}\rangle\geq\max_{m\in\mathcal{M}}\langle r_{2}^{x'},m\rangle$ or, equivalently, $m_{2}\in\BR(x')$. Since $m_{2}\in\BRdet^{\star}=\Mdet\cap\BR(\xopt)$, and any best response $m\in\BR(x^{\star})$ satisfies $\iprod{r_{1}}{m}=v_{1}^{\star}$, it follows that $\langle r_{1},m_{2}\rangle=v_{1}^{\star}$. This implies that $(x',m_{2})$ is a feasible solution of (\ref{eq:margin problem}). Since $C-\sum_{i=1}^{|\Sd|}x'_{i}=C-\sum_{i=1}^{|\Sd|}x_{i}^{\star}+c\geq c>0$, the optimal value of (\ref{eq:margin problem}) must be strictly positive. It then follows from Theorem~\ref{Thm:Sufficiency} that there exists an optimal interior-point allocation.
\end{proof}

\section{Proof of Proposition~\ref{thm:deterministic region contains non-1 }\label{sec:appendix-reward}}
\begin{lem}
\label{lem:optimal_action-2}Let $m^{\star}$ be an optimal occupancy measure of the reward design problem in~(\ref{eq:allocation_problem_om}), $\mr^{\star}\in\tilde{\mathcal{P}}_{m^{\star}}$, and $\pi^{\star}$ be a policy that induces $m^{\star}$, i.e., $\pistar\in\Pi(\mopt)$. For any $\pi\in\Pidet(\pi^{\star})$, the induced occupancy measure $m^{\pi}$ is a best response of $x^{\star}$, i.e. $m^{\pi}\in\BR(x^{\star})$ or, equivalently, $\mr^{\star}\in\tilde{\mathcal{P}}_{m^{\pi}}$. 
\end{lem}

\begin{proof}
Similar to Proposition~\ref{lem:optimal_action}.
\end{proof}
\begin{lem}
\label{lem: optimal_action-3}Under the same conditions in Lemma~\ref{lem:optimal_action-2}, the induced occupancy measure $m^{\pi}$ also satisfies $\iprod{r_{1}}{m^{\pi}}=v_{1}^{\star}$ . 
\end{lem}

\begin{proof}
Similar to Proposition~\ref{prop:all determinisitc m are optimal}.
\end{proof}
\begin{proof}[Proof of Proposition~\ref{thm:deterministic region contains non-1 }]
The proof is based on Lemma~\ref{lem:optimal_action-2} and Lemma~\ref{lem: optimal_action-3} and is similar to the proof of Theorem~\ref{thm:deterministic region contains non}.
\end{proof}